\documentclass[11pt]{amsart}

\usepackage{amsmath, amscd,hyperref}
\usepackage{amsfonts}
\usepackage{amsthm, upref}
\usepackage{graphicx}
\usepackage{enumerate}
\usepackage[dvipsnames]{xcolor}
\usepackage{tikz}
\usepackage{pgf}
\usetikzlibrary{matrix,arrows, decorations.markings}

\input xy
\xyoption{all}

    \numberwithin{equation}{section}

\usepackage{bm}
        
\bmdefine\alphab{\mathbf{\alpha}}
\bmdefine\betab{\mathbf{\beta}}
\bmdefine\pib{\mathbf{\pi}}
\bmdefine\xib{\mathbf{\xi}}
\bmdefine\sigmab{\mathbf{\sigma}}

\newcommand{\comment}[1]{}
\newcommand{\eq}{\begin{equation}}
\newcommand{\en}{\end{equation}}

\newcommand{\rr}{\mathbb{R}}
\newcommand{\pp}{\mathbb{P}}
\newcommand{\nn}{\mathbb{N}}

\newcommand{\abs}[1]{\left\lvert #1 \right\rvert}

\newcommand{\mcal}[1]{\mathcal{#1}}
\newcommand{\iprod}[1]{\left\langle #1 \right\rangle }

\newcommand{\ev}{\mathbb E}

\newcommand{\lamhat}{\widehat\Lambda}
\newcommand{\lhat}{\widehat{L}}

\newcommand{\ep}{\hfill $\Box$}

\newcommand{\intw}[1]{\left\langle #1 \right\rangle}
\newcommand{\tbf}[1]{\textbf{#1}}
\newcommand{\nin}{\noindent}
\newcommand{\bayL}{\widehat{L}}
\newcommand{\phat}{\widehat{P}}
\newcommand{\qhat}{\widehat{Q}}

\newcommand{\rhat}{\widehat{R}}

\newcommand{\mX}{\mathcal{X}}
\newcommand{\mY}{\mathcal{Y}}

\newcommand{\gens}{\mathcal{A}^S}

\DeclareMathOperator{\sgn}{sgn}

\begin{document}

\theoremstyle{plain}
\newtheorem{thm}{Theorem}
\newtheorem{rem}{Remark}
\newtheorem{lemma}[thm]{Lemma}
\newtheorem{prop}[thm]{Proposition}
\newtheorem{cor}[thm]{Corollary}

\theoremstyle{definition}
\newtheorem{defn}{Definition}
\newtheorem{cond}{Condition}
\newtheorem{asmp}{Assumption}
\newtheorem{notn}{Notation}
\newtheorem{prb}{Problem}

\theoremstyle{remark}
\newtheorem{rmk}{Remark}
\newtheorem{exm}{Example}
\newtheorem{clm}{Claim}

\title[Intertwining diffusions]{Intertwining diffusions and wave equations}

\author{Benjamin Budway, Soumik Pal, and Mykhaylo Shkolnikov}
\address{ORFE Department \\ Princeton University \\ Princeton, NJ 08544}
\email{bb2584@princeton.edu}
\address{Department of Mathematics\\ University of Washington\\ Seattle, WA 98195}
\email{soumikpal@gmail.com}
\address{Department of Mathematical Sciences and Center for Nonlinear Analysis\\Carnegie Mellon University \\  Pittsburgh, PA 15232}
\email{mshkolni@gmail.com}

\keywords{Diffusion processes, duality, growth models, hyperbolic PDEs, intertwining, time-reversal, transmutation, wave equations}  

\subjclass[2010]{60J60, 35L10, 35L20, 60B10}

\thanks{Soumik's research is partially supported by NSF grants DMS-2052239, DMS-2134012, DMS-2133244, and
PIMS PRN-01 granted to the Kantorovich Initiative. Mykhaylo's research is partially supported by NSF grant DMS-2108680.}

\date{\today}

\begin{abstract}
We develop a general theory of intertwined diffusion processes of any dimension. Our main result gives an SDE construction of intertwinings of diffusion processes and shows that they correspond to nonnegative solutions of hyperbolic partial differential equations. For example, solutions of the classical wave equation correspond to the intertwinings of two Brownian motions. The theory allows us to unify many older examples of intertwinings, such as the process extension of the beta-gamma algebra, with more recent examples such as the ones arising in the study of two-dimensional growth models. We also find many new classes of intertwinings and develop systematic procedures for building more complex intertwinings by combining simpler ones. In particular, `orthogonal waves' combine unidimensional intertwinings to produce multidimensional ones. Connections with duality, time reversals, and Doob's h-transforms are also explored. 
\end{abstract}

\maketitle

\section{Introduction}
\label{sec1}

We start with the definition of intertwining of two Markov semigroups that is reminiscent of a similarity transform of two finite-dimensional matrices. 

\begin{defn}\label{def:intertwin}
Let $\left(Q_t,\; t\ge 0\right)$, $\left(P_t,\; t\ge 0 \right)$ be two Markov semigroups on measurable spaces $\left(\mcal{E}_1,\mcal{B}_1\right)$, $\left(\mcal{E}_2,\mcal{B}_2 \right)$, respectively. Suppose $L$ is a stochastic transition operator that maps bounded measurable functions on $\mcal{E}_2$ to those on $\mcal{E}_1$. We say that the ordered pair $(Q,P)$ is intertwined with link $L$ if for all $t\geq0$ the relation $Q_t\,L=L\,P_t$ holds {(where both sides are viewed as operators acting on bounded measurable functions on $\mcal{E}_2$)}. If this is the case, we write $Q\intw{L}P$.  
\end{defn}

It is clear that intertwinings are special constructions which transfer a lot of spectral information from one semigroup to the other. Naturally one is interested in two kinds of broad questions: (a) Given two semigroups can we determine if they are intertwined via some link? (b) Can we find a coupling of two Markov processes, with transition semigroups $(Q_t)$ and $(P_t)$, respectively, such that the coupling construction naturally reflects the intertwining relationship? One should also ask what influence the analytic definition of intertwining has on the path properties of this coupling.

Question (a) is known to have an affirmative answer when the transition probabilities of a Markov process have symmetries. One can then intertwine this process with another process running on the quotient space. Other criteria were given based on the explicit knowledge of eigenvalues of the semigroup. Neither symmetries nor eigenvalues are generally available, and, hence, the answer to question (a) for general Markov processes is unknown. In the next subsection we outline briefly the development in this area over the last few decades.  

On the other hand, {Diaconis} and {Fill} \cite{DF} initiated a program of constructing couplings of two Markov chains whose semigroups $(Q_t)$ and $(P_t)$ satisfy $Q \iprod{L} P$. Such couplings lead to remarkable objects called strong stationary times which can be then used to determine the convergence rate of the Markov chain with transition semigroup $(P_t)$. 

\begin{figure}[t]
\centerline{
\xymatrix@=8em{
Z_2(s) \ar[d]_{L} \ar[r]^{ Q_t } & Z_2(s+t) \ar[d]^{L} \\
Z_1(s) \ar[r]^{P_t} & Z_1(s+t)
}
}
\caption{Commutative diagram of intertwining.}
\label{fig:commintw}
\end{figure}

Our main result settles both questions (a) and (b) when the semigroups are diffusion semigroups and we insist on the coupling to be a joint diffusion satisfying some natural conditional independence properties. We provide a general theory of intertwinings in the setting of diffusion processes allowing also for (possibly oblique) reflection at the boundary of their domains and on each other. 
This allows us to reprove many intertwining relations known so far, as well as to produce several large classes of new examples. The coupling that we propose can be thought of as a continuous time limit of the Diaconis-Fill construction. In this setting, the construction displays several remarkable properties, including stability under dimension reduction and time-reversals. Interestingly, it turns out that in this setup the link kernels are solutions to hyperbolic partial differential equations, such as the classical wave equation in the case of  intertwinings of two Brownian motions (see Theorems \ref{main1} and \ref{main2} below for the details). 

Throughout the paper we consider diffusion semigroups on finite-dimensional Euclidean spaces. Here, by a diffusion semigroup we mean a semigroup generated by a second order elliptic partial differential operator with no zero-order terms and either no boundary conditions or (possibly oblique) Neumann boundary conditions. Before we describe our coupling construction we recall a key concept in the {Diaconis-Fill} construction, namely the commutative diagram in Figure \ref{fig:commintw}, which we have extended to the continuous time setting. 

We consider two Markov processes in continuous time, $Z_1$ and $Z_2$, with transition semigroups $(P_t)$ and $(Q_t)$, respectively. The direction of arrows represents the action on measures (as opposed to that on functions). The diagram captures the following equivalence of sampling schemes: starting from $Z_2(s)$ it is possible to generate a sample of $Z_1(s+t)$ in two equivalent ways. Either sample $Z_2(s+t)$, conditionally on $Z_2(s)$ and then sample $Z_1(s+t)$ according to $L$. Or, sample $Z_1(s)$, conditionally on $Z_2(s)$, via $L$, and follow $Z_1$ to time $(s+t)$. It is a part of the construction that both $\left( Z_2(s), Z_2(s+t), Z_1(s+t)  \right)$ and $\left( Z_2(s), Z_1(s), Z_1(s+t) \right)$ are three step Markov chains. This insistence produces a coupling with nice path properties that can be further exploited. 
 
The above discussion motivates the following definition of a coupling realization of $Q\intw{L}P$ in terms of random processes. Let $\left(X(t),\; t\ge 0  \right)$ and $\left( Y(t),\; t\ge 0 \right)$ represent two time-homogeneous diffusions with locally compact state spaces ${\mathcal X}\subset\rr^m$, ${\mathcal Y}\subset\rr^n$ and transition semigroups $\left(P_t,\; t\geq0\right)$, $\left( Q_t,\; t\geq0\right)$, respectively. We abuse the notation slightly. Although, $X$ and $Y$ are diffusions, their laws are unspecified because we do not specify their initial distributions. They are merely processes with the correct transition semigroup. We also suppose that $L$ is a probability transition operator. 

\begin{defn}\label{idef}
We call a ${\mathcal X}\times{\mathcal Y}$-valued diffusion process $Z=(Z_1,Z_2)$ an intertwining of the diffusions $X$ and $Y$ with link $L$ (we write $Z=Y \intw{L} X$) if the following hold.
\begin{enumerate}[(i)]
\item \label{i1} $Z_1\stackrel{d}{=}X$ and $Z_2\stackrel{d}{=}Y$ where $\stackrel{d}{=}$ refers to identity in law, and 
\[
\ev\left[  f\left( Z_1(0)\right) \mid Z_2(0)=y\right]= (Lf)(y),
\]
for all bounded Borel measurable function $f$ on $\mcal{X}$. 
\item \label{i2} The transition semigroups are intertwined: $Q\intw{L}P$.
\item \label{i3} The process $Z_1$ is Markovian with respect to the joint filtration generated by $(Z_1, Z_2)$. 
\item \label{i4} For any $t\geq0$, conditional on $Z_2(t)$, the random variable $Z_1(t)$ is independent of $\left(Z_2(s),\; 0\le s \le t\right)$, and is conditionally distributed according to $L$.

\end{enumerate} 
\end{defn}



\medskip

Our primary results Theorem \ref{main1} and Theorem \ref{main2} answer the questions (b) and (a), respectively, raised at the beginning of the introduction. Given a locally compact $A$ in $\mathbb{R}^d$, it can be written as $A = O \cap \overline{A}$ where $O$ is an open subset of $\mathbb{R}^d$ and $\overline{E}$ denotes the closure of a set $E$ (see \cite[Theorem 18.4]{WI}). When we say that a function is continuous (resp. $C^m$) on $A$, we mean that it is the restriction of a continuous (resp. $C^m$) function on $O$ to $A$. Suppose we are given the two generators
\begin{eqnarray}
&&{\mathcal A}^X = \sum_{i=1}^m b_i(x)\partial_{x_i} + \frac{1}{2}\sum_{i,j=1}^m a_{ij}(x)\partial_{x_i}\partial_{x_j}\quad \text{and}
\label{Xgen}\\
&&{\mathcal A}^Y = \sum_{k=1}^n \gamma_k(y)\partial_{y_k} + \frac{1}{2}\sum_{k,l=1}^n \rho_{kl}(y)\partial_{y_k}\partial_{y_l},
\label{Ygen}
\end{eqnarray}
where $(b_i)_{i=1}^m$ is an $\rr^m$-valued function {continuous} on $\mcal{X}$, $(\gamma_k)_{k=1}^n$ is an $\rr^n$-valued function {continuous} on $\mcal{Y}$, $(a_{ij})_{1\leq i,j\leq m}$ and $(\rho_{kl})_{1\leq k,l\leq n}$ are  functions taking values in the set of positive semidefinite $m\times m$ and $n\times n$ matrices {continuous} on $\mcal{X}$ and $\mcal{Y}$, respectively. We make the following assumption.

\begin{asmp}\label{main_asmp}
Assume that each $X$ and $Y$ satisfy either one of the following two conditions.

\comment{satisfy either one of the following two conditions.}
\begin{enumerate}[(a)]
\item \textbf{No boundary conditions.} The domain $\mcal{X}$ (resp. $\mcal{Y}$) is open, and the SDE on $\mcal{X}$ with $\mcal{A}^X$  as its generator is well-posed and never reaches the boundary. Moreover, the solution $X$  is a Feller-Markov process. That is, its semigroup preserves the space $C_0(\mcal{X})$ of continuous functions vanishing at infinity. For $Y$ replace $\mcal{A}^X$ by $\mcal{A}^Y$, $\mX$ by $\mY$, and so on.   We also assume that $C_c^\infty(\mcal{X})$ (resp. $C_c^\infty(\mcal{Y})$) is a core (see \cite[page 374]{Ka}) of the domain of $\mcal{A}^X$ (resp. $\mcal{A}^Y$). \color{black}
\item \textbf{Neumann boundary conditions.} The domain $\mcal{X}$ is closed with $C^2$ boundary. Moreover, for some $C^2$ vector field $U_1:\,\partial\mcal{X}\rightarrow\rr^m$ whose scalar product with the unit inward normal vector field is uniformly positive on $\partial\mcal{X}$, the stochastic differential equation with reflection corresponding to $\mcal{A}^X$ with Neumann boundary conditions with respect to $U_1$ is well-posed in the sense of \cite{RK}. In addition, the solution $X$ is a Feller-Markov process. That is, its semigroup preserves the space $C_0(\mcal{X})$ of continuous functions vanishing at infinity. {Finally, the generator $\mcal{A}^X$ is regular in the sense that the intersection of the space $C_c^\infty(\mcal{X})$ of infinitely differentiable functions on $\mcal{X}$ with compact support with the domain of ${\mathcal A}^X$ in $C_0(\mcal{X})$ is dense in that domain with respect to the uniform norm on $C_0(\mcal{X})$.} For $Y$ replace $\partial \mX$ by $\partial \mY$, $U_1$ by $U_2$, and so on.\color{black}
\end{enumerate}
\end{asmp} 

\begin{asmp}\label{asmpthm1} 
We consider the following regularity conditions on the kernel $L$. 
\begin{enumerate}[(i)]
\item Suppose that $L$ is given by an integral operator 
\[
(Lf)(y) = \int_{\mathcal X} \Lambda(y,x)\,f(x)\,\mathrm{d}x
\]
{mapping $C_0(\mcal{X})$ into $C_0(\mcal{Y})$}.
\item Assume $\Lambda(\cdot,x)$ is {strictly positive and continuously differentiable} on $\mcal{Y}$ for every fixed $x$ in $\mcal{X}$. Set $V=\log \Lambda$ and let $\nabla_y V$ denote the gradient of $V$ with respect to $y$. 

\color{black}
\item $\Lambda(\cdot,x)$ is in the domain of ${\mathcal A}^Y$ for all $x\in{\mathcal X}$ {with ${\mathcal A}^Y\Lambda$ being continuous on $\mcal{Y}\times\mcal{X}$ and bounded on $\mcal{Y}\times K$ for any compact $K\subset\mcal{X}$}.
\item For all $y\in\mcal{Y}$, $\Lambda(y,\cdot)$ belongs to the domain of $\left({\mathcal A}^X\right)^*$, the adjoint of $\mcal{A}^X$ acting on measures (see, e.g., \cite[Definition B.8]{EN}).
\color{black}
\end{enumerate}
\end{asmp}

\comment{
}
As mentioned in the introduction, the intertwinings we will construct should be thought of as the natural continuous time extension of the construction performed in \cite{DF}. If one assumes that a Markov process $Z$ is an intertwining as in Definition \ref{idef} and additionally assumes that $Z_2(t)$ is conditionally independent of $Z_1(0)$ given $(Z_1(t),Z_2(0))$, then one can explicitly write down the transition kernel of $Z$ using Bayes' rule as
\begin{equation}\label{approxkernel}
\tilde{R}_t((x_0,y_0),\mathrm{d}(x_1,y_1)) = \frac{Q_t(y_0,\mathrm{d}y_1)P_t(x_0,\mathrm{d}x_1)\Lambda(y_1,x_1)}{\int_{\mcal{Y}} Q_t(y_0, \mathrm{d}y)\Lambda(y,x_1)}.
\end{equation}
This formula is nearly identical to the transition matrix proposed in \cite{DF}. However, as pointed out in \cite{JF}, this formula cannot be used to construct intertwinings in continuous time due to the fact that ($\tilde{R}_t$) does not necessarily satisfy the Chapman-Kolmogorov equations. Instead of studying a non-Markovian process satisfying this conditional independence property, we consider the following ``infinitesimal" conditional independence condition.

\smallskip

\begin{enumerate}[]
\item  A Feller-Markov process $Z$ is said to satisfy the infinitesimal Bayes' condition if for any function $h\in C_c^\infty(\mcal{X}\times\mcal{Y})\cap \mcal{D}(\mcal{A}^Z)$, in the regime as $t\downarrow 0$, the conditional expectation $\ev[h(Z(t))\!\mid\!Z(0)=(x_0,y_0)]$ is equal to    
\eq\label{eq:condlaw}
\int_{{\mathcal X}\times{\mathcal Y}} h(x_1,y_1)\,\tilde{R}_t((x_0,y_0),\mathrm{d}(x_1,y_1)) + o(t). 
\en
Here, the error term $o(t)$ is allowed to depend on $h$ as well as on $(x_0,y_0)$.
\end{enumerate}

\medskip

We now present our main theorems. Denote the transpose of a vector $x$ by $x'$. Suppose Assumptions \ref{main_asmp} and \ref{asmpthm1} are satisfied. Consider $z\in \rr^{m+n}$ as $z=(x,y)$ where $x\in \rr^m$ and $y \in \rr^n$. 

\begin{thm}\label{main1} Let $X$, $Y$ be the (reflected) diffusions given by the solutions of the above martingale (resp. submartingale) problems. Let $Z=(Z_1,Z_2)$ be a diffusion process on ${\mathcal X}\times{\mathcal Y}$ with generator
\eq\label{Zgen}
\begin{split}
{\mathcal A}^Z &= \mcal{A}^X + \mcal{A}^Y + \big(\nabla_y V(y,x)\big)'\,\rho(y)\,\nabla_y
\end{split}
\en
and boundary conditions on $\partial{\mathcal X}\times{\mathcal Y}$ (resp. ${\mathcal X}\times\partial{\mathcal Y}$) coinciding with those of $X$ on $\partial{\mathcal X}$ (resp. $Y$ on $\partial{\mathcal Y}$). Suppose that $C_c^{\infty}(\mcal{X}\times\mcal{Y}) \cap \mcal{D}(\mcal{A}^Z)$ is a core for $\mcal{D}(\mcal{A}^Z)$. Moreover, let the initial condition of the diffusion $Z$ satisfy
\[
P\left( Z_1(0) \in B \mid Z_2(0)=y \right)= \int_B \Lambda(y,x)\,\mathrm{d}x, \quad \text{for all Borel $B\subseteq \rr^m$}.
\]
If $\Lambda$ is such that the density of the measure $\left({\mathcal A}^X\right)^*\Lambda(y,\cdot)$ is given by $({\mathcal A}^Y\Lambda)(y,\cdot)$, in short:


\eq\label{PDE}
\left({\mathcal A}^X\right)^*\,\Lambda = {\mathcal A}^Y\,\Lambda \quad \text{on}\quad {\mathcal X}\times{\mathcal Y}\,,
\en
then $Z=Y\intw{L} X$ and $Z$ satisfies the infinitesimal Bayes' condition \eqref{eq:condlaw}. 
\end{thm}

As a quick example, consider the Cauchy density kernel 
\[
\Lambda(y,x)=\frac{1}{ \pi \left( 1+ (y-x)^2\right)}.
\]
It satisfies the one-dimensional wave equation. Consider the diffusion given by
\[
\mathrm{d}Z_1(t)=\mathrm{d}\beta_1(t),\quad
\mathrm{d}Z_2(t)=\mathrm{d}\beta_2(t) - \left(\frac{2 \left( Z_2(t) - Z_1(t) \right)}{1 + \left( Z_2(t) - Z_1(t) \right)^2}\right)\,\mathrm{d}t,
\]
where $\beta_1$, $\beta_2$ are two independent one-dimensional standard Brownian motions. Then, by Theorem \ref{main1}, for appropriate initial conditions the marginal law of $Z_2$ is that of a standard Brownian motion and the conditional law of $Z_1(t)$ given $Z_2(t)$ is Cauchy for every $t\geq0$. 
\medskip


\comment{The construction in Theorem \ref{main1} is a continuum analogue of the construction by Diaconis and Fill \cite{DF}. In their construction one step transition probabilities of the intertwined pair of Markov chains are given by an application of the Bayes' rule (see \cite[eqn. (2.21)]{DF}).} Our next theorem shows, under regularity conditions, that the infinitesimal Bayes' condition forces the generator of the intertwined diffusion to be given by \eqref{Zgen}. Let the generators ${\mathcal A}^X$, ${\mathcal A}^Y$ of \eqref{Xgen}, \eqref{Ygen} satisfy Assumption \ref{main_asmp} and $X$, $Y$ be the corresponding diffusion processes. Suppose there is a {Feller-Markov process} $Z$ satisfying conditions \eqref{i1}, \eqref{i2} in Definition \ref{idef} along with the infinitesimal Bayes' condition \eqref{eq:condlaw}.


\begin{thm}\label{main2} Suppose that the kernel $L$ satisfies Assumption \ref{asmpthm1}. Then the action of the  generator of $Z$ on $C_c^{\infty}(\mathcal{X}\times\mcal{Y})$ \color{black} is given by \eqref{Zgen} with the boundary conditions as in Theorem \ref{main1}, and $\Lambda$ satisfies \eqref{PDE}.  
 Moreover, for every function $f\in \mcal{D}\left({\mathcal A}^X\right)$, 
the commutativity relation holds:
\eq\label{comm}  
L{\mathcal A}^X f = {\mathcal A}^Y Lf.
\en
\end{thm}


\comment{\begin{rmk}\label{rmk:semigp}
We should warn the reader that, for our diffusions, the extended Diaconis-Fill condition \eqref{eq:condlaw} may not hold for every fixed time $t$ without the limit. One can construct examples where this violates the semigroup property of the joint process. 
\end{rmk}}

In the analytic literature the commutativity relation \eqref{comm} is usually referred to as \textit{transmutation of the operators} ${\mathcal A}^X$ and ${\mathcal A}^Y$. The latter is a classical concept in the study of partial differential equations and goes back to {Euler}, {Poisson} and {Darboux} in the case that ${\mathcal A}^X$ is the Laplacian and ${\mathcal A}^Y$ is its radial part (or, in other words, the generator of a Bessel process). An excellent introduction to this area is the book \cite{Ca2} by {Carroll} which, in particular, stresses the role that special functions play in the theory of transmutations.    

\medskip

The rest of the paper is structured as follows. 
\begin{enumerate}[(i)]
\item We end the introduction with the following subsection that reviews the literature that has led to the development of the subject so far. 
\item In Section \ref{sec:proofs} we give the proofs of Theorems \ref{main1} and \ref{main2}. We also prove a generalization to diffusions reflecting on moving boundaries and establish an important connection to harmonic functions and Doob's h-transforms.  
\item In Section 3 we explore the Markov chain of diffusions induced by intertwinings. We also explore the deep connection of intertwining with duality which demonstrates how the direction of intertwining reverses with time-reversal. We also construct \textit{simultaneous intertwining} that allows us to couple multiple duals with the same diffusion. 
\item Section \ref{sec:exmp} is in two parts. The first collects most known examples and shows that they are all covered by our results. This includes recent examples such as the $2$d-Whittaker growth model (related to the Hamiltonian of the quantum Toda lattice). In the second part, we produce classes of new examples by solving the corresponding hyperbolic partial differential equations. 
\item In Section \ref{sec:reflection} we cover diffusions reflected on a moving boundary. A major example is the Warren construction of interlacing Dyson Brownian motions on the Gelfand-Tsetlin cone for which we give two new proofs.
\comment{ 
\item  Section \ref{sec:rates} explains how intertwinings can be used to give bounds on the rate of convergence to equilibrium for diffusion semigroups. }
\item Finally, an appendix has been added on the literature on common hyperbolic PDEs for the benefit of a reader with a probability background. 
\end{enumerate}

\subsection{A brief review of the literature.} The study of intertwinings started with the question of when a function of a Markov process is again a Markov process. General criteria were given by {Dynkin} (see \cite{Dy}), {Kemeny} and {Snell} (see \cite{KS2}), and {Rosenblatt} (see \cite{Ro}). In \cite{RP}, {Rogers} and {Pitman} derived a new criterion of this type and used it to reprove the celebrated $2M-B$ Theorem of {Pitman} (see \cite{Pi} for the original result and \cite{JY} by Jeulin and Yor for yet another proof). These examples have been reviewed in detail in Section \ref{sec:exmp}.

Pitman's result triggered an extensive study of functionals of Brownian motion (and, more generally, of L\'{e}vy processes) through intertwining relations. Notable examples include the articles by {Matsumoto} and {Yor} (see \cite{MY1}, \cite{MY2}) which extend Pitman's Theorem to exponential functionals of Brownian motion by exploiting the fact that the latter are intertwined with the Brownian motion itself (see also Baudoin and O'Connell \cite{BO} for an extension to higher dimensions); the paper \cite{CPY} by Carmona, Petit, and Yor presents a new class of intertwining relations between Bessel processes of different dimensions, which can be viewed as the process extension of the well-known Beta-Gamma algebra; the article \cite{Du} by Dub\'edat shows that a certain reflected Brownian motion in a two-dimensional wedge is intertwined with a $3$-dimensional Bessel process and uses this fact to derive formulas for some hitting probabilities of the former; and the paper \cite{Y} extends the results in \cite{MY1}, \cite{MY2} further to exponential functionals of L\'{e}vy processes. 

More recently, interwining relations were discovered in the study of random matrices and related particle systems. In \cite{DM}, the authors Donati-Martin, Doumerc, Matsumoto, and Yor give a matrix version of the findings in \cite{CPY}, namely an intertwining relation between Wishart processes of different parameters. The works by Warren \cite{W}, Warren and Windridge \cite{WW}, O'Connell \cite{Oc}, Borodin and Corwin \cite{BC} and Gorin and Shkolnikov \cite{GS} exploit the idea that one can concatenate multiple finite-dimensional Markov processes, each viewed as a particle system on the real line given by its components, to a \textit{multilevel process} provided that any two consecutive levels obey an intertwining relation. This program was initiated by {Warren} in \cite{W} who construced a multilevel process in which the particle systems on the different levels are given by Dyson Brownian motions of varying dimensions with parameter $\beta=2$ (corresponding to the evolution of eigenvalues of a Hermitian Brownian motion). Related dynamics were studied in \cite{WW} and an extension to arbitrary positive $\beta$ is given in \cite{GS}. Such processes arise as diffusive limits of continuous time Markov chains defined in terms of symmetric polynomials (Schur polynomials in the case of $\beta=2$ and, more generally, Jack polynomials, see \cite{GS0}, \cite{GS} and the references therein). The articles \cite{BC}, \cite{Oc} explore (among other things) the multilevel diffusion processes corresponding to a class of Macdonald polynomials. The article \cite{AOW} studies intertwining relations among $h$-transforms of Markov processes whose transition densities have a determinantal structure and constructs multilevel couplings realizing these intertwinings.

In many situations, intertwining relations arise as the result of deep algebraic structures. {Biane} (see \cite{Bi}) gives a group theoretic construction that produces intertwinings based on Gelfand pairs. In Diaz and Weinberger \cite{DW} the construction of intertwinings is based on the determinantal ({Karlin}-{McGregor}) form of the transition semigroups involved. The paper by Gallardo and Yor \cite{GY} exploits the intertwining of Dunkl processes with Brownian motion and the link operator there is an algebraic isomorphism on the space of polynomials which preserves the subspaces of homogeneous polynomials of any fixed degree. Another example is the deep connection of the Robinson-Schensted correspondence with the intertwining relation between a Dyson Brownian motion and a standard Brownian motion of the same dimension established by {O'Connell} (see \cite{Oc2}). An example of intertwining given by an underlying branching structure appears in Johnson and Pal \cite{JP}.

Originally, intertwining relations have been used to derive explicit formulas for the more complicated of two intertwined processes from the simpler of the two processes (see the references above). However, there are other interesting applications of intertwinings. {Diaconis} and {Fill} \cite{DF} show that intertwinings of two Markov chains can be used to understand the convergence to equilibrium of one of the chains by understanding the hitting times of the other chain. This method relies on the fact that the latter hitting times are strong stationary times of the former Markov chain and, thus, give sharp control on its convergence to equilibrium in the separation distance as explained by Aldous and Diaconis \cite{AD}. Fill  \cite{JF} extended these ideas to the case of continuous-time Markov jump processes.  Another application of intertwinings lies in the construction of new Markov processes, typically ones with non-standard state spaces (such as a number of copies of $\rr_+$ glued together at $0$ in the case of Walsh's spider), from existing ones (see Barlow and Evans \cite{BE}, Evans and Sowers \cite{ES} for a collection of such constructions).

Yet another related concept comes from filtering theory. In the article \cite{Kur98} (see also \cite{KurOc}), Kurtz considers the martingale problem version of determining when a function of a Markov process is again Markov. The author develops the concept of a filtered martingale problem where one considers the martingale problem satisfied by the projection of the law of a Markov process onto a smaller filtration. It can be related to our problem at hand in the following way. Suppose we start with the coupling given in Theorem \ref{main1}. Take the Markov process to be $Z=(Z_1, Z_2)$ with its own associated filtration. Take the projection map $(z_1, z_2)\mapsto z_1$. If the regularity conditions in \cite{Kur98} are met, then the claim that $Z_1$ is Markov should follow from the approach in \cite{Kur98}. However, there is no systematic way to guess such couplings from the filtering approach. Moreover, the additional diagonal independence stipulated by condition (iv) of Definition \ref{idef} (or, the extended Diaconis-Fill condition (v) in \eqref{eq:condlaw}) does not follow from this general abstract approach. In particular, there are no counterparts to Theorem \ref{main2} and the results in Section \ref{sec:properties} in the filtering framework. On the other hand, filtered martingale problems can be applied to general Markov processes that are not diffusions and possibly admit jumps.  

In \cite{MP}, Miclo and Patie introduce a strengthening of intertwining relationships called \textit{interweaving}. A semigroup $(Q_s)$  is said to have an interweaving relation with another semigroup $(P_s)$ if there exist stochastic kernels $L$ and $\tilde{L}$ and a nonnegative random variable $\tau$ such that $Q \intw{L} P$, $P \langle \tilde{L} \rangle Q$, and 
\[
L \tilde{L} = \int_0^{\infty}Q_s \mathbb{P}(\tau \in \mathrm{d}s).
\]
When $(Q_s)$ has an interweaving relation with $(P_s)$, strong information about $(Q_s)$ (such as, e.g., convergence to equilibrium, hypercontractivity, and cut-off phenomena) can be deduced from that about $(P_s)$.

Two other interesting articles have considered strong stationary duality and intertwining of one-dimensional diffusions. Fill and Lyzinski \cite{FLyz} and Miclo \cite{Miclo13} are both primarily motivated by the question of rate of convergence of one-dimensional diffusions to equilibrium. These works are similar to ours in the sense that they are also extensions of the Diaconis-Fill construction to continuous time. In one dimension, these authors perform a much more detailed analysis of the dual using the scale function and the speed measure. Miclo, for example, extends the Morris-Peres idea of evolving sets to diffusions and constructs set-valued processes that intertwine the original semigroup. These ideas are extended in \cite{AKM} which constructs set-valued duals for Brownian motion on manifolds. This is different from our goal of characterizing the multidimensional intertwining coupling in terms of solutions of hyperbolic equations in its own right, and not just as a tool for the study of convergence rates.      

There is another notion of duality, originally due to Holley and Stroock \cite{HS}, which is prevalent in areas of probability such as interacting particle systems and population biology models. We refer to the book by Liggett \cite[Definition 2.3.1]{L} for numerous applications. This concept is sometimes called $h$-duality, a particular case of which is Siegmund duality \cite{Sieg}. Two Markov semigroups $(Q_t)$ and $(P_t)$ are dual with respect to a function $h:\mY \times \mX \rightarrow [0,\infty)$ if for every $(y,x)\in \mY \times \mX$ we have
\[
Q_t\left( h_x \right)(y)= P_t \left( h^y \right)(x), 
\] 
where $h_x(y)= h^y(x)=h(y,x)$. When $\mX=\mY=\rr$ and $h(y,x)=\sgn(y-x)$ this is called Siegmund duality. The notions of $h$-duality and intertwining are to some extent equivalent, in that the function $h$, suitably normalized, acts as an intertwining kernel between $Q$ and the time-reversal of $P$ under a Doob's $h$-transform. This has been shown in \cite[Proposition 5.1]{CPY} and in various results in \cite[Section 5.2]{DF}. Please consult these references for an exact statement. For more on the role of $h$-transforms in the context of intertwinings please see Section \ref{sec:proofs}.  

\subsection{Acknowledgement.} It is our pleasure to thank Alexei Borodin for pointing out the lack of a theory of intertwined diffusions to us and for many enlightening discussions. We also thank Alexei Borodin and Vadim Gorin for pointing out the asymptotic nature of the condition (v) preceding the statement of Theorem \ref{main2} above and S.~R.~S.~Varadhan for a very helpful discussion. We are grateful for helpful comments from Ioannis Karatzas and Sourav Chatterjee that led to an improvement of the presentation of the material from an earlier draft. Finally, we are indebted to the anonymous associate editor and referee for detecting a mistake in the original version of the paper. 

\section{Proofs of the main results, extensions, and generalizations} \label{sec:proofs}

\begin{notn} The following notations will be used throughout the text. For a subset $\mcal{X}$ of a Euclidean space, as before, $C_0\left(\mcal{X}\right)$ denotes the space of continuous functions on $\mcal{X}$ vanishing at infinity. In addition, we write $C^\infty_c\left(\mcal{X}\right)$ for the space of infinitely differentiable functions on $\mcal{X}$ with compact support.
\end{notn}

We start with the proof of Theorem \ref{main1}. 

\medskip

\noindent\textbf{Proof of Theorem \ref{main1}.} The proof is broken down into several steps. Throughout the proof we will assume that the underlying filtered probability space is given by the canonical space of continuous paths, $C\left( [0,\infty),\; {\mathcal X}\times{\mathcal Y} \right)$, from $[0, \infty)$ to ${\mathcal X}\times{\mathcal Y}$, along with the standard Borel $\sigma$-algebra and a probability measure $\pp$, the law of the process $Z$. This space is then equipped with the right-continuous filtration $\left\{  \mcal{F}_t,\; t\ge 0 \right\}$ generated by the coordinates and augmented with the null sets of $\mathbb{P}$. Let $\left(\pp_{z},\; z\in{\mathcal X}\times{\mathcal Y} \right)$ be the set of solutions of the martingale {(submartingale resp.)} problem for $\mcal{A}^Z$ starting at $z\in{\mathcal X}\times{\mathcal Y}$. The notation $\ev$ will refer to a generic expectation.

We will also need two sub-filtrations. Let $\left\{\mcal{F}^X_t, \; t\ge 0\right\}$ and $\left\{\mcal{F}^Y_t,\; t\ge 0 \right\}$ denote the right-continuous complete sub-filtrations of $\left\{\mcal{F}_t,\;t\ge 0\right\}$ generated the by the first $m$ and the next $n$ coordinate processes in $C\left( [0,\infty), {\mathcal X}\times{\mathcal Y} \right)$, respectively. 

\medskip
\comment{
\nin\tbf{Step 1.} We first claim that for all $f \in C_c(\mcal{X})$, $Lf \in \mcal{D}(\mcal{A}^Y)$ and
\begin{equation}\label{gen2}
\mcal{A}^Y Lf = \int \mcal{A}^Y \Lambda(y,x)f(x)\mathrm{d}x.
\end{equation}
The continuity of $\mcal{A}^Y\Lambda$, the compactness of the support of $f$, and the boundedness of $\mcal{A}^Y \Lambda$ on $\text{supp}(f)\times\mcal{Y}$ ensure the right-hand side is a $C_0(\mcal{Y})$ function. Without loss of generality, we may assume that $f \geq 0$. Then, Fubini's Theorem for nonnegative integrands ensures that 
\begin{equation}\label{gen1}
\frac{Q_tLf(y) - Lf(y)}{t} = \int \frac{[Q_t \Lambda(\cdot,x)](y) - \Lambda(y,x)}{t}f(x)\mathrm{d}x = \int \frac{1}{t}\int_0^t Q_s \mcal{A}^Y\Lambda(y,x)f(x)\mathrm{d}s\mathrm{d}x.
\end{equation}
Now, note that the continuity and boundedness of $\mcal{A}^Y\Lambda$ on $\text{supp}(f) \times \mcal{Y}$ imply that
\[
\Big|\frac{1}{t}\int_0^t Q_s\mcal{A}^Y\Lambda(y,x)\mathrm{d}s - \mcal{A}^Y\Lambda(y,x)\Big|=o(1)
\]\
where the $o(1)$ term is uniform over $x \in \text{supp}(f)$. Therefore, by the dominated convergence theorem, we have that the left-hand side of \ref{gen1} converges pointwise in $y$ to the right-handside of \ref{gen2}. \cite[Theorem 1.33]{LM} implies that $Lf \in \mcal{D}(\mcal{A}^Y)$ and that it is equal to the right-hand side of \ref{gen2}.}

\nin\tbf{Step 1.} We first prove that the process $Z_1$ is a {Feller-}Markov process with respect to its own filtration. It is easy to see that under any $\mathbb{P}_{(x,y)}$, $Z_1$ is a weak solution to the SDE with generator $\mcal{A}^X$ started from $x$. Since the SDE is well-posed,  we must have $Z_1\stackrel{d}{=}X$. In particular, $Z_1$ is a Feller-Markov process with respect to $\left\{\mcal{F}_t^X,\;t\ge0 \right\}$.

\comment{By applying It\^o's formula to functions of $Z_1$ it is easy to see that $Z_1$ solves the martingale (submartingale resp.) problem for ${\mathcal A}^X$. Since the initial distributions of $Z_1$ and $X$ match, we must have $Z_1\stackrel{d}{=}X$. In particular, $Z_1$ is a Feller-Markov process with respect to $\left\{\mcal{F}_t^X,\;t\ge0 \right\}$.}

\comment{
\medskip

}

\medskip
\noindent\textbf{Step 2.} Next, we show condition \eqref{i3} in Definition \ref{idef}. Fix any $0\le s < t<\infty$. We need to show that $Z_1(t)$, conditioned on $Z_1(s)$, is independent of the $\sigma$-algebra $\mcal{F}^Z_s$. Since $Z$ is assumed to be Markovian, it is enough to show that, given $Z_1(s)$, $Z_1(t)$ is independent of $Z_2(s)$. To this end, we observe that due to the time-homogeneity of the semigroup of $Z$ it is sufficient to consider $s=0$. Therefore, condition \eqref{i3} in Definition \ref{idef} holds if the following equality is true for all bounded measurable functions $f$ on $\mcal{X}$:
\eq\label{i3claim}
\ev\big[f(Z_1(t))\,\big|\,Z_1(0)=x,Z_2(0)=y\big]
=\ev\big[f(Z_1(t))\,\big|\,Z_1(0)=x\big],\quad (t,x,y)\in[0,\infty)\times{\mathcal X}\times{\mathcal Y}.
\en
To show this, it suffices to show that the law of $Z_1$ is the same under $\mathbb{P}_{(x,y)}$ and $\mathbb{P}_{(x,y')}$ for all $y,y' \in \mcal{Y}$. However, the law of $Z_1$ under both $\mathbb{P}_{(x,y)}$ and $\mathbb{P}_{(x,y')}$ is a weak solution to the SDE with generator $\mcal{A}^X$ started from $x$. Since the SDE was assumed to be well-posed, we must have that the law of $Z_1$ is identical under both probability measures.

\comment{However,  It\^o's  formula implies that the law of $Z_1$ under both $\mathbb{P}_{(x,y)}$ and $\mathbb{P}_{(x,y')}$ are solutions to the $C_c^{\infty}(\mcal{X})$ martingale (submartingale resp.) problem for $\mcal{A}^X$ starting at $x$. Since this problem was assumed to be well-posed, we must have that law of $Z_1$ is identical under both probability measures.}

\color{black}
\comment{In this setting {and in view of the Monotone Class Theorem, the Monotone Convergence Theorem, \color{blue} and the fact that $X$ spends Lebesgue almost no time on the boundary}, \color{black} condition \eqref{i3} in Definition \ref{idef} can be written as
\eq\label{i3claim}
\ev\big[f(Z_1(t))\,\big|\,Z_1(0)=x,Z_2(0)=y\big]
=\ev\big[f(Z_1(t))\,\big|\,Z_1(0)=x\big],\quad (t,x,y)\in[0,\infty)\times{\mathcal X}\times{\mathcal Y}
\en
for all functions $f\in C_c^\infty(\overset{\circ}{{\mathcal X}})$. The left-hand side of \eqref{i3claim} satisfies the Kolmogorov forward equation for ${\mathcal A}^Z$ in the variables $(x,y)$ (see, e.g., Theorem 17.6 in \cite{Ka}) with the initial condition $(x,y)\mapsto f(x)$. Thus, \eqref{i3claim} is a consequence of the uniqueness theorem for the latter (cf. Proposition II.6.2 in \cite{EN}) and an application of the following statement to the right-hand side of \eqref{i3claim}: any $C_0(\mathcal{X})$ function $u$ in the domain of ${\mathcal A}^X$, viewed as a function on $\mathcal{X}\times\mathcal{Y}$, satisfies ${\mathcal A}^Z u={\mathcal A}^X u$. To see this it suffices to approximate $u$ uniformly by a sequence of $C_c^\infty(\mathcal{X})$ functions $u_l$, $l\in\nn$ such that ${\mathcal A}^X u_l \to {\mathcal A}^X u$ in $C_0(\mcal{X})$ and to pass to the limit $l\to\infty$ in ${\mathcal A}^Z u_l={\mathcal A}^X u_l$ using the closedness of ${\mathcal A}^Z$.}
\medskip

\comment{
}
\nin\tbf{Step 3.}  We now claim the following. 
\medskip

\nin\tbf{Claim:} Take any $h\in \mcal{D}(\mcal{A}^Z)$. Then the function
\eq\label{whatisu}
{u(t):\;{\mathcal Y}\to\rr,\quad y\mapsto\ev \left [h(Z_1(t),Z_2(t))\mid Z_2(0)=y\right]} 
\en
is in the domain of ${\mathcal A}^Y$ in $C_0\left(\mcal{Y}\right)$ for every $t\ge0$, the function $t\mapsto u(t)$ is continuously differentiable with respect to the uniform norm on $C_0\left(\mcal{Y}\right)$, and 
\eq\label{MP}
\frac{\mathrm{d}}{\mathrm{d}t}\,u(t) = {\mathcal A}^Y u(t),\quad t\ge0.
\en

\medskip

To prove the claim we define, {for every fixed $t\ge0$, the function}
\eq\label{eq:vtxy}
{v(t):\;\mcal{X}\times\mcal{Y}\to\rr,\quad (x,y)\mapsto\ev\left [h(Z_1(t),Z_2(t)) \mid  Z_1(0)=x,Z_2(0)=y\right]}.
\en
Thanks to the assumption on the conditional distribution of $Z_1(0)$ given $Z_2(0)$ the expectation in \eqref{whatisu} can be rewritten as
\eq\label{whatisv}
\int_{\mathcal X} \Lambda(y,x)\,v(t)(x,y)\,\mathrm{d}x\,. 
\en
Moreover, by \cite[Theorem 17.6]{Ka}, $v(t)$ belongs to the domain of $\mcal{A}^Z$ {in $C_0\left(\mcal{X}\times\mcal{Y}\right)$} for every $t\ge0$\color{black} \comment{(here the superscript $'$ denotes the transpose)}, the function $t\mapsto v(t)$ is continuously differentiable with respect to the uniform norm on $C_0\left(\mcal{X}\times\mcal{Y}\right)$, and one has the Kolmogorov forward equation
\begin{equation}
\frac{\mathrm{d}}{\mathrm{d}t}\,v(t)=\comment{\big({\mathcal A}^X + {\mathcal A}^Y + (\nabla_y\,V)'\,\rho\,\nabla_y\big)\,}\mcal{A}^Z \, v(t),\quad t\ge0. 
\end{equation}

Since the derivative $\frac{\mathrm{d}}{\mathrm{d}t}\,v(t)$ was defined with respect to the uniform norm on $C_0\left(\mcal{X}\times\mcal{Y}\right)$, by the Feller-Markov property we have  
\eq\label{eq:somemanip}
\frac{\mathrm{d}}{\mathrm{d}t}\,u(t)=\int_{\mathcal X} \Lambda\,\frac{\mathrm{d}}{\mathrm{d}t}\,v(t)\,\mathrm{d}x
=\int_{\mathcal X} \Lambda\,\comment{\big({\mathcal A}^X + {\mathcal A}^Y + (\nabla_y\,V)'\,\rho\,\nabla_y\big)\,}\mcal{A}^Z \,v(t)
\,\mathrm{d}x.
\en
{Moreover, we note that the operator \comment{${\mathcal A}^X + {\mathcal A}^Y + (\nabla_y\,V)'\,\rho\,\nabla_y$} $\mcal{A}^Z$ is closed as an operator on $C_0\left(\mcal{X}\times\mcal{Y}\right)$ by  \cite[Lemma 17.8]{Ka}. \comment{In addition, due to the continuity of its coefficients, the subspace $C_c^\infty\left(\mcal{X}\times\mcal{Y}\right)$ is a dense subset of its domain. In other words, $C_c^\infty\left(\mcal{X}\times\mcal{Y}\right)$ is a core in the sense of Chapter 17 in \cite{Ka}.} By assumption, $C_c^\infty\left(\mcal{X}\times\mcal{Y}\right)\cap \mcal{D}(\mcal{A}^Z)$ is a core for the domain of $\mcal{A}^Z$, so there exists a sequence $v_l(t)$, $l\in\nn$ in $C_c^\infty\left(\mcal{X}\times\mcal{Y}\right)$ which converges to $v(t)$ uniformly on $\mcal{X}\times\mcal{Y}$ and such that
\[
\big({\mathcal A}^X + {\mathcal A}^Y + (\nabla_y\,V)'\,\rho\,\nabla_y\big)\,v_l(t) = \mcal{A}^Z v_l(t) \longrightarrow
\comment{\big({\mathcal A}^X + {\mathcal A}^Y + (\nabla_y\,V)'\,\rho\,\nabla_y\big)}\mcal{A}^Z\,v(t)\quad\text{as}\quad l\to\infty
\]   
uniformly on $\mcal{X}\times\mcal{Y}$ as well. Therefore the rightmost expression in \eqref{eq:somemanip} can be written as  the uniform limit
\eq\label{eq:somemanip'}
\begin{split}
&\lim_{l\to\infty} \int_{\mathcal X} \Lambda\,\big({\mathcal A}^X + {\mathcal A}^Y + (\nabla_y\,V)'\,\rho\,\nabla_y\big)\,v_l(t)\,\mathrm{d}x \\
&=\lim_{l\to\infty} \int_{\mathcal X} \Lambda\,{\mathcal A}^X\,v_l(t)
+\big(\Lambda\,{\mathcal A}^Y + \Lambda\,(\nabla_y\,V)'\,\rho\,\nabla_y+({\mathcal A}^Y\Lambda)\big)\,v_l(t)
-({\mathcal A}^Y\Lambda)\,v_l(t)\,\mathrm{d}x \\
&=\lim_{l\to\infty} \int_{\mathcal X} 
\big(\Lambda\,{\mathcal A}^Y + (\nabla_y\,\Lambda)'\,\rho\,\nabla_y+({\mathcal A}^Y \Lambda)\big)\,v_l(t)
+\Lambda\,{\mathcal A}^X\,v_l(t)-\big(({\mathcal A}^X)^*\,\Lambda\big)\,v_l(t)\,\mathrm{d}x \\
&=\lim_{l\to\infty} \int_{\mathcal X} 
\big(\Lambda\,{\mathcal A}^Y + (\nabla_y\,\Lambda)'\,\rho\,\nabla_y+({\mathcal A}^Y\Lambda)\big)\,v_l(t)
\,\mathrm{d}x,
\end{split}
\en
with the second and third identities being consequences of $V=\log \Lambda$, the equation \eqref{PDE}, and the defining property of the adjoint operator $({\mathcal A}^X)^*$ (see, e.g., \cite[Definition B.8]{EN}).}

\medskip
We now aim to simplify the integrand in the final term to $\mcal{A}^Y(\Lambda v_l(t))$. Fix $x \in \mcal{X}$. We will momentarily suppress the dependence of all functions on $x$. Then, since $\Lambda, v_{l}(t) \in \mcal{D}(\mcal{A}^Y)$, we have that $(\Lambda \pm v_l(t))(Y(s))$, $s \geq 0$ are semimartingales. Moreover, by Lemma \ref{lowreggito} in the appendix, we can identify the quadratic variations of these semimartingales as 
\[
\big\langle (\Lambda \pm v_l(t))(Y(\cdot)) \big\rangle_s = \int_0^s \nabla_y (\Lambda \pm v_l(t))(Y(\tau))' \rho(Y(\tau)) \nabla_y (\Lambda \pm v_l(t))(Y(\tau))\,\mathrm{d}\tau .
\]
Due to the polarization identity (\cite[Theorem IV.1.9]{RY}), we can identify the covariation between $\Lambda(Y(\cdot))$ and $v_l(t)(Y(\cdot))$ as 
\[
\mathrm{d}\big\langle \Lambda(Y(\cdot)),v_l(t)(Y(\cdot)) \big\rangle_s = \nabla_y \Lambda (Y(s))' \rho(Y(s)) \nabla_y v_l(t)(Y(s))\hspace{1.6pt}\mathrm{d}s.
\]
The product rule for semimartingales implies that 
\[
(\Lambda v_l(t))(Y(s)) - (\Lambda v_l(t))(Y(0)) - \int_0^s \big( \Lambda\mcal{A}^Yv_l(t) + v_l(t)\mcal{A}^Y\Lambda+ (\nabla_y \Lambda)' \rho \nabla_y v_l(t)\big)(Y(\tau)) \,\mathrm{d}\tau
\]
is a bounded local martingale on every compact time interval, and therefore a true martingale. (Recall the compact support of $v_l(t)$.) Therefore, by \cite[Proposition VII.1.7]{RY}, we have that $\Lambda v_l(t) \in \mcal{D}(\mcal{A}^Y)$ with 
\eq\label{prod_rule}
{\mathcal A}^Y \big(\Lambda\,v_l(t)\big)
=\Lambda{\mathcal A}^Y v_l(t) + (\nabla_y\,\Lambda)'\,\rho\nabla_y\,v_l(t) + ({\mathcal A}^Y \Lambda)\,v_l(t),
\en
thus, simplifying the end result of \eqref{eq:somemanip'} to $\lim_{l\to\infty} \int_{\mathcal X} {\mathcal A}^Y \big(\Lambda\,v_l(t)\big)\,\mathrm{d}x$.

\medskip
\color{black}

Finally, thanks to the compactness of the support of $v_l(t)$ and the regularity assumptions on $\Lambda$ we can approximate the integrals $\int_{\mathcal X} {\mathcal A}^Y \big(\Lambda\,v_l(t)\big)\,\mathrm{d}x$, $\int_{\mathcal X} \Lambda\,v_l(t)\,\mathrm{d}x$ uniformly by sums 
\[
\sum_{r=1}^R \mathrm{vol}(\mcal{X}_r)\,{\mathcal A}^Y \big(\Lambda(\cdot,x_r)\,v_l(t)(x_r,\cdot)\big),\quad
\sum_{r=1}^R \mathrm{vol}(\mcal{X}_r)\,\Lambda(\cdot,x_r)\,v_l(t)(x_r,\cdot),
\]
where $\{\mcal{X}_r:\,r=1,2,\ldots,R\}$ are partitions of $\cup_{y \in \mcal{Y}}\text{supp}(v_l(t)(\cdot,y))$ into disjoint bounded measurable sets, $\mathrm{vol}$ stands for the Euclidean volume, and $x_r\in\mcal{X}_r$, $r=1,2,\ldots,R$. Passing to the limit $R\to\infty$ and appealing to the closedness of ${\mathcal A}^Y$ we obtain
\[
\lim_{l\to\infty} \int_{\mathcal X} {\mathcal A}^Y \big(\Lambda\,v_l(t)\big)\,\mathrm{d}x 
= \lim_{l\to\infty} {\mathcal A}^Y \bigg(\int_{\mathcal X}  \Lambda\,v_l(t)\,\mathrm{d}x\bigg).
\]
Recalling that we started from a limit $l\to\infty$ that was uniform in $y$ and using the closedness of ${\mathcal A}^Y$ once again we identify the latter limit as ${\mathcal A}^Y u(t)$ which gives the claim.

\medskip

\noindent\textbf{Step 4}. We now claim that for all bounded and measurable $h$ on $\mcal{X} \times \mcal{Y}$, we have the following identity:
\eq\label{duallaw}
\mathbb{E}\big[h(Z_1(t),Z_2(t)) \,|\, Z_2(0)=y\big] = \mathbb{E}\bigg[\int_{\mcal{X}} \Lambda(Y(t),x)\,h(x,Y(t))\,\mathrm{d}x \,\bigg| \,Y(0)=y\bigg].
\en

By applying the claim in Step 3 to $u(0)$, we find that the function $y \hspace{-1.4pt}\rightarrow \hspace{-1.4pt}\int_{\mcal{X}}\Lambda(y,x)\,h(x,y)\,\mathrm{d}x$ is in $ \mcal{D}(\mcal{A}^Y)$ for all $h \in \mcal{D}(\mcal{A}^Z)$. By Proposition II.6.2 in \cite{EN}, the solution to equation (\ref{MP}) is unique, and we therefore have the identity for all $h \in \mcal{D}(\mcal{A}^Z)$. By Theorem 17.4 in \cite{Ka}, $\mcal{D}(\mcal{A}^Z)$ is dense in $C_0(\mcal{X} \times \mcal{Y})$ and so the above identity extends to the latter class of functions. Since a finite measure is uniquely determined by its action on $C_0(\mcal{X}\times \mcal{Y})$ functions, this concludes Step 4. 

\medskip

\noindent\textbf{Step 5.} We now prove condition \eqref{i2} in Definition \ref{idef}.
For a bounded, measurable function $h$ on $\mcal{X}$, the right-hand side of (\ref{duallaw}) is $Q_t L h$. For this same $h$, in view of our assumption on the initial distribution of $Z$, the left-hand side can be expanded as 
\[
\int_{\mcal{X}} \Lambda(y,x)\, \mathbb{E}\big[h(Z_1(t)) \,|\, Z_2(0)=y,Z_1(0)=x\big]\,\mathrm{d}x = \int_{\mcal{X}}\, \Lambda(y,x) \,\mathbb{E}\big[h(Z_1(t))\, |\, Z_1(0)=x\big]\,\mathrm{d}x,
\]
where the equality follows from Step 2. Due to Step 1, the term on the right-hand side can be identified as $L P_t h$. This proves condition (\ref{i2}).

\comment{To this end, \color{blue} we need to show that for any Borel set $B$ of $\mathbb{R}^m$ and $f$ the indicator function of $B \cap \mcal{X}$ we have the identity
\eq\label{QLf=LPf}
Q_t\,L\,f=L\,P_t\,f.
\en
{\color{blue} The $\pi$-$\lambda$ theorem implies that it suffices to show the equality for bounded open sets $B$. Furthermore, due to the fact that $P_t$ and $L$ assign zero mass to $\partial \mcal{X}$, it suffices to consider sets $B$ which are bounded open sets compactly contained in the interior of $\mcal{X}$. By the smooth Urysohn's Lemma (\cite[Theorem 8.18]{Fol}), indicator functions of these sets $B$ are pointwise limits of uniformly bounded $C_c^{\infty}(B)$ functions (which automatically belong to $\mcal{D}(\mcal{A}^X) )$. Therefore, \eqref{QLf=LPf} for all indicator functions is a consequence of \eqref{QLf=LPf} for the $f \in C_c^{\infty}(\overset{\circ}{\mcal{X}})$ and the Dominated Convergence Theorem. Therefore, fixing an $f \in C_c^{\infty}(\overset{\circ}{\mcal{X}})$, \color{black} we need to prove
\eq\label{comm_id}
\ev\left[\int_{\mathcal X} \Lambda(Z_2(t),x)\,f(x)\,\mathrm{d}x\; \middle | \; Z_2(0)=y\right]=\int_{\mathcal X} \Lambda(y,x)\,\ev\left[f(Z_1(t))\mid Z_1(0)=x\right]\mathrm{d}x
\en
for all $t\ge0$. 

\medskip

{To this end, we define 
\[
\begin{split}
& u(t):\;\mcal{Y}\to\rr,\quad y\mapsto\ev\left[\int_{\mathcal X} \Lambda(Z_2(t),x)\,f(x)\,\mathrm{d}x\; \middle |  \; Z_2(0)=y\right], \\ 
& v(t):\;\mcal{X}\to\rr,\quad x\mapsto\ev\left[f(Z_1(t))|Z_1(0)=x\right] 
\end{split}
\]
and consider the computation 
\[
\int_{\mcal{X}} \Lambda\,({\mathcal A}^X f)\,\mathrm{d}x 
= \int_{\mcal X} \big(({\mathcal A}^X)^* \Lambda\big)\,f\,\mathrm{d}x
= \int_{\mcal X} ({\mathcal A}^Y \Lambda)\,f\,\mathrm{d}x
\]
where the first identity follows directly from the definition of the adjoint operator $({\mathcal A}^X)^*$ and the second identity from \eqref{PDE}. Now, a locally uniform approximation of the latter integral by sums together with the locality (cf. \cite[Theorem 17.24]{Ka}) and the closedness of ${\mathcal A}^Y$ imply as in Step 2 that the function $y\mapsto \int_{\mcal{X}} \Lambda(y,x)\,f(x)\,\mathrm{d}x$ belongs to the domain of ${\mathcal A}^Y$. Consequently, the result of Step 2 shows} that $u(t)$, $t\ge0$ is a solution of the Cauchy problem for the Kolmogorov forward equation:
\eq\label{CP1}
\frac{\mathrm{d}}{\mathrm{d}t}\,u(t) = {\mathcal A}^Y u(t),
\quad u(0,\cdot)=\int_{\mathcal X} \Lambda(\cdot,x)\,f(x)\,\mathrm{d}x,
\en
{with the meaning of the equation $\frac{\mathrm{d}}{\mathrm{d}t}\,u(t) = {\mathcal A}^Y u(t)$ as specified in the beginning of Step 2.} By Proposition II.6.2 in \cite{EN} the solution of the problem \eqref{CP1} is unique and, thus, to prove \eqref{comm_id} it suffices to check that the right-hand side of \eqref{comm_id}, given by $\int_{\mathcal X} \Lambda(\cdot,x)\,v(t)(x)\,\mathrm{d}x$, also solves \eqref{CP1}. 

\medskip

{By Step 1 and Theorem 17.6 in \cite{Ka} the function $t\mapsto v(t)$ is differentiable with respect to the uniform norm on $C_0(\mcal{X})$ with $\frac{\mathrm{d}}{\mathrm{d}t}\,v(t)={\mathcal A}^X v(t)$ for all $t\ge0$, so that for every $y\in\mcal{Y}$ and $t\ge0$, 
\[
\begin{split}
& \frac{\mathrm{d}}{\mathrm{d}t} \int_{\mathcal X} \Lambda(y,x)\,v(t)(x)\,\mathrm{d}x 
= \int_{\mathcal X} \Lambda(y,x)\,\frac{\mathrm{d}}{\mathrm{d}t}\,v(t)(x)\,\mathrm{d}x 
= \int_{\mathcal X} \Lambda(y,x)\,\left({\mathcal A}^X\,v(t)\right)(x)\,\mathrm{d}x \\ 
& = \int_{\mathcal X} \big(({\mathcal A}^X)^*\Lambda\big)(y,x)\,v(t)(x)\,\mathrm{d}x 
= \int_{\mathcal X} ({\mathcal A}^Y \Lambda)(y,x)\,v(t)(x)\,\mathrm{d}x 
= {\mathcal A}^Y\;\int_{\mathcal X} \Lambda(y,x)\,v(t)(x)\,\mathrm{d}x.
\end{split}
\]
Here the third identity follows from the definition of the adjoint operator $({\mathcal A}^X)^*$, the fourth identity is a result of \eqref{PDE}, and the fifth identity can be obtained by approximating 
$v(t)$ uniformly by a sequence of $C^\infty_c({\mathcal X})$ functions $v_l(t)$, $l\in\nn$ such that ${\mathcal A}^X v_l(t)\to{\mathcal A}^X v(t)$ uniformly as well, interchanging ${\mathcal A}^Y$ with the integral over ${\mathcal X}$ as before, and passing to the limit $l\to\infty$ by means of \eqref{PDE} and the closedness of ${\mathcal A}^Y$, respectively. The proof of \eqref{QLf=LPf} is complete. }

}}

\medskip

\color{black}
\nin\tbf{Step 6.} We now prove condition \eqref{i4} of Definition \ref{idef}. The main claim is an iteration of the previous step. 

\medskip

\nin\tbf{Claim:} Fix $k \in \nn$, and let $0=t_0 < t_1 < \ldots < t_k=t$ be distinct time points. Let $\mcal{G}$ denote the sub-$\sigma$-algebra of $\mcal{F}^Y_t$ generated by $\big(Z_2(t_i),\; i=0,1,\ldots, k\big)$. Then, for all bounded measurable functions $f$ on $\mcal{X}$, we have 
\eq\label{eq:mkvlink}
\ev[f(Z_1(t))\,\big|\,\mcal{G}] = (Lf)(Z_2(t)). 
\en

\smallskip

The proof of the claim proceeds by induction {over $k$}. First, consider the case of $k=1$ {which amounts to showing
\eq\label{i5claim}
\ev\big[f(Z_1(t))\,g(Z_2(t))\,\big|\,Z_2(0)=y\big]
=\ev\big[(Lf)(Z_2(t))\,g(Z_2(t))\,\big|\,Z_2(0)=y\big]
\en
for all $y\in\mcal{Y}$ and bounded measurable functions $f$ on $\mcal{X}$ and $g$ on $\mcal{Y}$. Note that by applying (\ref{duallaw}) to $g$, we get the identity
\[
\mathbb{E}[g(Z_2(t)) \,|\, Z_2(0)=y] = \mathbb{E}[g(Y(t)) \,| \,Y(0)=y].
\]
Hence, the $k=1$ case follows directly from (\ref{duallaw}).

\medskip

Now, suppose the claim holds true for some $k\in \nn$. Then, the conditional expectation operator of $Z_1(t_k)$ given $\left( Z_2(0),\ldots,Z_2(t_{k})\right)$ is again $L$. To show that the claim holds true for $(k+1)$, one can repeat the argument for $k=1$ for the Feller-Markov process $Z(t_k+t)$, $t\ge 0$ after conditioning on $\left(Z_2(0),\ldots,Z_2(t_{k}) \right)$. This completes the proof of the claim. 

\medskip

We have shown so far that, for any bounded measurable function $f$ on $\mathcal{X}$, any $k\in\nn$, and any bounded measurable function $g$ on $\mcal{Y}^{k+1}$, we have
\[
\ev\big[f(Z_1(t_k))\,g(Z_2(t_0),\ldots,Z_2(t_k))\big]
=\ev\big[{(Lf)(Z_2(t_k))}\,g(Z_2(t_0),\ldots,Z_2(t_k))\big].
\] 
Since the $\sigma$-algebra $\mcal{F}^Y_t$ is generated by the coordinate projections, an application of the Monotone Class Theorem yields condition \eqref{i4}. 

\medskip

\nin\tbf{Step 7.} We now argue that $Z_2 \stackrel{d}{=} Y$. Given a measurable space $(\Omega,\mcal{F})$, denote by $B(\Omega)$ the set of bounded measurable functions on $\Omega$. Denote the Markov semigroup of $Z$ by $(R_t)$ and define the transition kernel $\Bar{\Lambda}$ from $ \mcal{Y}$ to $ \mcal{X} \times \mcal{Y}$ by $\Bar{\Lambda}(y',\mathrm{d}(y,x)) = \delta_{y'}(\mathrm{d}y)\Lambda(y,x)\mathrm{d}x$ where $\delta_{y'}(\mathrm{d}y)$ is a point mass at $y'$. Let $\Bar{L}$ be the integral operator of $\bar{\Lambda}$. Finally, define the function $\phi(x,y) = y$ and the operator $\Phi : B(\mcal{Y}) \rightarrow B(\mcal{X} \times \mcal{Y})$ by $\Phi f = f \circ \phi$. In view of our assumption on the initial distribution of $Z$, we can apply (\ref{duallaw}) to a function $f \in B(\mcal{Y})$ and arrive at the equality of kernels $\Bar{L} R_t \Phi = Q_t$. Applying (\ref{duallaw}) to a function $h \in B(\mcal{X}\times\mcal{Y})$ yields the equality $Q_t \Bar{L} = \bar{L} R_t$. One can also easily see that $\bar{L}\Phi$ is the identity operator on $B(\mcal{Y})$. Therefore, the assumptions of Theorem 2 in \cite{RP} are satisfied, and we get (under our assumptions on the initial distribution of $Z$) that $\phi(Z) = Z_2$ is a Markov process with transition semigroup $(Q_t)$.

\comment{{In view of the Monotone Class and the Monotone Convergence Theorems}, it is enough to show that, for any $k\in\nn$ and any choice of $0=t_0 < t_1 < \ldots < t_k=t$, {$y_0,\ldots,y_{k-1}\in\mcal{Y}$}, we have the correct transition probability:
\eq\label{eq:transz2}
\ev\big[f(Z_2(t))\,\big|\,Z_2(t_0)=y_0,\ldots,Z_2(t_{k-1})=y_{k-1}\big] = (Q_{t-t_{k-1}} f)(y_{k-1}) 
\en  
for all $f\in C_c^\infty(\mcal{Y})$. To this end, recall from Step 4 that the conditional distribution of $Z_1(t_{k-1})$ given $Z_2(t_0)=y_0,\ldots,Z_2(t_{k-1})=y_{k-1}$ has density $\Lambda(y_{k-1},\cdot)$, so that the conditional expectation on the left-hand side of \eqref{eq:transz2} can be written as
\[
\int_{\mcal{X}} \Lambda(y_{k-1},x)\,\ev\big[f(Z_2(t))\,\big|\,Z_2(t_0)=y_0,\ldots,Z_2(t_{k-1})=y_{k-1},\,Z_1(t_{k-1})=x\big]\,\mathrm{d}x.
\]
By the Feller-Markov property of $Z$ and Theorem 17.6 in \cite{Ka} the latter conditional expectation is a function of $x$ and $y_{k-1}$ only and satisfies the Kolmogorov forward equation for ${\mathcal A}^Z$ in those variables. Thus, the left-hand side of \eqref{eq:transz2} is a function of $y_{k-1}$ only and, by an identical argument as in Step 2, solves the Kolmogorov forward equation for ${\mathcal A}^Y$ in the variable $y_{k-1}$. Clearly, the right-hand side of \eqref{eq:transz2} solves the same equation with the same initial condition at $t=t_{k-1}$, and we conclude using the uniqueness theorem for the latter (cf. Proposition II.6.2 in \cite{EN}).}  

\medskip

\noindent\textbf{Step 8.} We now turn to the proof of (\ref{eq:condlaw}). Denote the transition kernel of the joint process $Z$ by $(R_t)$. For any $h \in \mcal{D}(\mcal{A}^Z)$, we have that $(R_t h)(x_0,y_0) = (\mcal{A}^Z h)(x_0,y_0) + o(t)$. Therefore, in order to prove condition (\ref{eq:condlaw}), it suffices to show that $(\tilde{R}_t h)(x_0,y_0) = (\mcal{A}^Z h)(x_0,y_0) + o(t)$ where $(\tilde{R}_t)$ is defined by (\ref{approxkernel}) and the error term is allowed to depend on $h$ and $(x_0,y_0)$. This will follow from Step 1 in the proof of Theorem 2 (which has the same assumptions on $\Lambda$).
\ep
\color{black}
\comment{ \noindent\textbf{Step 7.} 
{Finally, we turn to the proof of condition \eqref{i5} in Definition \ref{idef}. To this end, we fix a $t\geq0$ and recall that, given $Z_2(0)=y_0$ for some $y_0\in\mcal{Y}$, the joint distribution/density of $\big(Z_1(0),Z_1(t)\big)$ at $(x_0, x_1)$ is $\Lambda(y_0,x_0)\,P_t(x_0,x_1)$ by condition \eqref{i3} in Definition \ref{idef}. The term $\Lambda(y_0,x_0)\,P_t(x_0,x_1)$ is an abuse of notation since $\Lambda$ is a density and $P_t$ is a semigroup. However, the meaning of the statement can be made rigorous through integration.

Similarly, given $Z_2(0)=y_0$, the joint distribution/density of $\big(Z_2(t),Z_1(t)\big)$ at $(y_1, x_1)$ reads $Q_t(y_0,y_1)\,\Lambda(y_1,x_1)$ by condition \eqref{i4} in Definition \ref{idef}. Hence, the conditional marginal distribution of $Z_1(t)$ at $x_1$ computes to $\int_{\mcal{Y}} Q_t(y_0,\mathrm{d}\tilde{y}_1)\,\Lambda(\tilde{y}_1,x_1)$. 

Therefore the conditional independence of $Z_1(0)$ and $Z_2(t)$ given $\big(Z_2(0),Z_1(t)\big)$ is equivalent to the statement that the joint distribution/density of $\big(Z_1(0), Z_1(t),Z_2(t)\big)$, given $Z_2(0)=y_0$, at $(x_0, x_1, y_1)$ is
\[
\int_{\mcal{Y}} Q_t(y_0,\mathrm{d}\tilde{y}_1)\,\Lambda(\tilde{y}_1,x_1)
\,\frac{\Lambda(y_0,x_0)\,P_t(x_0,x_1)}
{\int_{\mcal{X}} \Lambda(y_0,\tilde{x}_0)\,\mathrm{d}\tilde{x}_0\,P_t(\tilde{x}_0,x_1)}
\,\frac{Q_t(y_0,y_1)\,\Lambda(y_1,x_1)}
{\int_{\mcal{Y}} Q_t(y_0,\mathrm{d}\tilde{y}_1)\,\Lambda(\tilde{y}_1,x_1)}.
\]
In view of condition \eqref{i2} in Definition \ref{idef} this amounts to showing that the time $t$ transition operator of the process $Z$ can be written as
\eq\label{Z_transition}
P_t(x_0,x_1)\,\frac{Q_t(y_0,y_1)\,\Lambda(y_1,x_1)}
{\int_{\mcal{Y}} Q_t(y_0,\mathrm{d}\tilde{y}_1)\,\Lambda(\tilde{y}_1,x_1)},
\en
and the rest of the proof is devoted to a derivation of the latter formula.}

\medskip

{Note that, on the time interval $[0,t]$, the law of the diffusion $Z$ is absolutely continuous with respect to the law of the product diffusion $\widehat{Z}=(\widehat{Z}_1,\widehat{Z}_2)$ on $\mcal{X}\times\mcal{Y}$ with generator ${\mathcal A}^X+{\mathcal A}^Y$ and initial condition $\widehat{Z}(0)=Z(0)$, and the corresponding density is given by
\eq\label{Girsanov}
\exp\bigg(\int_0^t \big((\nabla_y V)(\widehat{Z}_2(s),\widehat{Z}_1(s))\big)'\mathrm{d}\widehat{M}_2(s) 
-\frac{1}{2}\,\int_0^t \big\|(\nabla_y V)(\widehat{Z}_2(s),\widehat{Z}_1(s))\big\|_{\rho(\widehat{Z}_2(s))}^2
\,\mathrm{d}s\bigg)
\en
where $\widehat{M}_2$ is the local martingale part of $\widehat{Z}_2$ and $\|y\|_{\rho(\widehat{Z}_2(s))}^2=y'\rho\big(\widehat{Z}_2(s)\big)\,y$. Indeed, start with the product diffusion $\widehat{Z}$, and consider the stopping times
\[
\tau_R=\inf\bigg\{s\ge0:\;\int_0^s \big\|(\nabla_y V)(\widehat{Z}_2(\tilde{s}),\widehat{Z}_1(\tilde{s}))\big\|
_{\rho(\widehat{Z}_2(\tilde{s}))}^2\,\mathrm{d}\tilde{s}\ge R\bigg\},\quad R\in\nn,
\]
as well as the change of measure densities 
\[
\exp\bigg(\int_0^{t\wedge\tau_R} \big((\nabla_y V)(\widehat{Z}_2(s),\widehat{Z}_1(s))\big)'\mathrm{d}\widehat{M}_2(s) 
-\frac{1}{2}\,\int_0^{t\wedge\tau_R} \big\|(\nabla_y V)(\widehat{Z}_2(s),\widehat{Z}_1(s))\big\|_{\rho(\widehat{Z}_2(s))}^2\,\mathrm{d}s\bigg),
\]
$R\in\nn$. By Girsanov's Theorem (see, e.g., Chapter 3, Theorem 5.1 and Proposition 5.12 in \cite{KS}), for each $R\in\nn$, such change of measure is well-defined, and $\widehat{Z}$ solves the SDE for $Z$ on $[0,t\wedge\tau_R]$ under the new measure. Since the law of the solution is unique (Assumption \ref{main_asmp}), it follows that, on each of the events $\{\tau_R\ge t\}$, the law of $Z$ on $[0,t]$ is absolutely continuous with respect to the law of $\widehat{Z}$ on $[0,t]$, and the corresponding density is given by \eqref{Girsanov}. Moreover, the assumed continuity of $\nabla_y V$ and $\rho$ implies that, with probability one,
\[
\int_0^t \big\|(\nabla_y V)(Z_2(s),Z_1(s))\big\|_{\rho(Z_2(s))}^2\,\mathrm{d}s<\infty,
\]
and consequently the events $\{\tau_R\ge t\}$, as $R\rightarrow \infty$, increase to a probability one event. The desired absolute continuity and \eqref{Girsanov} readily follow.}

\medskip

{Next, for any fixed $\epsilon\in(0,t)$, set $t^\epsilon_k=(k\epsilon)\wedge t$, $k=0,1,\ldots,K(\epsilon)$ where $K(\epsilon)$ is the smallest integer greater or equal to ${t}/{\epsilon}$. We claim that the density of \eqref{Girsanov} can be then rewritten as the almost sure limit
\eq\label{Girsanov'}
\begin{split}
& \lim_{\epsilon\downarrow0}\;\exp\bigg(\sum_{k=1}^{K(\epsilon)} \int_{t^\epsilon_{k-1}}^{t^\epsilon_k} 
\big((\nabla_y V)(\widehat{Z}_2(s),\widehat{Z}_1(t^\epsilon_k))\big)'\mathrm{d}\widehat{M}_2(s) \\
&\qquad\qquad\;\,
-\frac{1}{2}\,\sum_{k=1}^{K(\epsilon)}\,\int_{t^\epsilon_{k-1}}^{t^\epsilon_k} \big\|(\nabla_y V)(\widehat{Z}_2(s),\widehat{Z}_1(t^\epsilon_k))\big\|_{\rho(\widehat{Z}_2(s))}^2\,\mathrm{d}s\bigg)
\end{split}
\en
along a suitable subsequence. Indeed, recall that the functions $\nabla_y V$ and $\rho$ are assumed to be continuous. Now, compare the exponents in \eqref{Girsanov} and \eqref{Girsanov'}. The difference between the stochastic integrals can be viewed as a standard Brownian motion evaluated at the quadratic variation of that difference (see, e.g., Chapter 3, Theorem 4.6 in \cite{KS}), and the latter quadratic variation tends to zero almost surely in the limit $\epsilon\downarrow0$ thanks to the uniform continuity of $\nabla_y V$ and the uniform boundedness of $\rho$ on compact sets. It follows that the difference between the stochastic integrals converges to zero in probability and, hence, also almost surely along a suitable subsequence. The uniform continuity of $\nabla_y V$ and the uniform boundedness of $\rho$ on compact sets also show that the difference between the bounded variation terms in the exponents of \eqref{Girsanov} and \eqref{Girsanov'} tends to zero almost surely as well, yielding the representation \eqref{Girsanov'}.}

\medskip

{Note further that, for each $k=1,2,\ldots,K(\epsilon)$, 
\[
\begin{split}
\Lambda(\widehat{Z}_2(t^\epsilon_k),\widehat{Z}_1(t^\epsilon_k))
-\Lambda(\widehat{Z}_2(t^\epsilon_{k-1}),\widehat{Z}_1(t^\epsilon_k))
=\int_{t^\epsilon_{k-1}}^{t^\epsilon_k} 
\big((\nabla_y \Lambda)(\widehat{Z}_2(s),\widehat{Z}_1(t^\epsilon_k))\big)'\mathrm{d}\widehat{M}_2(s) \\
+\int_{t^\epsilon_{k-1}}^{t^\epsilon_k} ({\mathcal A}^Y \Lambda)(\widehat{Z}_2(s),\widehat{Z}_1(t^\epsilon_k))
\,\mathrm{d}s.
\end{split}
\]
This follows from an approximation of $\Lambda(\cdot,\widehat{Z}_1(t^\epsilon_k))$ by infinitely differentiable functions $g_q(\cdot,\widehat{Z}_1(t^\epsilon_k))$, $q\in\nn$ in $C_0(\mcal{Y})$ such that $g_q(\cdot,\widehat{Z}_1(t^\epsilon_k))\to\Lambda(\cdot,\widehat{Z}_1(t^\epsilon_k))$, $(\nabla_y g_q)(\cdot,\widehat{Z}_1(t^\epsilon_k))\to(\nabla_y\Lambda)(\cdot,\widehat{Z}_1(t^\epsilon_k))$, and $({\mathcal A}^Y g_q)(\cdot,\widehat{Z}_1(t^\epsilon_k))\to({\mathcal A}^Y \Lambda)(\cdot,\widehat{Z}_1(t^\epsilon_k))$ uniformly on $\mcal{Y}$ (see Step 2 for a construction of such functions), an application of It\^o's formula to the difference $g_q(\widehat{Z}_2(t^\epsilon_k),\widehat{Z}_1(t^\epsilon_k))-g_q(\widehat{Z}_2(t^\epsilon_{k-1}),\widehat{Z}_1(t^\epsilon_k))$ (recall the independence of $\widehat{Z}_2$ from $\widehat{Z}_1$), and the limit transition $q\to\infty$ in probability. Another application of It\^o's formula (relying on the positivity of $\Lambda$) gives
\[
\begin{split}
\log \Lambda(\widehat{Z}_2(t^\epsilon_k),\widehat{Z}_1(t^\epsilon_k))
-\log \Lambda(\widehat{Z}_2(t^\epsilon_{k-1}),\widehat{Z}_1(t^\epsilon_k))
=\int_{t^\epsilon_{k-1}}^{t^\epsilon_k} 
\big((\nabla_y V)(\widehat{Z}_2(s),\widehat{Z}_1(t^\epsilon_k))\big)'\mathrm{d}\widehat{M}_2(s) \\
+\int_{t^\epsilon_{k-1}}^{t^\epsilon_k} 
\frac{{\mathcal A}^Y \Lambda}{\Lambda}(\widehat{Z}_2(s),\widehat{Z}_1(t^\epsilon_k))
\,\mathrm{d}s
-\frac{1}{2}\,\int_{t^\epsilon_{k-1}}^{t^\epsilon_k} \big\|(\nabla_y V)(\widehat{Z}_2(s),\widehat{Z}_1(t^\epsilon_k))\big\|_{\rho(\widehat{Z}_2(s))}^2\,\mathrm{d}s,
\end{split}
\]
$k=1,2,\ldots,K(\epsilon)$. Inserting this into \eqref{Girsanov'} one obtains
\eq\label{Girsanov''}
\begin{split}
\lim_{\epsilon\downarrow0}\;\prod_{k=1}^{K(\epsilon)} 
\bigg(\frac{\Lambda(\widehat{Z}_2(t^\epsilon_k),\widehat{Z}_1(t^\epsilon_k))}
{\Lambda(\widehat{Z}_2(t^\epsilon_{k-1}),\widehat{Z}_1(t^\epsilon_k))}
\,\exp\bigg(-\int_{t^\epsilon_{k-1}}^{t^\epsilon_k} 
\frac{{\mathcal A}^Y \Lambda}{\Lambda}(\widehat{Z}_2(s),\widehat{Z}_1(t^\epsilon_k))\,\mathrm{d}s\bigg)\bigg).
\end{split}
\en}

\medskip

{For each $k=1,2,\ldots,K(\epsilon)$, one can further perform the following manipulations sample pathwise:
\[
\begin{split}
& \frac{1}{\Lambda(\widehat{Z}_2(t^\epsilon_{k-1}),\widehat{Z}_1(t^\epsilon_k))}\,
\exp\bigg(-\int_{t^\epsilon_{k-1}}^{t^\epsilon_k} 
\frac{{\mathcal A}^Y \Lambda}{\Lambda}(\widehat{Z}_2(s),\widehat{Z}_1(t^\epsilon_k))\,\mathrm{d}s\bigg) \\
& = \frac{1}{\Lambda(\widehat{Z}_2(t^\epsilon_{k-1}),\widehat{Z}_1(t^\epsilon_k))
\Big(1+\int_{t^\epsilon_{k-1}}^{t^\epsilon_k} 
\frac{{\mathcal A}^Y \Lambda}{\Lambda}(\widehat{Z}_2(s),\widehat{Z}_1(t^\epsilon_k))\,\mathrm{d}s+o(\epsilon)\Big)} \\
& = \frac{1}{\Lambda(\widehat{Z}_2(t^\epsilon_{k-1}),\widehat{Z}_1(t^\epsilon_k))
+({\mathcal A}^Y \Lambda)(\widehat{Z}_2(t^\epsilon_{k-1}),\widehat{Z}_1(t^\epsilon_k))(t_k^\epsilon-t_{k-1}^\epsilon)
+o(\epsilon)} \\
& = \frac{1}{\Lambda(\widehat{Z}_2(t^\epsilon_{k-1}),\widehat{Z}_1(t^\epsilon_k))
+\int_0^{t^\epsilon_k-t^\epsilon_{k-1}} 
(Q_s\,{\mathcal A}^Y \Lambda)(\widehat{Z}_2(t^\epsilon_{k-1}),\widehat{Z}_1(t^\epsilon_k))\,\mathrm{d}s+o(\epsilon)} \\
& =\frac{1}{(Q_{t_k^\epsilon-t_{k-1}^\epsilon}\Lambda)(\widehat{Z}_2(t^\epsilon_{k-1}),\widehat{Z}_1(t^\epsilon_k))
+o(\epsilon)},
\end{split}
\]
where in the third equality we have used the Feller property of $(Q_t)$ (cf. property ($\mathrm{F}_2$) on p. 315 in \cite{Ka}) and in the fourth the Kolmogorov forward equation (see, e.g., Theorem 17.6 in \cite{Ka}). Since ${\mathcal A}^Y\Lambda$ and $\Lambda$ are assumed to be continuous (hence, uniformly bounded and uniformly continuous on compact sets), and ${\mathcal A}^Y\Lambda$ is assumed to be bounded on $\mcal{Y}\times K$ for any compact $K\subset\mcal{X}$, all $o(\epsilon)$ error terms above are uniform in $k$ (but may depend on the path of $\widehat{Z}$). At this point, \eqref{Girsanov''} can be rewritten as
\eq\label{Girsanov'''}
\lim_{\epsilon\downarrow0}\;\prod_{k=1}^{K(\epsilon)} 
\frac{\Lambda(\widehat{Z}_2(t^\epsilon_k),\widehat{Z}_1(t^\epsilon_k))}
{(Q_{t_k^\epsilon-t_{k-1}^\epsilon}\Lambda)(\widehat{Z}_2(t^\epsilon_{k-1}),\widehat{Z}_1(t^\epsilon_k))}.
\en
This can be seen by taking the logarithm of the ratio of \eqref{Girsanov'''} and \eqref{Girsanov''}, using thereafter a uniform in $k$, $\epsilon$ positive lower bound on $(Q_{t_k^\epsilon-t_{k-1}^\epsilon}\Lambda)(\widehat{Z}_2(t^\epsilon_{k-1}),\widehat{Z}_1(t^\epsilon_k))$ (resulting from a positive lower bound on $\Lambda$ on the set product of an open neighborhood of the range of $\widehat{Z}_2$ on $[0,t]$ with the range of $\widehat{Z}_1$ on $[0,t]$), and finally noting $K(\epsilon)\,o(\epsilon)=o(1)$.}

\medskip

{To conclude we fix an initial condition $z_0=(x_0,y_0)\in\mcal{X}\times\mcal{Y}$ and a measurable set $\mcal{Z}\subset\mcal{X}\times\mcal{Y}$ and make the following computation:
\[
\begin{split}
\pp\big(Z(t)\in\mcal{Z}\,|\,Z(0)=z_0\big)
= \ev\bigg[\lim_{\epsilon\downarrow0}\;\prod_{k=1}^{K(\epsilon)} 
\frac{\Lambda(\widehat{Z}_2(t^\epsilon_k),\widehat{Z}_1(t^\epsilon_k))}
{(Q_{t_k^\epsilon-t_{k-1}^\epsilon}\Lambda)(\widehat{Z}_2(t^\epsilon_{k-1}),\widehat{Z}_1(t^\epsilon_k))}
\,\mathbf{1}_{\mcal{Z}}(\widehat{Z}(t))\,\Big|\,\widehat{Z}(0)=z_0\bigg] \\
\leq \liminf_{\epsilon\downarrow0}\; 
\ev\bigg[\prod_{k=1}^{K(\epsilon)} 
\frac{\Lambda(\widehat{Z}_2(t^\epsilon_k),\widehat{Z}_1(t^\epsilon_k))}
{(Q_{t_k^\epsilon-t_{k-1}^\epsilon}\Lambda)(\widehat{Z}_2(t^\epsilon_{k-1}),\widehat{Z}_1(t^\epsilon_k))}
\,\mathbf{1}_{\mcal{Z}}(\widehat{Z}(t))\,\Big|\,\widehat{Z}(0)=z_0\bigg] \\
= \liminf_{\epsilon\downarrow0}\;
\int_{\mcal{X}\times\mcal{Y}} P_{t^\epsilon_1}(x_0,x_1)
\,\frac{Q_{t^\epsilon_1}(y_0,y_1)\,\Lambda(y_1,x_1)}
{\int_{\mcal{Y}} Q_{t^\epsilon_1}(y_0,\mathrm{d}\tilde{y}_1)\,\Lambda(\tilde{y}_1,x_1)} 
\ldots \int_{\mcal{X}\times\mcal{Y}} 
P_{t^\epsilon_{K(\epsilon)}-t^\epsilon_{K(\epsilon)-1}}\big(x_{K(\epsilon)-1},x_{K(\epsilon)}\big) \\
\frac{Q_{t^\epsilon_{K(\epsilon)}-t^\epsilon_{K(\epsilon)-1}}\big(y_{K(\epsilon)-1},y_{K(\epsilon)}\big)
\,\Lambda\big(y_{K(\epsilon)},x_{K(\epsilon)}\big)}
{\int_{\mcal{Y}} Q_{t^\epsilon_{K(\epsilon)}-t^\epsilon_{K(\epsilon)-1}}\big(y_{K(\epsilon)-1},
\mathrm{d}\tilde{y}_{K(\epsilon)}\big)\,\Lambda\big(\tilde{y}_{K(\epsilon)},x_{K(\epsilon)}\big)}
\,\mathbf{1}_{\mcal{Z}}\big(x_{K(\epsilon)},y_{K(\epsilon)}\big) \\
= \int_{\mcal{X}\times\mcal{Y}} P_t(x_0,x_1)
\,\frac{Q_t(y_0,y_1)\,\Lambda(y_1,x_1)}
{\int_{\mcal{Y}} Q_t(y_0,\mathrm{d}\tilde{y}_1)\,\Lambda(\tilde{y}_1,x_1)}
\,\mathbf{1}_{\mcal{Z}}(x_1,y_1).
\end{split}
\]
Hereby, the first equality follows from the preceding considerations; the first inequality is a direct consequence of Fatou's Lemma; the second equality is an application of the Markov property of $\widehat{Z}$; and the last equality follows from the fact that the transition operators of \eqref{Z_transition} satisfy the Chapman-Kolmogorov equation, since each of them corresponds to an evolution of the $x$-components according to $P_t$ and the subsequent sampling of the $y$-components given their previous value and the new value of the $x$-components according to the Bayes rule. Finally, note that the first and the last expressions of the latter display give the probability of $\mcal{Z}$ under two different probability measures. By considering the same inequality for the complement of $\mcal{Z}$ in $\mcal{X}\times\mcal{Y}$ it now follows that the probability of $\mcal{Z}$ under the two probability measures must be the same and, since $\mcal{Z}$ was arbitrary, the two probability measures have to be equal.} }

\medskip

We now turn to the proof of Theorem \ref{main2}. 

\medskip

\noindent\textbf{Proof of Theorem \ref{main2}. 
Step 1.} We start by fixing a point $(x_0,y_0)\in\mcal{X}\times\mcal{Y}$ and by assuming condition \eqref{eq:condlaw}. 
{To identify the generator ${\mathcal A^Z}$ of $Z$, consider a $C_c^\infty(\mcal{X}\times \mcal{Y})$-function $h$ with the appropriate boundary conditions.

We claim first that  the probability of $Z_1$ leaving a small enough ball around $x_0$ decays exponentially in $\frac{1}{t}$ as $t \downarrow 0$. If $X$ satisfies Assumption 1(a) or $x_0$ is in the interior of $\mcal{X}$, this is a consequence of the local boundedness of the drift and diffusion coefficients. If $X$ satisfies Assumption 1(b) and $x_0$ is on the boundary of $\mcal{X}$, one can apply a (Lipschitz) transformation as in Section $1.3$ of \cite{AO} to (up until the exit of a small ball) reduce the problem to that of locally bounded coefficients in the half-space with normal reflection.  The Skorokhod map on this space is Lipschitz by Theorem 2.2 in \cite{DIs}. Thus, again due to the local boundedness of the coefficients, the probability of leaving a small ball decays exponentially in $\frac{1}{t}$. Therefore, when considering the integral $\tilde{R}_t h$, it suffices to integrate the $x_1$ variable over a compact region $K$ containing a neighborhood of $x_0$. Also, due to the exponentially small probability of leaving a small ball around $x_0$, we may further restrict the integral to the compact set $\hat{K} = K \cap \overline{\cup_{y_1 \in \mcal{Y}}\text{supp}\big( h(\cdot,y_1)\big)}$ where $\overline{E}$ denotes the closure of a set $E$.  


\medskip 

{Recall that, for any $x_1\in\mcal{X}$, $\Lambda(\cdot,x_1)$ belongs to the domain of ${\mathcal A}^Y$ by assumption. Therefore the product rule \eqref{prod_rule} for ${\mathcal A}^Y$ shows that $\Lambda(\cdot,x_1)\,h(x_1,\cdot)$ must also belong to the domain of ${\mathcal A}^Y$ for every $x_1\in\mcal{X}$. 
Using  (\ref{eq:condlaw}) and the Kolmogorov forward equation for the Feller semigroup $(Q_t)$ twice (with the initial conditions $\Lambda(\cdot,x_1)\,h(x_1,\cdot)$ and $\Lambda(\cdot,x_1)$, respectively), one obtains
\eq\label{eq:condlaw_re}
\begin{split}
&\ev[h(Z_1(t),Z_2(t))\mid Z(0)=(x_0,y_0)] \\
&=\int_{\hat{K}} \frac{\Lambda(y_0,x_1)\,h(x_1,y_0)+t\,\mathcal{A}^Y\big(\Lambda(\cdot,x_1)\,h(x_1,\cdot)\big)(y_0)+t\,\epsilon_1(t,x_1,y_0)}
{\Lambda(y_0,x_1)+t\,{\mathcal A}^Y \Lambda(\cdot,x_1)(y_0)+t\,\epsilon_2(t,x_1,y_0)}\,P_t(x_0,\mathrm{d}x_1) + o(t),
\end{split}
\en
where the constant in $o(t)$ depends only on $h$ and $(x_0,y_0)$ and where we have defined
\begin{eqnarray}\label{uniferror}
&& \epsilon_1(t,x_1,y_0)=\frac{1}{t}\,\int_0^t Q_s\big({\mathcal A}^Y(\Lambda(\cdot,x_1)\,h(x_1,\cdot))\big)(y_0)\,\mathrm{d}s
-{\mathcal A}^Y(\Lambda(\cdot,x_1)\,h(x_1,\cdot))(y_0), \\
&& \epsilon_2(t,x_1,y_0)=\frac{1}{t}\,\int_0^t Q_s({\mathcal A}^Y\Lambda(\cdot,x_1))(y_0)\,\mathrm{d}s
-{\mathcal A}^Y\Lambda(\cdot,x_1)(y_0).
\end{eqnarray}
Note that, in view of a product rule for ${\mathcal A}^Y$ as in \eqref{prod_rule} and the continuity of $\Lambda$, $\nabla_y\Lambda$, and ${\mathcal A}^Y\Lambda$, the function ${\mathcal A}^Y(\Lambda\,h)$ is uniformly bounded on $\hat{K} \times \mcal{Y}$ and uniformly continuous on $\hat{K}\times \tilde{K}$ for any compact $\tilde{K}\subset\mcal{Y}$. Moreover, by assumption the same holds for the function ${\mathcal A}^Y \Lambda$. It follows that the error terms $\epsilon_1$ and $\epsilon_2$ in \eqref{eq:condlaw_re} converge to zero in the limit $t\downarrow 0$ uniformly in $\hat{K}$\color{black}.

\medskip

{Next, we use the elementary expansion
\eq\label{elem_exp}
\frac{a_1+ta_2+ta_3}{b_1+tb_2+tb_3}=\frac{a_1}{b_1}+t\,\frac{a_2b_1-a_1b_2}{b_1^2}
+t\,\frac{a_3b_1^2-a_1b_1b_3+t(a_1b_2^2+a_1b_2b_3-a_2b_1b_2-a_2b_1b_3)}{b_1^3+tb_1^2(b_2+b_3)}.
\en
Consider the first term on the right-hand side of \eqref{eq:condlaw_re} (i.e., the term preceding ``$+o(t)$"). By applying \eqref{elem_exp} to the fraction inside the integral, it can be rewritten as
\eq\label{eq:condlaw_re1}
\int_{\hat{K}}\Big( h(x_1,y_0)
+t\,\frac{{\mathcal A}^Y(\Lambda(\cdot,x_1)\,h(x_1,\cdot))(y_0)-h(x_1,y_0)\,{\mathcal A}^Y\Lambda(\cdot,x_1)(y_0)}{\Lambda(y_0,x_1)}
+t\,\epsilon_3\Big)\,P_t(x_0,\mathrm{d}x_1)
\en
where an explicit expression for the remainder $\epsilon_3 = \epsilon_3(t,x_1,y_0)$ can be read off from \eqref{elem_exp}. The uniform in $x_1\in \hat{K}$ control on $\epsilon_1,\,\epsilon_2$ together with the continuity of $\Lambda\,h$, ${\mathcal A}^Y(\Lambda\,h)$, $\Lambda$, and ${\mathcal A}^Y\Lambda$ show further that $\epsilon_3$ converges to zero in the limit $t\downarrow 0$ uniformly in $x_1\in \hat{K}$.} 

\medskip

{We now interchange sum and integration in the formula \eqref{eq:condlaw_re1}. First, since $h(\cdot,y_0)$ belongs to the domain of ${\mathcal A}^X$, one has
\[
\int_{\hat{K}} h(x_1,y_0)\,P_t(x_0,\mathrm{d}x_1)=h(x_0,y_0)+t\,(\mcal{A}^Xh(\cdot,y_0))(x_0)+o(t),
\quad t\downarrow0.
\] 
Second, a product rule for ${\mathcal A}^Y$ as in \eqref{prod_rule} and the continuity in the variable $x_1$ of all the functions involved yield
\[
\begin{split}
& \int_{\hat{K}} t\,\frac{{\mathcal A}^Y(\Lambda(\cdot,x_1)\,h(x_1,\cdot))(y_0)-h(x_1,y_0)\,{\mathcal A}^Y\Lambda(\cdot,x_1)(y_0)}
{\Lambda(y_0,x_1)}\,P_t(x_0,\mathrm{d}x_1) \\
& =t\,\int_{\hat{K}}\frac{(\nabla_y\Lambda(y_0,x_1))'\,\rho(y_0)\,\nabla_y h(x_1,y_0)
+\Lambda(y_0,x_1)\,({\mathcal A}^Yh(x_1,\cdot))(y_0)}{\Lambda(y_0,x_1)}\,P_t(x_0,\mathrm{d}x_1) \\
&=t\,\big((\nabla_y V(y_0,x_0))'\,\rho(y_0)\,\nabla_y h(x_0,y_0)+({\mathcal A}^Y h(x_0,\cdot))(y_0)\big)+o(t),\quad \text{as}\; t\downarrow0.
\end{split}
\]
Lastly, the uniform in $x_1 \in \hat{K}$  control on $\epsilon_3$ reveals
\[
\int_{\hat{K}} t\,\epsilon_3(t,x_1,y_0)\,P_t(x_0,\mathrm{d}x_1)=o(t),\quad t\downarrow0.
\]
Putting everything together one obtains
\[
\ev[h(Z_1(t),Z_2(t))\, \big|\, Z(0)=(x_0,y_0)]=h(x_0,y_0)+t\,({\mathcal A}^Z h)(x_0,y_0)+o(t),\quad t\downarrow0 
\]
with ${\mathcal A}^Z$ of \eqref{Zgen}.
We conclude by \cite[Theorem 1.33]{LM} that $h \in \mcal{D}(\mcal{A}^Z)$ and $\mcal{A}^Z h$ is given by the application of the differential operator to $h$. 

\medskip

\noindent\textbf{Step 2.} {It remains to prove \eqref{PDE} and \eqref{comm}. To this end, let $f$ be a bounded measurable function on $\mcal{X}$. By the intertwining identity (see Definition \ref{def:intertwin}), $L\,P_t\,f = Q_t\,L\,f$ for all $t\ge0$, that is,
\eq\label{eq:thm2step2disp1}
\int_{\mcal{X}} \Lambda(y,x)\,(P_t\,f)(x)\,\mathrm{d}x=Q_t\,\int_{\mcal{X}} \Lambda(y,x)\,f(x)\,\mathrm{d}x,
\quad y\in\mcal{Y},\;t\ge0. 
\en
Let $(P^*_t)$ denote the adjoint semigroup associated with $(P_t)$ acting on the space of signed Borel regular measures on $\mcal{X}$ of finite total variation (i.e., the Banach space dual to $C_0(\mcal{X})$ by the Riesz Representation Theorem). Using Fubini's Theorem we obtain from \eqref{eq:thm2step2disp1}: 
\[
\int_{\mcal{X}} f(x)\,P_t^*\Lambda(y,\mathrm{d}x)
=\int_{\mcal{X}} f(x)\,(Q_t\,\Lambda)(y,x)\,\mathrm{d}x,\quad y\in\mcal{Y},\;t\ge0. 
\]     
Consequently, for all $y\in\mcal{Y}$ and $t>0$, one has the equality of measures $P^*_t\Lambda(y,\mathrm{d}x)=(Q_t\,\Lambda)(y,x)\,\mathrm{d}x$ on $\mcal{X}$, yielding
\[
\frac{P^*_t \Lambda(y,\mathrm{d}x)-\Lambda(y,x)\,\mathrm{d}x}{t}
=\frac{(Q_t\,\Lambda)(y,x)-\Lambda(y,x)}{t}\,\mathrm{d}x.   
\]
For fixed $y\in\mcal{Y}$ and in the limit $t\downarrow0$, the left-hand side converges weakly to $\left({\mathcal A^X}^* \right)\Lambda(y,\mathrm{d}x)$ (see, e.g., Section II.2.5 in \cite{EN}). Due to the Kolmogorov forward equation for the Feller semigroup $(Q_t)$, the ratio on the right-hand side converges to ${\mathcal A}^Y\Lambda(y,x)$ locally uniformly in $x$ as discussed in Step 1. Consequently, the measure $({\mathcal A^X})^*\Lambda(y,\mathrm{d}x)$ must have ${\mathcal A}^Y\Lambda(y,x)$ as its density, i.e., \eqref{PDE} holds.}

\medskip

{To obtain \eqref{comm} we pick a $C_0(\mcal{X})$-function $f$ in the domain of ${\mathcal A}^X$ and rewrite the intertwining identity as 
\eq\label{pre_comm}
\frac{L\,P_t\,f  - L\,f}{t} = \frac{Q_t\,L\,f - L\,f}{t}, \quad t>0.
\en
Since $f$ is in the domain of ${\mathcal A}^X$, one has $\frac{P_t\,f-f}{t}\to{\mathcal A}^X f$ in $C_0(\mcal{X})$ in the limit $t\downarrow0$ and, hence, $\frac{L\,P_t\,f  - L\,f}{t}\to L{\mathcal A}^X f$ in $C_0(\mcal{Y})$. Note that, being a stochastic transition operator, $L$ is a bounded linear operator from $C_0(\mcal{X})$ to $C_0(\mcal{Y})$. Therefore the uniform (in $y$) $t\downarrow0$ limit of the right-hand side of \eqref{pre_comm} must exist as well and, by the definition of the generator ${\mathcal A}^Y$, be given by ${\mathcal A}^Y Lf$. The commutativity relation \eqref{comm} readily follows. \ep}

\bigskip

Two restrictions of Theorem \ref{main1} are the assumptions that the kernel $\Lambda$ satisfies \eqref{PDE} on the entire space $\mX \times \mY$ and is stochastic. This leaves out situations where the domain of $Z$ is not of product form or $\Lambda$ is a nonnegative, but not necessarily stochastic solution of \eqref{PDE}. Our next results relax these constraints and will allow us to cover several important examples. For the sake of clarity we keep the following theorem restricted to the case where the state space of $Z$ is (almost) polyhedral and the components of $Z$ are driven by independent standard Brownian motions. This covers all known examples, although it is not hard to see that the scope of the theorem can be enlarged significantly.

Consider the set-up of Assumption \ref{main_asmp} with $a_{ij}=\delta_{ij}$ and $\rho_{kl}=\delta_{kl}$ (i.e., identity diffusion matrices). As before, we write $z\in \rr^{m+n}$ as $z=(x,y)$ where $x\in\rr^m$ and $y\in\rr^n$. Let $D\subset\rr^{m+n}$ be a domain such that:
\begin{enumerate}

\item[(i)] \label{i} $D$ is convex with nonempty interior.
\item[(ii)] \label{ii} The projection of $D$ on $\rr^m$, given by $\cup_{y\in \rr^n} D(\cdot,y)$, is $\mX$, and the projection of $D$ on $\rr^n$, given by $\cup_{x\in \rr^m} D(x,\cdot)$, is $\mY$  which we assume is open.
\item[(iii)] \label{iii}For every $y\in\mY$, the domain $D(y):=D(\cdot,y)$ has a boundary $\partial D(y)$ such that the Divergence Theorem and Green's second identity hold for $D(y)$. For example, piecewise smooth boundaries suffice.
\item[(iv)] \label{iv}At each point $x\in\partial D(y)$ the directional derivatives $\Psi^j$ of that boundary point with respect to changes in the coordinates $y_j$ exist and are piecewise constant in $(x,y)$. In addition, $\eta=\sum_{j=1}^n\Psi^j\,\langle\Psi^j,\eta\rangle$ on $\partial D(y)$ where $\eta$ is the unit outward normal vector field on $\partial D(y)$.
\end{enumerate}

\comment{
\color{red} NOTE: The below conditions are the old conditions.
\begin{enumerate}
[(i)]

\item The projection of $D$ on $\rr^m$, given by $\cup_{y\in \rr^n} D(\cdot,y)$, is $\mX$, and the projection of $D$ on $\rr^n$, given by $\cup_{x\in \rr^m} D(x,\cdot)$, is $\mY$  which we assume is open. Moreover, the boundary of $D$ in $\mcal{X}\times\mcal{Y}$ can be locally parametrized as the graph $(x(y),y)'$ of a smooth function $x(y)$.  
\item For every $y\in\mY$, the domain $D(y):=D(\cdot,y)$ has a boundary $\partial D(y)$ such that the Divergence Theorem and Green's second identity hold for $D(y)$. For example, piecewise smooth boundaries suffice. 
\item At each point $x\in\partial D(y)$ the directional derivatives $\Psi^j$ of that boundary point with respect to changes in the coordinates $y_j$ exist and are piecewise constant in $(x,y)$. In addition, $\eta=\sum_{j=1}^n\Psi^j\,\langle\Psi^j,\eta\rangle$ on $\partial D(y)$ where $\eta$ is the unit outward normal vector field on $\partial D(y)$. 
\end{enumerate}
}
\color{black}
\medskip 
In the setting where the domain is not of product form, we rely on reflection in order to keep the diffusion process in the domain. When the process is started at the boundary of $D$, we do not expect (\ref{eq:condlaw}) to hold. We consider a modified condition:
\smallskip
\begin{enumerate}[]
\item  For every $h \in C_c^{\infty}(\overset{\circ}{D}) \cap \mcal{D}(\mcal{A}^Z)$ and every $(x_0,y_0)$ \textit{in the interior of} $D$, in the regime as $t \downarrow 0$, $\ev[h(Z(t))\!\mid\!Z(0)=(x_0,y_0)]$ is equal to    
\eq\label{eq:condlawinterior}
\int_{{\mathcal X}\times{\mathcal Y}} h(x_1,y_1)\tilde{R}_t((x_0,y_0),\mathrm{d}(x_1,y_1)) + o(t). 
\en
Here, the error term $o(t)$ is allowed to depend on $h$ as well as $(x_0,y_0)$.
\end{enumerate}

\color{black}
\medskip
The following regularity conditions on the link are assumed. 

\begin{asmp}\label{Asmp3}
Suppose that $L$ is an integral operator, as in Assumption \ref{asmpthm1}, mapping $C_0(\mcal{X})$ into $C_0(\mcal{Y})$ with kernel $\Lambda$ being strictly positive and continuous on $D$. As before, write $V$ for $\log \Lambda$. Moreover, assume:
\begin{enumerate}[(i)]
\item $\Lambda$ is continuously differentiable in $x$ in the interior of $D$, and $\nabla_x \Lambda$ extends to a continuous function on $D$
\item $\Lambda$ is twice continuously differentiable in $y$ on a neighborhood $U_\partial$ of the boundary of $D$ in $\mcal{X}\times\mcal{Y}$. 
\item For every $x$, $\Lambda$ can be extended to a nonnegative function $\tilde{\Lambda}$ on $\mcal{X}\times \mcal{Y}$ such that $\tilde{\Lambda}(\cdot,x) \in C^2(\mcal{Y})$ and  $\mcal{A}^Y\tilde{\Lambda}$ is continuous on $\mcal{X} \times \mcal{Y}$. Here, $\mcal{A}^Y$ should be interpreted as a differential operator.
\item For every $y \in \mcal{Y}$ and every compact set $K \subseteq \mcal{X}$, there exist $p>1$, $C<\infty$, and $M <\infty$ such that in the regime as $t \downarrow 0$,
\[
\mathbb{E}_y[\tilde{\Lambda}(Y(t),x)^p] \leq Ct^{-M}
\]
uniformly over $x \in K$.
\item For every $y\in\mcal{Y}$, the measure $\left({\mathcal A}^X\right)^*\Lambda(y,\cdot)$ integrated against each $f\in C_c^\infty(D(y))$ gives 
\eq\label{PDE_mod}
\int_{D(y)} ({\mathcal A}^Y \Lambda)\,f\,\mathrm{d}x
+\frac{1}{2}\,\int_{\partial D(y)} \Lambda\,\iprod{2f\,b + \nabla f-f\,\nabla_x V,\eta}\,\mathrm{d}\theta(x)
\en
where $\theta$ is the Lebesgue surface measure on $\partial D(y)$.
\end{enumerate}
\end{asmp}

\begin{rmk}
Condition (iv) in Assumption \ref{Asmp3} is needed to prove (\ref{eq:condlawinterior}), but conditions (i)-(iv) of Definition \ref{idef} hold without this assumption. In Section 5, we check this condition when $Y$ is a Dyson Brownian motion and $\tilde{\Lambda}(y,x)$ is the inverse of the Vandermonde determinant of $y$.
\end{rmk}
\begin{rmk}\label{PDE_simpl_rmk}
{A particular case in which the representation \eqref{PDE_mod} applies is when $b$ is continuously differentiable, $\Lambda$ is twice continuously differentiable in $x$, and \eqref{PDE} holds on $D$ with $({\mathcal A}^X)^*$ being interpreted as a differential operator. Indeed, in that case one can use the Divergence Theorem and Green's second identity to compute  
\[
\begin{split}
\int_{D(y)} \Lambda\,(\mcal{A}^X f)\,\mathrm{d}x 
=&\int_{D(y)}\Lambda\iprod{b,\nabla f}\,\mathrm{d}x+\frac{1}{2}\int_{D(y)}\Lambda\,\Delta f\,\mathrm{d}x \\
=&-\int_{D(y)} \mathrm{div}_x(\Lambda\,b)\,f\,\mathrm{d}x
+\int_{\partial D(y)} \Lambda\,f\,\iprod{b,\eta}\,\mathrm{d}\theta(x) \\
&+\frac{1}{2}\int_{D(y)} (\Delta_x\,\Lambda)\,f\,\mathrm{d}x
+\frac{1}{2}\int_{\partial D(y)} \Lambda\,\iprod{\nabla f-f\,\nabla_x V, \eta}\,\mathrm{d}\theta(x) \\
=&\,\int_{D(y)} ((\mcal{A}^X)^*\Lambda)\,f\,\mathrm{d}x
+\frac{1}{2}\int_{\partial D(y)} \Lambda\,\iprod{2f\,b + \nabla f-f\,\nabla_x V,\eta}\,\mathrm{d}\theta(x) \\
=&\,\int_{D(y)} (\mcal{A}^Y \Lambda)\,f\,\mathrm{d}x
+\frac{1}{2}\int_{\partial D(y)} \Lambda\,\iprod{2f\,b + \nabla f-f\,\nabla_x V,\eta}\,\mathrm{d}\theta(x).
\end{split}
\]}
\end{rmk}

\begin{thm}\label{thm:reflect} {Let $Z=(Z_1,Z_2)$ be a diffusion process on $D$ with generator given by \eqref{Zgen} and boundary conditions of ${\mathcal A}^X$ on $\partial\mcal{X}\times\mcal{Y}$. Assume that  
${\mathcal A}^Y$ has no boundary conditions
 \color{black} and the normal reflection of the $Z_2$-components on $\partial D(Z_1(\cdot),\cdot)$. Suppose that the associated stochastic differential equation with reflection is well-posed and its solution is a Feller-Markov process with $C_c^{\infty}(D) \cap \mcal{D}(\mcal{A}^Z)$ being a core for the domain of $Z$.}
{Finally, suppose that  
\eq\label{LambdaNeumann}
\Lambda\,\langle b,\eta\rangle-\langle\nabla_x\Lambda,\eta\rangle
=\sum_{j=1}^m \iprod{ \Psi^j,\eta} \big(\gamma_j\,\Lambda+\partial_{y_j}\Lambda\big) 
\;\;\;\mathrm{on}\;\;\partial D(y)\;\;\mathrm{for\;each}\;\;y\in\mY.
\en
Then $Z=Y\iprod{L}X$ and $Z$ satisfies \eqref{eq:condlawinterior}, provided that $Z(0)$ is as in condition \eqref{i1} of Definition \ref{idef}.}  
\end{thm}

\begin{rmk}
{The normal reflection of the $y$-components of $Z$ on $\partial D(Z_1(\cdot),\cdot)$ can be equivalently phrased as a Neumann boundary condition with respect to the vector field  
\eq\label{thm3Neumann}
\sum_{j=1}^n \langle \Psi^j,\eta\rangle\,\partial_{y_j}\quad\mathrm{on}\quad\partial D(y)  
\en
for the generator of $Z$. Indeed, parametrizing $\partial D$ locally as the graph $(x(y,\xi),y)'$ of a smooth function $x(y,\xi)$ and writing $\eta_i$ for the components of $\eta$ one computes
\[
\sum_{j=1}^n \langle \Psi^j,\eta\rangle\,\partial_{y_j}
=\sum_{j=1}^n \sum_{i=1}^m \partial_{y_j} x_i(y,\xi)\,\eta_i\,\partial_{y_j}
=\Big\langle \sum_{i=1}^m \eta_i\,\nabla x_i(y,\xi),\nabla_y\Big\rangle.  
\]   
Moreover, letting $\hat{\eta}$ be the unit outward normal vector field on $\partial D(x,\cdot)$ one finds locally a constant $c>0$ such that $\eta+c\,\hat{\eta}$ is an outward normal vector field on $\partial D$ and, in particular, $\sum_{i=1}^m \eta_i\,\nabla x_i(y,\xi)+c\,\hat{\eta}=0$ (every component of the latter vector being the inner product of the normal vector $\eta+c\,\hat{\eta}$ with a vector tangent to $\partial D$). Hence, a Neumann boundary condition with respect to $\sum_{j=1}^n \langle \Psi^j,\eta\rangle\,\partial_{y_j}=\iprod{-c\,\hat{\eta},\nabla_y}$ corresponds to a normal reflection of the $y$-components of $Z$ on $\partial D(Z_1(\cdot),\cdot)$ as claimed.}
\end{rmk}

\smallskip

\noindent\textbf{Proof of Theorem \ref{thm:reflect}.} The proof has the same structure as that of Theorem \ref{main1}. Steps 1 and 2 remain the same, and we move on to Step 3. Define the functions $u(t)$, $v(t)$ as in \eqref{whatisu}, \eqref{eq:vtxy} for some $h\in \mcal{D}(\mcal{A}^Z)$. {The representation \eqref{whatisv} for $u(t)$ now takes the form 
\eq
u(t)(y)=\int_{D(y)} \Lambda(y,x)\,v(t)(x,y)\,\mathrm{d}x
\en
where, for every $t\ge0$, $v(t)$ belongs to the domain of the generator ${\mathcal A}^Z$ specified in the theorem, and $\frac{\mathrm{d}}{\mathrm{d}t}\,v(t)={\mathcal A}^Z\,v(t)$, $t\ge0$. By assumption, for each $t$, there exists a sequence $v_l(t) \in C_c^{\infty}(D) \cap \mcal{D}(\mcal{A}^Z)$ such that $v_l(t)$ converges uniformly to $v(t)$ and $\mcal{A}^Zv_l(t)$ converges uniformly to $\mcal{A}^Z v(t)$.  This allows us to compute 
\eq\label{eq:divappl}
\begin{split}
\frac{\mathrm{d}}{\mathrm{d}t}\,u(t) = \int_{D(y)} \Lambda\,\Big(\frac{\mathrm{d}}{\mathrm{d}t}\,v(t)\Big)\,\mathrm{d}x 
= &\lim_{l\to\infty} \, \int_{D(y)} \Lambda\,\big(\mcal{A}^X + \mcal{A}^Y + (\nabla_y V)'\,\nabla_y\big)\,v_l(t)\,\mathrm{d}x \\
= &\lim_{l\to\infty} \bigg(\!\int_{D(y)} \big(({\mathcal A}^Y \Lambda)+\Lambda\,{\mathcal A}^Y + \Lambda\,(\nabla_y V)'\,\nabla_y \big)
\,v_l(t)\,\mathrm{d}x \\
&\quad\;\;+\frac{1}{2}\int_{\partial D(y)} \Lambda\,\iprod{2v_l(t)\,b + \nabla_x v_l(t)-v_l(t)\,\nabla_x V,\eta}\,\mathrm{d}\theta(x)\!\bigg),
\end{split} 
\en
where the second identity reveals that the limit is uniform in $y$, and the third identity has been obtained using the representation \eqref{PDE_mod}.}

\medskip

 Next, we pick a sequence $\Lambda_q$, $q\in\nn$ of $C_c^\infty(D)$ functions such that the convergences $\Lambda_q\to\Lambda$, $\nabla_y\Lambda_q\to\nabla_y\Lambda$, $\nabla_x \Lambda_q \to \nabla_x \Lambda$, and ${\mathcal A}^Y \Lambda_q\to{\mathcal A}^Y \Lambda$ hold uniformly on compact subsets of $D$. Such a sequence can be constructed by first decomposing $\Lambda$ into a finite sum according to a suitable partition of unity on $D$. For elements of the open cover in the interior of $D$, one may convolve the summand with a smooth kernel. For elements of the open cover near the boundary, one may push the points to the interior on a scale $\epsilon$, then convolve with a smoothing kernel on a scale of $\epsilon^2$ similar to \cite[Section 5.3.3, Theorem 3]{Ev}. For every fixed $l,q\in\nn$, one can now use the multidimensional Leibniz rule and the Divergence Theorem to compute
\begin{eqnarray*}
&&\partial_{y_j} \int_{D(y)} \Lambda_q\,v_l(t)\,\mathrm{d}x=\int_{D(y)} \mathrm{div}_x(\Lambda_q\,v_l(t)\,\Psi^j)
+\partial_{y_j}(\Lambda_q\,v_l(t))\,\mathrm{d}x, \\
&&\partial_{y_jy_j}  \int_{D(y)} \Lambda_q\,v_l(t)\,\mathrm{d}x 
= \int_{D(y)} \Big(\mathrm{div}_x(\mathrm{div}_x(\Lambda_q\,v_l(t)\,\Psi^j)\,\Psi^j) 
+ \partial_{y_j} \big(\mathrm{div}_x(\Lambda_q\,v_l(t)\,\Psi^j)\big) \\
&&\quad\quad\quad\quad\quad\quad\quad\quad\quad\quad\quad\quad\quad\;\;
+ \mathrm{div}_x\big(\partial_{y_j}(\Lambda_q\,v_l(t))\,\Psi^j\big)
+\partial_{y_jy_j}(\Lambda_q\,v_l(t))\Big)\,\mathrm{d}x.  
\end{eqnarray*}
Therefore, noting that Itô's formula and \cite[Proposition VII.1.7]{RY}  imply that the functions $\Lambda_q v_l(t)$ and  $\int_{D(y)}\Lambda_q v_l(t)\, \mathrm{d}x$ are in $\mcal{D}(\mcal{A}^Y )$, we have
\begin{eqnarray*}
{\mathcal A}^Y \int_{D(y)} \Lambda_q\,v_l(t)\,\mathrm{d}x
=\int_{D(y)} {\mathcal A}^Y (\Lambda_q\,v_l(t))\,\mathrm{d}x
+\sum_{j=1}^n \int_{D(y)} \gamma_j\,\mathrm{div}_x(\Lambda_q\,v_l(t)\,\Psi^j)\,\mathrm{d}x 
\quad\quad\quad\quad\quad \\
+\frac{1}{2}\,\Big(\mathrm{div}_x(\mathrm{div}_x(\Lambda_q\,v_l(t)\,\Psi^j)\,\Psi^j) 
+ \partial_{y_j} \big(\mathrm{div}_x(\Lambda_q\,v_l(t)\,\Psi^j)\big)
+ \mathrm{div}_x\big(\partial_{y_j}(\Lambda_q\,v_l(t))\,\Psi^j\big)\Big)\,\mathrm{d}x.
\end{eqnarray*}
In view of the Divergence Theorem, the latter expression can be rewritten as 
\eq\label{AYu}
\begin{split}
\int_{D(y)} {\mathcal A}^Y (\Lambda_q\,v_l(t))\,\mathrm{d}x
+\sum_{j=1}^n \int_{\partial D(y)}\gamma_j\,\Lambda_q\,v_l(t)\,\langle \Psi^j,\eta\rangle
+\frac{1}{2}\,\mathrm{div}_x(\Lambda_q\,v_l(t)\,\Psi^j)\langle\Psi^j,\eta\rangle 
\quad\quad\quad\;\;\;\\
+\frac{1}{2}\langle \partial_{y_j}(\Lambda_q\,v_l(t)\,\Psi^j),\eta\rangle
+\frac{1}{2}\partial_{y_j}(\Lambda_q\,v_l(t))\langle \Psi^j,\eta\rangle\,\mathrm{d}\theta(x).
\end{split}
\en

\smallskip

{Note further that ${\mathcal A}^Y (\Lambda_q\,v_l(t))$ is given by the product rule \eqref{prod_rule}, and therefore the expression in \eqref{AYu} converges in the limit $q\to\infty$ uniformly to 
\eq\label{AYu'}
\begin{split}
& \int_{D(y)} ({\mathcal A}^Y \Lambda)\,v_l(t)+(\nabla_y\Lambda)'\,\nabla_y v_l(t)+\Lambda\,({\mathcal A}^Y v_l(t))\,\mathrm{d}x \\
& +\sum_{j=1}^n \int_{\partial D(y)}\Big(\gamma_j\,\Lambda\,v_l(t)\,\langle \Psi^j,\eta\rangle
+\frac{1}{2}\,\mathrm{div}_x(\Lambda\,v_l(t)\,\Psi^j)\langle\Psi^j,\eta\rangle 
+\frac{1}{2}\langle \partial_{y_j}(\Lambda\,v_l(t)\,\Psi^j),\eta\rangle \\
& \qquad\qquad\qquad+\frac{1}{2}\partial_{y_j}(\Lambda\,v_l(t))\langle \Psi^j,\eta\rangle\Big)\,\mathrm{d}\theta(x).
\end{split}
\en
Since the operator ${\mathcal A}^Y$ is closed (\cite[Lemma 17.8]{Ka}), the latter can be further identified as ${\mathcal A}^Y \int_{D(y)} \Lambda\,v_l(t)\,\mathrm{d}x$.} We proceed by using the fact that each $\Psi^j$ is piecewise constant, $\eta=\sum_{j=1}^n\Psi^j\,\langle\Psi^j,\eta\rangle$, \eqref{LambdaNeumann}, and the Neumann boundary condition with respect to the vector field of \eqref{thm3Neumann} satisfied by $v_l(t)$ to simplify the boundary integrand in \eqref{AYu'}. For the terms of the boundary integrand containing $v_l(t)$ we compute  
\[
\begin{split}
& v_l(t)\,\sum_{j=1}^n \Big( \Lambda \gamma_j \iprod{\Psi^j,\eta}  + \frac{1}{2}\iprod{\nabla_x\Lambda, \Psi^j} \iprod{\Psi^j, \eta} 
+ \partial_{y_j} \Lambda \iprod{\Psi^j, \eta}\Big)\\
&= v_l(t)\,\Big( \sum_{j=1}^n \iprod{\Psi^j, \eta}\left( \gamma_j \Lambda  + \partial_{y_j} \Lambda \right) 
+ \frac{1}{2} \iprod{\nabla_x \Lambda, \eta} \Big)
= v_l(t)\,\Lambda \iprod{b , \eta} -\frac{1}{2}\,v_l(t)\,\iprod{\nabla_x \Lambda,\eta},
\end{split}
\]
whereas for the remaining terms of the boundary integrand we get
\[
\sum_{j=1}^n \Big( \frac{1}{2}\,\Lambda \iprod{\nabla_x v_l(t), \Psi^j} \iprod{\Psi^j, \eta} 
+ \Lambda\,\partial_{y_j} v_l(t) \iprod{\Psi^j, \eta} \Big)= \frac{1}{2}\,\Lambda \iprod{\nabla_x v_l(t),\eta}.
\]
{Plugging this into \eqref{AYu'} and comparing the result with \eqref{eq:divappl} we obtain
\[
\frac{\mathrm{d}}{\mathrm{d}t}\,u(t)=\lim_{l\to\infty}\;{\mathcal A}^Y \int_{D(y)} \Lambda\,v_l(t)\,\mathrm{d}x,
\]
where the limit is uniform in $y$ as pointed out after \eqref{eq:divappl}. Another application of the closedness of ${\mathcal A}^Y$ yields $\frac{\mathrm{d}}{\mathrm{d}t}\,u(t)={\mathcal A}^Y u(t)$, completing Step 3.} The arguments in Steps 4 through 7 can be repeated word by word, only replacing the references to Step 3 in the proof of Theorem 2 by those to Step 3 herein.

\medskip

\comment{We proceed to Step 3. {We note first that Step 2 and Proposition II.6.2 in \cite{EN} imply $\ev\big[f(Z_2(t))\,|\,Z_2(0)=y]=(Q_t f)(y)$ for all $y\in\mcal{Y}$ and $f\in C_c^\infty(\mcal{Y})$. Moreover, a straightforward approximation argument allows to extend this identity to all $f$ in the domain of ${\mathcal A}^Y$. On the other hand, an approximation argument as in the previous step and the locality of ${\mathcal A}^Y$ reveal that, for any $g\in C_c^\infty\left(\mcal{X}\right)$, \color{red} the function $\int_{D(y)} \Lambda\,g\,\mathrm{d}x$ belongs to the domain of ${\mathcal A}^Y$,\color{black} and ${\mathcal A}^Y \int_{D(y)} \Lambda\,g\,\mathrm{d}x$ is given by \eqref{AYu'} with $g$ replacing $v_l(t)$. Hence, to prove the desired commutativity relation
\eq\label{thm3comm}
\ev\bigg[\int_{D(Z_2(t))} \Lambda(Z_2(t),x)\,g(x)\,\mathrm{d}x\,\Big|\,Z_2(0)=y\bigg]
=\int_{D(y)} \Lambda(y,x)\,\ev\big[g(Z_1(t))\,|\,Z_1(0)=x\big]\,\mathrm{d}x
\en
for functions $g\in C_c^\infty\left(\mcal{X}\right)$, it suffices to verify that the right-hand side is a solution of $\frac{\mathrm{d}}{\mathrm{d}t} u(t)={\mathcal A}^Y u(t)$ (the left-hand side is such a solution by Theorem 17.6 in \cite{Ka} and the solution is unique by Proposition II.6.2 in \cite{EN}).} 

\medskip

{For the purpose of the latter verification write $v(t)$ for $\ev\big[g(Z_1(t))\,|\,Z_1(0)=x\big]$ and note that $\frac{\mathrm{d}}{\mathrm{d}t}\,v(t)={\mathcal A}^X v(t)$ by Step 1. Picking a sequence $v_l(t)$, $l\in\nn$ of $C_c^\infty(\mcal{X})$ functions in the domain of ${\mathcal A}^X$ converging uniformly to $v(t)$ such that ${\mathcal A}^X v_l(t)\to {\mathcal A}^X v(t)$ uniformly as well and making the same calculation as in \eqref{eq:divappl} one determines the time derivative of the right-hand side of \eqref{thm3comm} as
\eq\label{thm3step3lim}
\lim_{l\to\infty} \bigg(\int_{D(y)} ({\mathcal A}^Y\Lambda)\,v_l(t)\,\mathrm{d}x
+\frac{1}{2}\int_{\partial D(y)} \Lambda\,\iprod{2 v_l(t)\,b + \nabla v_l(t)-v_l(t)\,\nabla_x V,\eta}\,\mathrm{d}\theta(x)\bigg)
\en
where the limit is uniform in $y$. Moreover, proceeding exactly as in the derivation of \eqref{AYu'} one arrives at 
\[
\begin{split}
& \int_{D(y)} ({\mathcal A}^Y \Lambda)\,v_l(t)\,\mathrm{d}x 
+\sum_{j=1}^m \int_{\partial D(y)}\Big(\gamma_j\,\Lambda\,v_l(t)\,\langle \Psi^j,\eta\rangle
+\frac{1}{2}\,\mathrm{div}_x(\Lambda\,v_l(t)\,\Psi^j)\langle\Psi^j,\eta\rangle \\
& \qquad\qquad\qquad\qquad\qquad\qquad\qquad\quad +(\partial_{y_j}\Lambda)\,v_l(t)\langle \Psi^j,\eta\rangle\Big)\,\mathrm{d}\theta(x).
\end{split}
\]
This expression can be identified as ${\mathcal A}^Y \int_{D(y)} \Lambda\,v_l(t)\,\mathrm{d}x$ thanks to the locality and closedness of ${\mathcal A}^Y$ on the one hand and simplified to the expression inside the limit in \eqref{thm3step3lim} using that the $\Psi^j$ are piecewise constant, $\eta=\sum_{j=1}^n\Psi^j\,\langle\Psi^j,\eta\rangle$, and \eqref{LambdaNeumann} on the other hand. All in all,
\[
\frac{\mathrm{d}}{\mathrm{d}t} \int_{D(y)} \Lambda\,v(t)\,\mathrm{d}x
=\lim_{l\to\infty} \; {\mathcal A}^Y \int_{D(y)} \Lambda\,v_l(t)\,\mathrm{d}x
= {\mathcal A}^Y \int_{D(y)} \Lambda\,v(t)\,\mathrm{d}x
\]
thanks to the closedness of ${\mathcal A}^Y$. This finishes Step 3.} Moreover, the arguments in Steps 4 through 6 can be repeated word by word, only replacing the references to Step 2 in the proof of Theorem \ref{main1} by references to Step 2 in the current proof. }

\noindent\textbf{Step 8.} We now turn to the proof of condition \eqref{eq:condlawinterior}. Fix $(x_0,y_0)$ in the interior of $D$. We introduce two compact sets $K_{x_0}, K_{y_0}$ with nonempty interior,  $x_0 \in K_{x_0}$, $y_0 \in K_{y_0}$, and  $K_{x_0} \times K_{y_0} \subseteq \overset{\circ}{D}$. Fix a function $h \in C_c^{\infty}(D)$ satisfying the boundary conditions introduced in the statement of the theorem. As in Step 1 of the proof of Theorem 2 and using the same notation, we may restrict the integral over the $x_1$ variable in $\tilde{R}_t h$ to $K_{x_0}$. 

\smallskip
First, note that $\Lambda h = \tilde{\Lambda} h$, and so
\[
Q_t\big(\Lambda(\cdot,x_1)h(x_1,\cdot)\big)(y_0) = \Lambda(y_0,x_1)h(x_1,y_0) + t \mcal{A}^Y\big(\Lambda(\cdot,x_1)h(x_1,\cdot)\big)(y_0) + t \epsilon_1(x_1,y_0,t)
\]
where $\epsilon_1(x_1,y_0,t)$ is $o(1)$ uniformly in $x_1 \in K_{x_0}$ due to the uniform continuity and boundedness of $\mcal{A}^Y\big(\Lambda h\big)$.
Introduce an open neighborhood $U$ of $y_0$ compactly contained in $K_{y_0}$. Let $\phi$ be a smooth function from $\mcal{Y}$ to $[0,1]$ that is $1$ inside $U$ and $0$ outside $K_{y_0}$. Now, since $\tilde{\Lambda}$ is an extension of $\Lambda$, Hölder's inequality implies that
\begin{equation}\label{probest}
\begin{split}
\big|\big(Q_t\tilde{\Lambda}(\cdot,x_1)\big)(y_0) - \big(Q_t\Lambda(\cdot,x_1)\big)(y_0)\big| &\leq \mathbb{E}_{y_0}\big[\tilde{\Lambda}(Y(t),x_1)(1-\phi(Y(t))\big]\\
&\leq Ct^{-\frac{M}{p}} \mathbb{P}_{y_0}(Y(t) \not \in U)^{\frac{1}{q}},
\end{split}
\end{equation}
where $C, M$, and $p$ come from Assumption 3(iv) and $q^{-1}=1-p^{-1}$. Due to the local boundedness of the drift of $Y$, the latter probability decays exponentially in $\frac{1}{t}$ as $t \downarrow 0$. This ensures that the right-hand side of (\ref{probest}) is $o(t)$ uniformly over $x_1 \in K_{x_0}$. Likewise,
\begin{equation}\label{est2}
\big(Q_t\tilde{\Lambda}(\cdot,x_1)\big)(y_0) = \big(Q_t\tilde{\Lambda}(\cdot,x_1)\phi(\cdot)\big)(y_0) + o(t),
\end{equation}
where, again, the $o(t)$ is uniform over $x_1 \in K_{x_0}$. Now, $\tilde{\Lambda}(\cdot,x_1)\phi(\cdot)$ is a uniformly bounded, $C^2$-function with compact support, and so
\begin{equation}\label{est3}
Q_t\big(\tilde{\Lambda}(\cdot,x_1)\phi(\cdot)\big)(y_0) = \Lambda(y_0,x_1) + t \mcal{A}^Y\Lambda(\cdot,x_1)(y_0) + o(t),
\end{equation}
with $o(t)$ uniform over $x_1 \in K_{x_0}$. Putting equations (\ref{probest}), (\ref{est2}), and (\ref{est3}) together, we find that 
\[
\big(Q_t\Lambda(\cdot,x_1)\big)(y_0) = \Lambda(y_0,x_1) + t \mcal{A}^Y\Lambda(\cdot,x_1)(y_0) + o(t).
\]
The rest of the proof is exactly the same as Step 1 in the proof of Theorem 2.
\color{black}
\ep

\medskip

\comment{

{As in Step 7 in the proof of Theorem \ref{main1}, it remains to show that the transition operators of $Z$ are given by  
\eq\label{thm3Z_transition}
\int_{\mcal{X}} P_t(x_0,\mathrm{d}x_1)\,\frac{\int_{D(x_1,\cdot)} Q_t(y_0,\mathrm{d}y_1)\,\Lambda(y_1,x_1)\,f(x_1,y_1)}
{\int_{D(x_1,\cdot)} Q_t(y_0,\mathrm{d}y_1)\,\Lambda(y_1,x_1)},\quad t>0
\en
on $C_0(D)$ functions $f$. One can see as there that the transition operators of \eqref{thm3Z_transition} satisfy the Chapman-Kolmogorov equation, thereby defining a Markov process $\tilde{Z}$, and that the $x$-components ($y$-components resp.) of $\tilde{Z}$ evolve according to the semigroup $(P_t)$ ($(Q_t)$ resp.), so that, in particular, $\tilde{Z}$ has continuous paths. Therefore it suffices to prove that $\tilde{Z}$ has the same law as $Z$. Our strategy of establishing this is to check that $\tilde{Z}$ solves the submartingale problem associated with $Z$. To this end, we fix a $C_c^2(D)$ function $f$ such that the intersection of its support with $\partial D$ is smooth, and the derivatives of $f$ with respect to the vector fields defining the boundary conditions of $Z$ are nonnegative there. Moreover, we let $K$ be a compact subset of $D$ containing the support of $f$ and such that the boundary of $K$ in $D$ is of positive distance from the support of $f$ if such a set exists and let $K$ be the support of $f$ otherwise (that is, if the support of $f$ is all of $D$). Note that there exists an $\epsilon>0$ such that every point in $D$ is either in $K$ and of distance more than $\epsilon$ from the boundary of $D$ in $\mcal{X}\times\mcal{Y}$ (``case 1''), or in $K$ and the ball of radius $\epsilon$ around the point is contained in $U_\partial$ (see Assumption \ref{Asmp3} (ii) for the definition of the latter) and does not intersect $(\mcal{X}\times\partial\mcal{Y})\cup(\partial\mcal{X}\times\mcal{Y})$ (``case 2''), or not in $K$ (``case 3''). We proceed by studying the $t\downarrow0$ asymptotics of the fraction in \eqref{thm3Z_transition} in these three cases.} 

\medskip
 
\noindent{\textbf{Case 1.} Note first that the probability of the process $Y$ exiting the ball of radius $\epsilon$ around $y_0$ decays exponentially as $t\downarrow0$. Indeed, if $y_0$ is of distance of more than $\epsilon$ from $\partial\mcal{Y}$, then this is a direct consequence of the boundedness of the drift and diffusion coefficients of $Y$ on such a ball; and if $y_0$ is of distance less than $\epsilon$ from $\partial\mcal{Y}$, then one can apply a diffeomorphism to the $\epsilon$-neighborhood of $y_0$ under which the boundary vector field $U_2$ becomes constant and subsequently use the local boundedness of the drift and diffusion coefficients of the process resulting from $Y$ under such a diffeomorphism. It now follows from the boundedness of $\Lambda f$ that the $t\downarrow0$ asymptotics of $\int_{D(x_1,\cdot)} Q_t(y_0,\mathrm{d}y_1)\,\Lambda(y_1,x_1)\,f(x_1,y_1)$ is captured up to an exponentially small error if one replaces $Q_t$ by the semigroup of the process $Y$ stopped upon the exit time $\tau_\epsilon$ from the ball of radius $\epsilon$ around $y_0$. Approximating $\Lambda(\cdot,x_1)$ by a sequence of $C^\infty_c(D(x_1,\cdot))$ functions $\Lambda_q(\cdot,x_1)$ such that the convergences $\Lambda_q(\cdot,x_1)\to\Lambda(\cdot,x_1)$, $\nabla_y \Lambda_q(\cdot,x_1)\to\nabla_y\Lambda(\cdot,x_1)$, and ${\mathcal A}^Y \Lambda_q(\cdot,x_1)\to{\mathcal A}^Y \Lambda(\cdot,x_1)$ hold uniformly on the ball in consideration, applying It\^o's formula to $\Lambda_q(Y(t\wedge\tau_\epsilon),x_1)\,f(x_1,Y(t\wedge\tau_\epsilon))$, taking the expectation, and passing to the limit $q\to\infty$ we obtain  
\[
\begin{split}
& \ev\big[\Lambda(Y(t\wedge\tau_\epsilon),x_1)\,f(x_1,Y(t\wedge\tau_\epsilon))\big]
\\
& \ge \Lambda(y_0,x_1)\,f(x_1,y_0) + \ev\bigg[\int_0^{t\wedge\tau_\epsilon} 
\big(({\mathcal A}^Y\Lambda) f+(\nabla_y \Lambda)'\,\nabla_y f+\Lambda\,{\mathcal A}^Y f\big)(Y(s),x_1)\,\mathrm{d}s\bigg].
\end{split}
\]
Here the inequality results from dropping the boundary terms and recalling the boundary conditions satisfied by $\Lambda$ and $f$ on $\mcal{X}\times\partial\mcal{Y}$. At this point, we can rely on the continuity of the function $({\mathcal A}^Y\Lambda) f+(\nabla_y \Lambda)'\,\nabla_y f+\Lambda\,{\mathcal A}^Y f$ on the ball in consideration and the exponentially small probability of the event $\{\tau_\epsilon<t\}$ to conclude 
\[
\begin{split}
& \int_{D(x_1,\cdot)} Q_t(y_0,\mathrm{d}y_1)\,\Lambda(y_1,x_1)\,f(x_1,y_1) \\
& \ge(\Lambda f)(y_0,x_1)+t\,\big(({\mathcal A}^Y\Lambda) f+(\nabla_y \Lambda)'\,\nabla_y f+\Lambda\,{\mathcal A}^Y f\big)(y_0,x_1)+o(t),
\quad t\downarrow0.
\end{split}
\]
Moreover, the error term is uniform in the point $(x_1,y_0)$, since all estimates did not depend on the latter.} 

\medskip

{Next, extend $\Lambda$ to $\mcal{X}\times\mcal{Y}$ by setting it to zero outside of $D$, and pick functions $\underline{\Lambda}(\cdot,x_1)\le\Lambda(\cdot,x_1)\le\overline{\Lambda}(\cdot,x_1)$ on $\mcal{Y}$ in the domain of ${\mathcal A}^Y$ in $C_0(\mcal{Y})$ coinciding with $\Lambda(\cdot,x_1)$ on a ball of radius $\frac{\epsilon}{2}$ around $y_0$ and such that ${\mathcal A}^Y \underline{\Lambda}(\cdot,x_1)$, ${\mathcal A}^Y \overline{\Lambda}(\cdot,x_1)$ are equicontinuous on that ball and uniformly bounded on $\mcal{Y}$ as $(x_1,y_0)$ varies. Such a choice is possible due to the continuity and boundedness assumptions on ${\mathcal A}^Y\Lambda$. Invoking the Kolmogorov forward equation one further computes 
\[
\begin{split}
& \underline{\Lambda}(y_0,x_1)+t\,({\mathcal A}^Y \underline{\Lambda})(y_0,x_1)
+\int_0^t \Big(Q_s\big({\mathcal A}^Y \underline{\Lambda}(\cdot,x_1)\big)(y_0)-{\mathcal A}^Y \underline{\Lambda}(y_0,x_1)\Big)\,\mathrm{d}s \\
& =\int_{D(x_1,\cdot)} \!\!Q_t(y_0,\mathrm{d}y_1)\,\underline{\Lambda}(y_1,x_1) 
\le\int_{D(x_1,\cdot)} \!\!Q_t(y_0,\mathrm{d}y_1)\,\Lambda(y_1,x_1)\le \int_{D(x_1,\cdot)} \!\!Q_t(y_0,\mathrm{d}y_1)\,\overline{\Lambda}(y_1,x_1) \\
&\qquad\qquad\quad\;\;
=\overline{\Lambda}(y_0,x_1)+t\,({\mathcal A}^Y \overline{\Lambda})(y_0,x_1)
+\int_0^t \Big(Q_s\big({\mathcal A}^Y \overline{\Lambda}(\cdot,x_1)\big)(y_0)-{\mathcal A}^Y \overline{\Lambda}(y_0,x_1)\Big)\,\mathrm{d}s. 
\end{split}
\]
Thanks to the properties of $\underline{\Lambda}$, $\overline{\Lambda}$ this yields the asymptotics 
\[
\int_{D(x_1,\cdot)} Q_t(y_0,\mathrm{d}y_1)\,\Lambda(y_1,x_1)=\Lambda(y_0,x_1)+t\,{\mathcal A}^Y\Lambda(y_0,x_1)+o(t),
\quad t\downarrow0
\]
with an error term uniform in $(x_1,y_0)$. Inserting the asymptotic lower bound on the numerator and the asymptotics of the denominator into the fraction in \eqref{thm3Z_transition} and expanding the result as in Step 1 in the proof of Theorem \ref{main2} one ends up with
\[
\frac{\int_{D(x_1,\cdot)} Q_t(y_0,\mathrm{d}y_1)\,\Lambda(y_1,x_1)\,f(x_1,y_1)}{\int_{D(x_1,\cdot)} Q_t(y_0,\mathrm{d}y_1)\,\Lambda(y_1,x_1)}
\ge f(x_1,y_0)+t\,({\mathcal A}^Y f+(\nabla_y V)'\,\nabla_y f)(x_1,y_0)+o(t).
\]
The error term is uniform in $(x_1,y_0)$ due to the uniformity of the preceding error terms and the uniform positivity of $\Lambda$ on $K$.} 

\medskip

\noindent{\textbf{Case 2.} As in case 1, one can replace $\int_{D(x_1,\cdot)} Q_t(y_0,\mathrm{d}y_1)\,\Lambda(y_1,x_1)\,f(x_1,y_1)$ by the expectation of $\Lambda(Y(t\wedge\tau_\epsilon),x_1)\,f(x_1,Y(t\wedge\tau_\epsilon))$ (where $\tau_\epsilon$ is the first exit time of $Y$ from the ball of radius $\epsilon$ around $y_0$), making only a uniformly exponentially small error as $t\downarrow0$. Moreover, on the intersection of that ball with $D(x_1,\cdot)$ one has the Taylor expansion
\[
\begin{split}
(\Lambda f)(y_1,x_1)= & (\Lambda f)(y_0,x_1)+(\nabla_y(\Lambda f)(y_0,x_1))'\,(y_1-y_0) \\
& +\frac{1}{2}\,(y_1-y_0)'\,\mathrm{Hess}_y(\Lambda f)(y_0,x_1)\,(y_1-y_0)+o(|y_1-y_0|^2)
\end{split}
\]
where $\mathrm{Hess}_y$ stands for the Hessian with respect to $y$, and the error term can be controlled uniformly by the modulus of continuity of $\mathrm{Hess}_y(\Lambda f)$ on $K\cap U_\partial$. Inserting this into the expectation and rearranging terms one ends up with the expectation of
\eq\label{Gaussian}
\begin{split}
& \mathbf{1}_{\{Y(t\wedge\tau_\epsilon)\in D(x_1,\cdot)\}}\Big(f(x_1,y_0)\Big(\Lambda(y_0,x_1)+\nabla_y\Lambda(y_0,x_1)'(Y(t\wedge\tau_\epsilon)-y_0) \\ 
& +\frac{1}{2}\,(Y(t\wedge\tau_\epsilon)-y_0)'
(\mathrm{Hess}_y\Lambda)(y_0,x_1)(Y(t\wedge\tau_\epsilon)-y_0)\Big) 
+(\Lambda\,\nabla_y f)(y_0,x_1)'(Y(t\wedge\tau_\epsilon)-y_0) \\
& +\frac{1}{2}\,(Y(t\wedge\tau_\epsilon)-y_0)'
\big(2\,(\nabla_y\Lambda)(\nabla_y f)'+\Lambda\,\mathrm{Hess}_y f\big)(y_0,x_1)(Y(t\wedge\tau_\epsilon)-y_0)\Big) \\
& +o\big(|Y(t\wedge\tau_\epsilon)-y_0|^2\big).
\end{split}
\en} 

{To study $\int_{D(x_1,\cdot)} Q_t(y_0,\mathrm{d}y_1)\,\Lambda(y_1,x_1)$ we let $h$ be a $[0,1]$-valued $C_c^\infty(\mcal{Y})$ function constantly equal to $1$ on an open neighborhood of the boundary of $D(x_1,\cdot)$ in $\mcal{Y}$ that contains $y_0$ and constantly equal to $0$ on $\{y_1\in D(x_1,\cdot):(x_1,y_1)\notin U_\partial\}$. Then the Kolmogorov forward equation shows
\[
\int_{D(x_1,\cdot)} Q_t(y_0,\mathrm{d}y_1)\,\Lambda(y_1,x_1)\,(1-h(y_1))
=\int_0^t Q_s\big({\mathcal A}^Y(\Lambda(\cdot,x_1)\,(1-h))\big)(y_0)\,\mathrm{d}s=o(t),\;t\downarrow0.
\]
The uniform in $(x_1,y_0)$ order $o(t)$ of the latter term is a consequence of a product rule as in \eqref{prod_rule}, the boundedness assumption on ${\mathcal A}^Y \Lambda$, and the continuity assumptions on ${\mathcal A}^Y\Lambda$, $\nabla_y\Lambda$, and $\Lambda$. Moreover, we analyze the $t\downarrow0$ asymptotics of $\int_{D(x_1,\cdot)} Q_t(y_0,\mathrm{d}y_1)\,\Lambda(y_1,x_1)\,h(y_1)$ as for $\int_{D(x_1,\cdot)} Q_t(y_0,\mathrm{d}y_1)\,\Lambda(y_1,x_1)\,f(x_1,y_1)$ and end up with the expectation of a random variable as in \eqref{Gaussian}, with $f(x_1,y_0)$ replaced by $h(y_0)=1$, $(\nabla_y\,f)(x_1,y_0)$ by $(\nabla\,h)(y_0)=0$, and $(\mathrm{Hess}_y\,f)(x_1,y_0)$ by $(\mathrm{Hess}\,h)(y_0)=0$.}

\medskip

{Next, we put the asymptotics of this case together introducing the Brownian motion $W(t\wedge\tau_\epsilon):=Y(t\wedge\tau_\epsilon)-y_0-\int_0^{t\wedge\tau_\epsilon} \gamma(Y(s))\,\mathrm{d}s$, $t\ge0$ and noting
\eq\label{bdry_arg}
(\Lambda\,\nabla_y\,f)(x_1,y_0)\,\ev\big[W(t\wedge\tau_\epsilon)\,\mathbf{1}_{\{y_0+\int_0^{t\wedge\tau_\epsilon} \gamma(Y(s))\,\mathrm{d}s+W(t\wedge\tau_\epsilon)\in D(x_1,\cdot)\}}\big]\ge -O(t),
\quad t\downarrow0
\en
in view of the boundary condition on $f$ along the boundary of $D$ in $\mcal{X}\times\mcal{Y}$ and the Lipschitz property of $(\Lambda\,\nabla_y\,f)(x_1,\cdot)$ on $D(x_1,\cdot)$. It follows that the fraction in \eqref{thm3Z_transition} admits the asymptotic lower bound 
\[
\begin{split}
& f(x_1,y_0)+\frac{(\Lambda\,\nabla_y\,f)(x_1,y_0)\,\ev\big[\int_0^{t\wedge\tau_\epsilon} \gamma(Y(s))\,\mathrm{d}s
\,\mathbf{1}_{\{y_0+\int_0^{t\wedge\tau_\epsilon} \gamma(Y(s))\,\mathrm{d}s+W(t\wedge\tau_\epsilon)\in D(x_1,\cdot)\}}\big]}
{\Lambda(y_0,x_1)\,\pp\big(y_0+\int_0^{t\wedge\tau_\epsilon} \gamma(Y(s))\,\mathrm{d}s+W(t\wedge\tau_\epsilon)\in D(x_1,\cdot)\big)} \\
& + \frac{(\Lambda\,\nabla_y\,f)(x_1,y_0)\,\ev\big[W(t\wedge\tau_\epsilon)\,\mathbf{1}_{\{y_0+\int_0^{t\wedge\tau_\epsilon} \gamma(Y(s))\,\mathrm{d}s+W(t\wedge\tau_\epsilon)\in D(x_1,\cdot)\}}\big]}
{\Lambda(y_0,x_1)\,\pp\big(y_0+\int_0^{t\wedge\tau_\epsilon} \gamma(Y(s))\,\mathrm{d}s+W(t\wedge\tau_\epsilon)\in D(x_1,\cdot)\big)} \\
& + \frac{\frac{1}{2}\,\ev\big[W(t\wedge\tau_\epsilon)'\,
(2\,(\nabla_y\Lambda)(\nabla_y f)'+\Lambda\,(\mathrm{Hess}_y f))(y_0,x_1)\,W(t\wedge\tau_\epsilon)\,\mathbf{1}_{\{Y(t\wedge\tau_\epsilon)\in D(x_1,\cdot)\}}\big]}{\Lambda(y_0,x_1)\,\pp\big(y_0+\int_0^{t\wedge\tau_\epsilon} \gamma(Y(s))\,\mathrm{d}s+W(t\wedge\tau_\epsilon)\in D(x_1,\cdot)\big)}+o(t)
\end{split}
\]
as $t\downarrow0$. At this point, $\int_0^{t\wedge\tau_\epsilon} \gamma(Y(s))\,\mathrm{d}s=\gamma(y_0)\,(t\wedge\tau_\epsilon)+o(t)$, $t\downarrow0$, the exponentially small probability of the event $\{\tau_\epsilon<t\}$, and a straightforward computation with Gaussian distributions show that the asymptotic lower bound is of the desired form
\[
f(x_1,y_0)+t\,({\mathcal A}^Y f+(\nabla_y V)'\,(\nabla_y f))(x_1,y_0)+o(t)
\]
with an error term uniform in $(x_1,y_0)$ as long as the distance $\mathrm{dist}(y_0,\partial D(x_1,\cdot))$ from $y_0$ to $\partial D(x_1,\cdot)$ is not on the order of $\sqrt{t}$. More precisely, for each $\delta>0$, there exist constants $0<c_1(\delta),c_2(\delta),c_3(\delta)<\infty$ such that the asymptotic lower bound
\eq\label{case2LBD}
f(x_1,y_0)+t({\mathcal A}^Y f+(\nabla_y V)'\,(\nabla_y f))(x_1,y_0)
-t\big(\delta+c_1(\delta)\mathbf{1}_{\{\mathrm{dist}(y_0,\partial D(x_1,\cdot))\in[c_2(\delta)\sqrt{t},c_3(\delta)\sqrt{t}]\}}\big)+o(t)
\en
applies with an error term uniform in $(x_1,y_0)$.} 

\medskip

\noindent{\textbf{Case 3.} Lastly, consider the fraction in \eqref{thm3Z_transition} in the case that $(x_1,y_0)\notin K$. View the integrals in the numerator and the denominator as expectations over the process $Y$ and estimate the fraction in absolute value from above by replacing $f$ with $|f|$ and restricting both expectations to the event that $Y$ reaches the boundary of $\{y:\,(x_1,y)\in K\}$ in $D(x_1,\cdot)$ by time $t$ (note that this does not change the expectation in the numerator). Dividing the numerator and the denominator by the probability of the latter event one can further rewrite the expectations as conditional expectations given that event. As $t\downarrow0$, the numerator in the resulting fraction becomes exponentially small uniformly in $(x_1,y_0)\notin K$, since $Y$ has an exponentially small probability of making it from the boundary of $\{y:\,(x_1,y)\in K\}$ to $\{y:\,(x_1,y)\in\mathrm{supp}\,f\}$ in time less than $t$ (note that the drift of $Y$ is bounded on the compact $\{y:\,(x_1,y)\in K\}$) and $\Lambda f$ is bounded. On the other hand, the denominator is bounded away from zero uniformly in $(x_1,y_0)\notin K$ as $t\downarrow0$ thanks to the uniform positivity of $\Lambda$ on the compact $K$. Therefore, in this case, the fraction in \eqref{thm3Z_transition} is $o(t)$ uniformly in $(x_1,y_0)\notin K$.} 

\medskip

{Combining the results in the three cases we conclude that, for each $\delta>0$, one has the asymptotic lower bound
\[
\begin{split}
& \int_{\mcal{X}} P_t(x_0,\mathrm{d}x_1)\,\Big(f(x_1,y_0)+t\,\big({\mathcal A}^Y f+(\nabla_y V)'\,\nabla_y f\big)(x_1,y_0) \\
& \qquad\qquad\quad\quad\quad -t\,c_1(\delta)\,\mathbf{1}_{\{\mathrm{dist}(y_0,\partial D(x_1,\cdot))\in[c_2(\delta)\sqrt{t},c_3(\delta)\sqrt{t}],(x_1,y_0)\in K\}}\Big)-t\,\delta+o(t),\quad t\downarrow0
\end{split}
\]
on the quantities of \eqref{thm3Z_transition} where the error term is uniform in $(x_0,y_0)\in D$ and $0<c_1(\delta),c_2(\delta),c_3(\delta)<\infty$ are suitable constants. Moreover, the same procedure as in Step 1 in the proof of Theorem \ref{main2} allows to simplify the asymptotic lower bound further to
\eq\label{asymp_LBD}
\begin{split}
& f(x_0,y_0)+t\,\big({\mathcal A}^X f + {\mathcal A}^Y f + (\nabla_y V)'\,\nabla_y f\big)(x_0,y_0) \\
& -t\,c_1(\delta)\,\pp\big(\mathrm{dist}(y_0,\partial D(\tilde{Z}_1(t),\cdot))
\in[c_2(\delta)\sqrt{t},c_3(\delta)\sqrt{t}],(\tilde{Z}_1(t),y_0)\in K\big)
-t\,\delta+o(t) ,\quad t\downarrow0.
\end{split}
\en
Here the error term is again uniform in $(x_0,y_0)\in D$, since ${\mathcal A}^X f$ and ${\mathcal A}^Y f+(\nabla_y V)'\,\nabla_y f$ are bounded and uniformly continuous, and the probability of the process $X$ moving by more than a fixed amount within $\{x:\,(x,y_0)\in K\}$ in time less than $t$ becomes exponentially small as $t\downarrow0$.} 

\medskip

{We are now ready to show that $f(\tilde{Z}(t))-\int_0^t ({\mathcal A}^Z f)(\tilde{Z}(s))\,\mathrm{d}s$, $t\ge0$ is a submartingale in the filtration generated by $\tilde{Z}$ and to thereby finish the proof of the theorem. Since $\tilde{Z}$ is a time-homogeneous Markov process by definition, it is enough to prove 
\eq\label{submart}
\ev\bigg[f(\tilde{Z}(t))-\int_0^t ({\mathcal A}^Z f)(\tilde{Z}(s))\,\mathrm{d}s\bigg]\ge f(\tilde{Z}(0))
\en
for every $t>0$. To this end, for every $t>0$ and $\iota\in(0,t)$, we set $t^\iota_l=(l\iota)\wedge t$, $l=0,1,\ldots,L(\iota)$ where $L(\iota)$ is the smallest integer greater or equal to $\frac{t}{\iota}$. With these notations we can write
\eq\label{tele}
\ev\big[f(\tilde{Z}(t))\big]-f(\tilde{Z}(0))
=\sum_{l=1}^{L(\iota)} \ev\big[\ev\big[f(\tilde{Z}(t^\iota_l))-f(\tilde{Z}(t^\iota_{l-1}))\,\big|\,\tilde{Z}(t^\iota_{l-1})\big]\big]
\en
where each summand admits the asymptotic lower bound 
\[
\begin{split}
\iota\ev\big[({\mathcal A}^Z f)(\tilde{Z}(t^\iota_{l-1}))\big]-\iota c_1(\delta)\pp\Big(\mathrm{dist}(\tilde{Z}_2(t^\iota_{l-1}),
\partial D(\tilde{Z}_1(t^\iota_l),\cdot))
\in[c_2(\delta)\sqrt{\iota},c_3(\delta)\sqrt{\iota}], 
\quad\quad\quad\quad\quad \\
(\tilde{Z}_1(t^\iota_l),\tilde{Z}_2(t^\iota_{l-1}))\in K\Big)
-\iota\,\delta+o(\iota)
\end{split}
\]
as $\iota\downarrow0$ by \eqref{asymp_LBD} (recall the time-homogeneity of the Markov process $\tilde{Z}$ and the uniformity of the error term in \eqref{asymp_LBD}). Since $\tilde{Z}_2$ has the law of $Y$ and therefore increments on the order of $\sqrt{\iota}$, we can bound this further by
\[
\begin{split}
\iota\ev\big[({\mathcal A}^Z f)(\tilde{Z}(t^\iota_{l-1}))\big]-\iota c_1(\delta)\pp\Big(\mathrm{dist}(\tilde{Z}_2(t^\iota_l),
\partial D(\tilde{Z}_1(t^\iota_l),\cdot))\le c_4(\delta)\sqrt{\iota}, 
(\tilde{Z}_1(t^\iota_l),\tilde{Z}_2(t^\iota_l))\in K^{\delta,\iota}\Big) \\
-2\iota\delta+o(\iota)
\end{split}
\]
where $0<c_4(\delta)<\infty$ is a suitable constant and $K^{\delta,\iota}$ is the $c_4(\delta)\sqrt{\iota}$-neighborhood of $K$ in $D$. Conditioning on $\tilde{Z}_2(t^\iota_l)$, parametrizing $\partial D$ locally as the graph $(x(y),y)'$ of a smooth function $x(y)$, and using the boundedness of $\Lambda$ on $K^{\delta,\iota}$ we see that the latter probabilities are all not greater than $c_5(\delta)\sqrt{\iota}$ for some $0<c_5(\delta)<\infty$. Thus, from \eqref{tele} we obtain   
\[
\begin{split}
\ev\big[f(\tilde{Z}(t))\big]\!-\!f(\tilde{Z}(0))
\!\ge\!\lim_{\iota\downarrow0} \sum_{l=1}^{L(\iota)} \iota
\ev[({\mathcal A}^Z f)(\tilde{Z}(t^\iota_{l-1}))]\!-\!2\delta t
\!=\! \ev\bigg[\lim_{\iota\downarrow0} \sum_{l=1}^{L(\iota)} \iota
({\mathcal A}^Z f)(\tilde{Z}(t^\iota_{l-1}))\bigg]\!-\!2\delta t \\
= \ev\bigg[\int_0^t ({\mathcal A}^Z f)(\tilde{Z}(s))\,\mathrm{d}s\bigg]\!-\!2\delta t.
\end{split}
\]
Here the first equality is a consequence of the boundedness of ${\mathcal A}^Z f$ and the Dominated Convergence Theorem, the second equality follows from the Riemann integrability of the bounded continuous function $s\mapsto ({\mathcal A}^Z f)(\tilde{Z}(s))$, and the existence of the two limits can be justified by following the computation in the reversed direction. The desired inequality \eqref{submart} now follows by taking the limit $\delta\downarrow0$. \ep}
}

\medskip

In Theorem \ref{main1} we impose that $\Lambda(y,\cdot)$ is a probability density for each $y$. Suppose $\Lambda$ is a solution of \eqref{PDE} in the sense specified in Theorem \ref{main1} with $\Lambda(y,\cdot)$ being the density of a finite positive measure with total mass $\tau(y)$. Then, we can define the normalized density according to 
\eq\label{eq:normkernel}
\xi(y,x)=\frac{\Lambda(y,x)}{\tau(y)}.
\en
Let $\Xi$ denote the Markov transition operator corresponding to $\xi$. Our next theorem shows that $\Xi$ intertwines the semigroup $(P_t,\;t\ge 0)$ with a Doob's $h$-transform of the semigroup $(Q_t,\; t\ge 0)$. 

\begin{thm}\label{thm:htrans}
Consider the setup of the preceding paragraph {and suppose that the total variation norm of $({\mathcal A}^X)^*\Lambda(y,\cdot)$ is locally bounded as $y$ varies, and that the function $\tau$ is continuous.} Then $\tau$ is a harmonic function for $\mcal{A}^Y$, that is, $\tau(Y(t))$, $t\geq0$ is a positive local martingale for the diffusion $Y$ of Assumption \ref{main_asmp}. 

{Define the stopping times $\upsilon_R$, $R>0$, as the first exit times of $Y$ from balls of radius $R$ around $y_0:=Y(0)$ and suppose that the process $Y^\tau$ resulting from $Y$ by changes of measure with densities $\frac{\tau(Y(\upsilon_R))}{\tau(y_0)}$, $R=1,2,\ldots$ on ${\mathcal F}^Y_{\upsilon_R}$, $R=1,2,\ldots$, respectively, does not explode. Then $Y^\tau$ is a Feller-Markov process whose generator reads   
\eq\label{htrans_gen}
\mcal{A}^\tau\,\phi=\tau^{-1} \mcal{A}^Y\left( \tau \phi \right)
\en
for functions $\phi$ with $\tau\phi$ in the domain of $\mcal{A}^Y$, and whose semigroup $(Q^\tau_t)$ satisfies $Q^\tau \iprod{\Xi} P$.}
\end{thm}

\noindent\textbf{Proof.} {To see that $\tau$ is harmonic it suffices to show that $\tau(Y(t\wedge\upsilon_R))$, $t\ge0$ is a martingale for every $R=1,2,\ldots$. We only prove 
\eq\label{tau_mart}
\ev\big[\tau(Y(t\wedge\upsilon_R))\big]=\tau(y_0),\quad t\ge0,
\en  
since then the martingale property of $\tau(Y(t\wedge\upsilon_R))$, $t\ge0$ can be obtained by the same argument in view of the Markov property of $Y$. To establish \eqref{tau_mart} we let $f_l$, $l\in\nn$ be a sequence of nonnegative $C_0(\mcal{X})$ functions increasing to the function constantly equal to $1$ on $\mcal{X}$ and set $g_l=\int_0^1 P_s f_l\,\mathrm{d}s$, $l\in\nn$. Then it easy to check (see, e.g., the proof of Lemma II.1.3 (iii), (iv) in \cite{EN}) that each function $g_l$ is in the domain of ${\mathcal A}^X$ and ${\mathcal A}^X g_l=P_1 f_l-f_l$. Now, \eqref{tau_mart} can be obtained by the following computation:
\[
\begin{split}
& \ev\big[\tau(Y(t\wedge\upsilon_R))\big]-\tau(y_0)
=\int_{\mcal X} \ev\big[\Lambda(Y(t\wedge\upsilon_R),x)\big]-\Lambda(y_0,x)\,\mathrm{d}x \\
& =\lim_{l\to\infty} \int_{\mcal X} \ev\big[\Lambda(Y(t\wedge\upsilon_R),x)\big]\,g_l(x)-\Lambda(y_0,x)\,g_l(x)\,\mathrm{d}x \\
& =\lim_{l\to\infty} \int_{\mcal X} \ev\bigg[\int_0^{t\wedge\upsilon_R}
{\mathcal A}^Y\Lambda(Y(s),x)\,\mathrm{d}s\bigg]\,g_l(x)\,\mathrm{d}x \\
& =\lim_{l\to\infty}  \ev\bigg[\int_0^{t\wedge\upsilon_R} \int_{\mcal X}
{\mathcal A}^Y\Lambda(Y(s),x)\,g_l(x)\,\mathrm{d}x\,\mathrm{d}s\bigg] \\
& =\lim_{l\to\infty}  \ev\bigg[\int_0^{t\wedge\upsilon_R} \int_{\mcal X}
\Lambda(Y(s),x)\,(P_1 f_l-f_l)(x)\,\mathrm{d}x\,\mathrm{d}s\bigg]=0.
\end{split}
\]
Here the first identity follows from Fubini's Theorem with nonnegative integrands; the second identity is a consequence of the Monotone Convergence Theorem; the third identity results from Dynkin's formula (see, e.g., Lemma 17.21 in \cite{Ka}); the fourth identity follows from Fubini's Theorem upon recalling \eqref{PDE} and the assumed local boundedness of the total variation norm of $({\mathcal A}^X)^*\Lambda(y,\cdot)$; the fifth identity is a direct consequence of \eqref{PDE} and the defining property of $({\mathcal A}^X)^*$; and the last identity is due to the pointwise convergence $P_1 f_l-f_l\to 0$, which in turn follows from the Monotone Convergence Theorem, and the Dominated Convergence Theorem (note $|P_1f_l-f_l|\le 1$ and recall that $\tau$ is continuous).}

\medskip

{Next, consider the process $Y^\tau$. Localizing by means of the stopping times $\upsilon_R$, $R=1,2,\ldots$ and using the non-explosion of $Y^\tau$ it is easy to see that, for every $t\ge0$, the law of $Y^\tau$ is absolutely continuous with respect to the law of $Y$ on ${\mathcal F}_t^Y$ with the corresponding density being given by $\frac{\tau(Y(t))}{\tau(y_0)}$  (see, e.g., the proof of Theorem 7.2 in \cite{LS} for a similar argument). Moreover, to establish the Markov property of $Y^\tau$ it suffices to show that, for every $h\in C_c(\mcal{Y})$ and $0\le s<t<\infty$, 
\eq\label{htrans_semi}
\ev\big[h(Y^\tau(t))\,\big|\,{\mathcal F}_s^Y\big]
=\tau(Y^\tau(s))^{-1} Q_{t-s}(\tau h)(Y^\tau(s)).
\en
To this end, we pick an event $A\in{\mathcal F}_s^Y$ and compute 
\[
\begin{split}
& \ev\big[\tau(Y^\tau(s))^{-1} Q_{t-s}(\tau h)(Y^\tau(s))\,\mathbf{1}_A\big]
=\frac{1}{\tau(y_0)}\,\ev\big[Q_{t-s}(\tau h)(Y(s))\,\mathbf{1}_A\big] \\
& =\frac{1}{\tau(y_0)}\,\ev\big[\ev\big[(\tau h)(Y(t))\,\mathbf{1}_A\,\big|
\,{\mathcal F}_s^Y\big]\big]=\ev[h(Y^\tau(t))\,\mathbf{1}_A].
\end{split}
\]}

\smallskip

{We proceed to the Feller property of $Y^\tau$. Consider the function $y\mapsto\tau(y)^{-1}\,Q_t(\tau h)(y)$ for some $h\in C_0(\mcal{Y})$ and $0\le t<\infty$ whose membership in $C_0(\mcal{Y})$ we need to show. A uniform approximation of $h$ by functions in $C_c(\mcal{Y})$ reveals that we may assume without loss of generality that $h\in C_c(\mcal{Y})$. For such an $h$ the continuity of $y\mapsto\tau(y)^{-1}\,Q_t(\tau h)(y)$ is a direct consequence of the Feller property of $Y$. Moreover, for a point $y_0$ of distance $R$ from the support of $h$ we have
\[
\begin{split}
& \big|\tau(y_0)^{-1}\,Q_t(\tau h)(y_0)\big|
=\Big|\ev\big[(\tau(Y(\upsilon_R)))^{-1}\ev\big[\tau(Y(t))\,h(Y(t))\,\mathbf{1}_{\{\upsilon_R\le t\}}\,\big|\,{\mathcal F}^Y_{\upsilon_R}\big]\big]\Big| \\
& \le\frac{\sup_{y\in\mathrm{supp}\,h} \tau(y)}{\inf_{y\in\mathrm{supp}\,h} \tau(y)}\,\ev[|h(Y(t))|].
\end{split}
\]
The latter expectation tends to zero in the limit $R\to\infty$ by the Feller property of $Y$. Therefore the function $y\mapsto\tau(y)^{-1}\,Q_t(\tau h)(y)$ belongs to $C_0(\mcal{Y})$ which, in view of path continuity, implies that $Y^\tau$ is a Feller process. The formula \eqref{htrans_gen} for its generator follows immediately from the formula \eqref{htrans_semi} for its semigroup. Now, to prove $Q^\tau\iprod{\Xi} P$, we first claim that for $f \in C_c^{\infty}(\mcal{X})\cap\mcal{D}(\mcal{A}^X)$, $Lf \in \mcal{D}(\mcal{A}^Y)$ and $\mcal{A}^YLf = L\mcal{A}^Xf$. We calculate

\begin{equation}\label{gencalc}
\begin{split}
\frac{1}{t}\big(Q_t Lf(y) - Lf(y)\big) = \frac{1}{t}&\int_{\mcal{X}}\big(Q_t \Lambda(y,x) - \Lambda(y,x)\big)f(x)\,\mathrm{d}x \\
=\frac{1}{t}&\int_{\mcal{X}}\bigg(\int_0^t Q_s \mcal{A}^Y \Lambda(y,x)\, \mathrm{d}s\bigg)f(x)\,\mathrm{d}x \\
= \frac{1}{t}&\int_0^t Q_s\bigg(\int_{\mcal{X}}\mcal{A}^Y \Lambda(\cdot,x)f(x)\,\mathrm{d}x \bigg)(y)\,\mathrm{d}s\\
=\frac{1}{t}&\int_0^t Q_s\bigg(\int_{\mcal{X}}\Lambda(\cdot,x)\mcal{A}^Xf(x)\,\mathrm{d}x \bigg)(y)\,\mathrm{d}s
\end{split}
\end{equation}
The first equality follows from Fubini's Theorem and the boundedness of $f$ and the second equality is due to the Kolmogorov forward equation for the semigroup $(Q_t)$. The third equality results from Fubini's theorem which applies due to the uniform boundedness of $\mcal{A}^Y\Lambda$ on $\text{supp}(f) \times \mcal{Y}$ and the compactness of $\text{supp}(f)$. The final equality follows from \eqref{PDE}. Due to the fact that $L\mcal{A}^Xf \in C_0(\mcal{Y})$, the Feller-Markov property of $Y$ implies that the final term in \eqref{gencalc} converges uniformly to $L\mcal{A}^X f$ and so we have our claim. The formula \eqref{htrans_gen} then shows that $\mcal{A}^\tau \Xi f = \tau^{-1} \mcal{A}^Y L f = \tau^{-1}L\mcal{A}^X f = \Xi \mcal{A}^X f$ which can be extended to $f \in \mcal{D}(\mcal{A}^X)$ due to our assumption that $C_c^{\infty}(\mcal{X})\cap\mcal{D}(\mcal{A}^X)$ is a core for $\mcal{D}(\mcal{A}^X)$. This, along with the uniqueness for the Cauchy problem associated with ${\mathcal A}^\tau$ (Proposition II.6.2 in \cite{EN}), yields $Q^\tau\iprod{\Xi} P$.
} \ep

\medskip

If $\mcal{A}^Y$ is the generator of a one-dimensional homogeneous diffusion, then there are only two linearly independent choices for $\tau$, the constant function and the scale function of $\mcal{A}^Y$. See Remark \ref{rmk:harfn} in Section \ref{sec:exmp} below and the proposition preceding it for more details. In general, suppose $\mcal{A}^Y$ satisfies the Liouville property, that is, any bounded function $\tau$ satisfying $\mcal{A}^Y \tau=0$ has to be constant. Then, once we show $\tau$ is bounded, a further $h$-transform is unnecessary.  The Liouville property is satisfied by many natural operators. For example, if {$\mcal{A}^Y$ is a strictly elliptic operator of the form $\frac{1}{2}\sum_{k,l=1}^n \partial_{y_k}\rho_{kl}(y)\partial_{y_l}$ with $\rho$ being bounded, then the Liouville property holds (see \cite{Mo}, p. 590). For examples of nonreversible diffusions possessing the Liouville property we refer to \cite{PW}.}

\section{On various properties of intertwined diffusions} \label{sec:properties}

We prove several results on properties of intertwined processes and semigroups. {We start with an iteration of the coupling construction in Theorem \ref{main1}. To this end, consider the setup of Theorem \ref{main1} and suppose one is given another diffusion $S$ with state space $\mcal{S}\subset\rr^k$ and generator 
\eq\label{Sgen}
\mcal{A}^S = \sum_{i=1}^k \eta_i(s) \partial_{s_i} + \frac{1}{2}\sum_{i,j=1}^k \sigma_{ij}(s) \partial_{s_i} \partial_{s_j}
\en
satisfying Assumption \ref{main_asmp}. In addition, let $\tilde{L}$ be a stochastic transition operator from $\mcal{S}$ to $\mcal{Y}$ with a positive kernel $\tilde{\Lambda}$ and set $\tilde{V}=\log\tilde{\Lambda}$. The following theorem provides a coupling construction realizing the commutative diagram in Figure \ref{fig:Mkvdiff}.} 

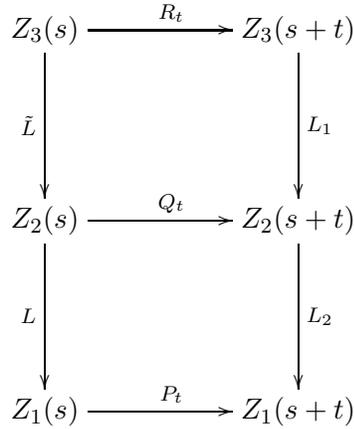
\begin{figure}[t]
\centerline{
\xymatrix@=5em{
Z_3(s) \ar[d]_{\tilde{L}} \ar[r]^{R_t} & Z_3(s+t) \ar[d]^{\tilde{L}} \\
Z_2(s) \ar[d]_{L} \ar[r]^{ Q_t } & Z_2(s+t) \ar[d]^{L} \\
Z_1(s) \ar[r]^{P_t} & Z_1(s+t)
}
} 
\caption{Hierarchy of intertwined diffusions.}
\label{fig:Mkvdiff}
\end{figure}

\begin{thm}\label{thm:product}
{In the setting of the previous paragraph suppose that the operator 
\[
f\mapsto\int_{\mcal{Y}} \tilde{\Lambda}(\cdot,y)\,f(y)\,\mathrm{d}y
\]
maps $C_0(\mcal{Y})$ into $C_0(\mcal{S})$ with $\tilde{\Lambda}$ being continuously differentiable in $s$. Assume that the diffusion $(Z_1,Z_2)$  whose generator is given by \eqref{Zgen} satisfies Assumption \ref{main_asmp} and the assumptions of Theorem 1 (in particular, both $\mcal{X}$ and $\mcal{Y}$ must be open). For any $z\in \rr^{m+n+k}=\rr^m\times\rr^n\times\rr^k$ write $z=(x,y,s)$ and consider a diffusion $Z=(Z_1,Z_2,Z_3)$ with state space $\mcal{X}\times\mcal{Y}\times\mcal{S}$, generator   
\[
\mcal{A}^Z = \mcal{A}^X + \mcal{A}^Y + \mcal{A}^S + \big(\nabla_y V(y,x)\big)' \rho(y)\,\nabla_y + \big(\nabla_s \tilde{V}(s,y)\big)'\sigma(s)\,\nabla_s\,,
\] 
and boundary conditions corresponding to those of $X,\,Y,\,S$. Suppose that the SDE or SDE with reflection (SDER) associated with $\mcal{A}^Z$ is well-posed, its solution is a Feller-Markov process and that the conditional density of $Z_2(0)$ at $y$, given $Z_3(0)=s$, is $\tilde{\Lambda}(s,y)$, and the conditional density of $Z_1(0)$ at $x$, given $Z_2(0)=y, Z_3(0)=s$, is $\Lambda(y,x)$ (in particular, it is independent of $s$).} 
 
{If $\tilde{\Lambda}$ is such that $\tilde{\Lambda}(\cdot,y)$ is in the domain of ${\mathcal A}^S$ for all $y\in\mcal{Y}$ with ${\mathcal A}^S\tilde{\Lambda}$ being continuous on $\mcal{S}\times\mcal{Y}$ and bounded on $\mcal{S}\times K$ for any compact subset $K$ of $\mcal{Y}$, $\tilde{\Lambda}(s,\cdot)$ is in the domain of $({\mathcal A}^Y)^*$ for all $s\in\mcal{S}$, $C_c^{\infty}(\mcal{X}\times\mcal{Y}\times\mcal{S})\cap\mcal{D}(\mcal{A}^Z)$ is a core for $\mcal{D}(\mcal{A}^Z)$, and
\eq\label{new_wave}
\left(\mcal{A}^Y\right)^* \tilde{\Lambda} = \mcal{A}^S\,\tilde{\Lambda}\quad\text{on}\quad\mcal{Y}\times\mcal{S},
\en 
then $Z=S\,\langle\tilde{\Lambda}\Lambda\rangle\,(Z_1,Z_2)$ and satisfies \eqref{eq:condlaw}.} 
\end{thm}

\noindent\textbf{Proof.} {By applying It\^o's formula to functions of $(Z_1,Z_2)$ it is easy to see that $(Z_1,Z_2)$ solves the SDE (SDER resp.) associated with the generator of \eqref{Zgen} and the reflection directios corresponding to those of $X,\,Y$. In particular, $(Z_1,Z_2)$ is the intertwining constructed in Theorem \ref{main1}, and we write ${\mathcal A}^{Z_1,Z_2}$ for the corresponding generator.}

\medskip

{It is easily checked that $\tilde{\Lambda}\Lambda$ satisfies conditions (i)-(iii) of Assumption \ref{asmpthm1}, so it only remains to show that $\tilde{\Lambda}(s,\cdot)\,\Lambda$ is in the domain of  $\left({\mathcal A}^{Z_1,Z_2}\right)^*$ for all $s\in\mcal{S}$, and 
\eq\label{lift}
\big({\mathcal A}^{Z_1,Z_2}\big)^*(\tilde{\Lambda}\,\Lambda)
=\big(\big(\mcal{A}^Y\big)^* \tilde{\Lambda}\big)\,\Lambda
\quad\text{on}\quad\mcal{X}\times\mcal{Y}\times\mcal{S},
\en
since then the theorem will follow from Theorem \ref{main1} for the diffusions $(Z_1,Z_2)$, $S$ and kernel $\tilde{\Lambda}(s,y)\,\Lambda(y,x)$ (note that the right-hand side of \eqref{lift} is $\mcal{A}^S (\tilde{\Lambda}\,\Lambda)$ by \eqref{new_wave}). In other words, we need to prove  
\eq\label{lift2}
\int_{\mcal{X}\times\mcal{Y}} \big(\big(\mcal{A}^Y\big)^* \tilde{\Lambda}\big)(s,y)\,\Lambda(y,x)\,f(x,y)\,\mathrm{d}x\,\mathrm{d}y
=\int_{\mcal{X}\times\mcal{Y}} \tilde{\Lambda}(s,y)\,\Lambda(y,x)\,({\mathcal A}^{Z_1,Z_2}f)(x,y)\,\mathrm{d}x\,\mathrm{d}y
\en
for all $f\in C_0(\mcal{X}\times\mcal{Y})$ in the domain of ${\mathcal A}^{Z_1,Z_2}$.}

\medskip

{Without loss of generality we may and will assume that $f\in C^\infty_c(\mcal{X}\times\mcal{Y})\cap\mcal{D}(\mcal{A}^{Z_1,Z_2})$, since otherwise we can approximate $f$ by a sequence of functions $f_l$, $l\in\nn$ in $C^\infty_c(\mcal{X}\times\mcal{Y})\cap\mcal{D}(\mcal{A}^{Z_1,Z_2})$ such that $f_l\to f$ and $({\mathcal A}^{Z_1,Z_2}f_l)\to({\mathcal A}^{Z_1,Z_2}f)$ uniformly on $\mcal{X}\times\mcal{Y}$ and pass to the limit $l\to\infty$ in the identity \eqref{lift2} for $f_l$. Now, an application of Fubini's Theorem together with the definition of $({\mathcal A}^Y)^*$ and a product rule as in \eqref{prod_rule} gives for the left-hand side of \eqref{lift2}: 
\[
\begin{split}
&\int_{\mcal{X}} \int_{\mcal{Y}} \tilde{\Lambda}(s,y)\,
\big(({\mathcal A}^Y\Lambda)f+(\nabla_y\Lambda)'\rho\,\nabla_y f+\Lambda\,{\mathcal A}^Y f\big)(y,x)\,\mathrm{d}y\,\mathrm{d}x \\
&=\int_{\mcal{X}} \int_{\mcal{Y}} \tilde{\Lambda}(s,y)\big(({\mathcal A}^Y\Lambda)f\big)(y,x)\,\mathrm{d}y\,\mathrm{d}x
+\int_{\mcal{X}} \int_{\mcal{Y}} \tilde{\Lambda}(s,y)
\big(\Lambda((\nabla_y V)'\rho\,\nabla_y f+{\mathcal A}^Y f)\big)(y,x)\,\mathrm{d}y\,\mathrm{d}x.
\end{split}
\]
In view of Fubini's Theorem, \eqref{PDE}, and the definition of $({\mathcal A}^X)^*$, the first summand in the latter expression computes to 
\[
\int_{\mcal{Y}} \tilde{\Lambda}(s,y)\,\int_{\mcal{X}} \Lambda(y,x)\,({\mathcal A}^X f)(x,y)\,\mathrm{d}x\,\mathrm{d}y.  
\] 
Plugging this in one obtains the right-hand side of \eqref{lift2} thanks to Fubini's Theorem.} \ep

\begin{rmk} 
{It is clear that a repeated application of the above theorem can create couplings $(Z_1,Z_2,\ldots,Z_l)$ of any number of diffusions. We refer to Section \ref{sec_Whittaker} below for an important example arising in the study of random polymers.}
\end{rmk}

\medskip

\nin\tbf{Duality and time-reversal.} Our next result is a version of Bayes' rule. Suppose $Q \iprod{L} P$ for some $(P_t)$, $(Q_t)$, and $L$. Is there a transition kernel $\lhat$ such that $P\,\langle\lhat\rangle\,Q$ (see Figure \ref{fig:dualintw})? We show that this is the case when both $(P_t)$ and $(Q_t)$ are reversible with respect to their respective invariant measures. This also allows to find the time reversal of the diffusion with generator given by \eqref{Zgen}.

\begin{figure}[t]
\centerline{
\xymatrix@=8em{
Z_2(s) \ar @<1ex> [d]^{L} \ar[r]^{ Q_t } & Z_2(s+t) \ar[d]^{L}\\
Z_1(s) \ar[r]^{P_t} \ar @/^2pc/ [u]^{\bayL} & Z_1(s+t) \ar @/_2pc/ [u]_{\bayL} 
}
} 
\caption{Flipping the order of intertwining.}
\label{fig:dualintw}
\end{figure}
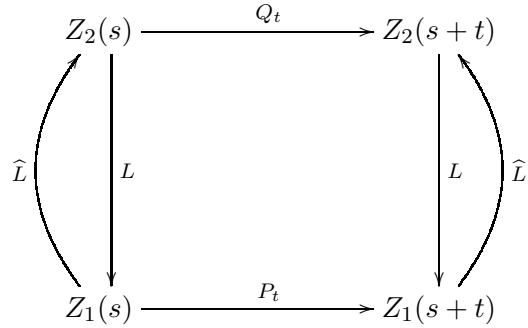

\begin{defn}
We say that two semigroups $(P_t)$ and $(\phat_t)$ on $\rr^d$ are in duality with respect to a probability measure $\nu$ if they satisfy
\eq\label{eq:duality}
\int_{\rr^d} \left(P_t\,f\right)\,g\,\mathrm{d}\nu = \int_{\rr^d} f\,(\phat_t\,g)\,\mathrm{d}\nu \quad \text{for all bounded measurable $f,\,g$ and all $t\ge 0$}. 
\en
We say $(P_t)$ is reversible with respect to $\nu$ if the above holds with $(\phat_t)=(P_t)$. 
\end{defn}

The definition can be restated as: the Markov process with semigroup $(P_t)$ and initial distribution $\nu$, looked at backwards in time, is Markovian with transition semigroup $(\phat_t)$.

\medskip

Consider two diffusion semigroups $(P_t)$ and $(Q_t)$ as in Assumption \ref{main_asmp} and a stochastic transition operator $L$ such that $Q \iprod{L} P$. Suppose there exist semigroups $(\phat_t)$, $(\qhat_t)$ and two probability measures $\nu_1$, $\nu_2$ such that
\begin{enumerate}
\item[(i)] $(P_t)$ and $(\phat_t)$ are in duality with respect to $\nu_1$, and $(Q_t)$ and $(\qhat_t)$ are in duality with respect to $\nu_2$.
\item[(ii)] $\nu_1$, $\nu_2$ have full support on $\mcal{X}$, $\mcal{Y}$ and are absolutely continuous with respect to the Lebesgue measure with continuous density functions $h_1$, $h_2$, respectively.
\item[(iii)]  $\nu_1$ is the unique stationary measure for $(P_t)$ and $\nu_2$ is a stationary measure for $(Q_t)$.
\comment{\item[(iii')] $(P_t)$ and $(Q_t)$ are ergodic in the sense that, for any probability measures $\mu$, $\nu$ on $\mcal{X}$, $\mcal{Y}$, respectively, one has the weak convergences
\eq\label{eq:ergodic}
\lim_{t\rightarrow \infty} \frac{1}{t}\int_0^t \mu\,P_s\,\mathrm{d}s
= \nu_1, 
\quad \lim_{t\rightarrow \infty} \frac{1}{t}\int_0^t \nu\,Q_s\,\mathrm{d}s
= \nu_2. 
\en }
\end{enumerate}
\begin{thm}\label{thm:time reversal}
Let $\Lambda$ denote the transition kernel corresponding to $L$ and suppose that it is jointly continuous. Define
\eq\label{eq:bden}
\lamhat(x,y)=\Lambda(y,x)\,\frac{h_2(y)}{h_1(x)}
\en
and write $\lhat$ for the corresponding transition operator. Then, $\lamhat$ is a stochastic transition kernel, and $\phat\,\langle\lhat\rangle\,\qhat$. 
\end{thm}

\comment{
we have the following conclusions:

\begin{enumerate}[(i)]
\item $\lamhat$ is a stochastic transition kernel, and $\phat\,\langle\lhat\rangle\,\qhat$. 
\item Suppose that the conditions of Theorem \ref{main1} are satisfied for each of the triplets $(P_t)$, $(Q_t)$, $L$ and $(\phat_t)$, $(\qhat_t)$, $\lhat$ and let $(R_t)$ and $(\rhat_t)$ be the semigroups of the diffusions corresponding to the intertwinings $Q \iprod{L} P$ and $\phat\,\langle\lhat\rangle\,\qhat$ via Theorem \ref{main1}, respectively. Define a probability measure $\rho$ on $\mcal{X}\times\mcal{Y}$ given by its density $\Lambda(y,x)\,h_2(y)$ with respect to the Lebesgue measure. Then, the semigroups $(R_t)$ and $(\rhat_t)$ are in duality with respect to $\rho$. 
\end{enumerate}
}

\noindent\textbf{Proof.} We first argue that $\lamhat$ is a stochastic transition kernel (and, thus, $\lhat$ is a stochastic transition operator). We need to show that
\eq\label{eq:revprob}
\int_{\mcal{Y}} \Lambda(y,x)\,h_2(y)\,\mathrm{d}y = h_1(x),
\en
which is equivalent to the identity $\nu_2 L = \nu_1$. We calculate $\nu_2 L P_t = \nu_2 Q_t L = \nu_2 L$ and, by assumption (iii), conclude that $\nu_2 L = \nu_1$ from which \eqref{eq:revprob} readily follows.
\comment{To see this, let $f$ be a continuous bounded function on $\mcal{X}$, $\nu$ be a probability measure on $\mcal{Y}$, and $\mu=\nu\,L$. Using the ergodicity condition \eqref{eq:ergodic} twice, $Q\iprod{L}P$, and Fubini's Theorem we derive
\[
\begin{split}
& \int_{\mcal{X}} f(x)\,h_1(x)\,\mathrm{d}x
= \lim_{t\rightarrow \infty} \frac{1}{t}\int_0^t \int_{\mcal X} f\,\mathrm{d}(\mu\,P_s)\,\mathrm{d}s 
= \lim_{t\rightarrow \infty} \frac{1}{t}\int_0^t 
\int_{\mcal{X}} f\,\mathrm{d}(\nu\,Q_s\,L)\,\mathrm{d}s \\
& = \lim_{t\rightarrow \infty} \frac{1}{t}\int_0^t 
\int_{\mcal{Y}} (Lf)\,\mathrm{d}(\nu\,Q_s)\,\mathrm{d}s
= \int_{\mcal{Y}} (Lf)(y)\,h_2(y)\,\mathrm{d}y 
=\int_{\mcal{X}} f(x)\,\int_{\mcal{Y}} \Lambda(y,x)\,h_2(y)\,\mathrm{d}y\,\mathrm{d}x.
\end{split}
\]
Here we have used that $Lf$ is continuous (and bounded) due to the continuity of $\Lambda$, the boundedness of $f$, and the generalized dominated convergence theorem. The claim \eqref{eq:revprob} readily follows. } 

\medskip
\color{black}
Next, we show $\phat\,\langle\lhat\rangle\,\qhat$. To this end, consider continuous bounded functions $f$, $g$ on $\mcal{X}$, $\mcal{Y}$, respectively. For any fixed $t>0$, the duality relation \eqref{eq:duality}, Fubini's Theorem, and $Q\iprod{L} P$ yield
\eq\label{eq:timerev1}
\begin{split}
&\int_{\mcal{X}} (\phat_t\,\lhat\,g)(x)\,f(x)\,\mathrm{d}\nu_1(x)
= \int_{\mcal{X}} (\lhat\,g)(x)\,(P_t\,f)(x)\,h_1(x)\,\mathrm{d}x \\
&= \int_{\mcal{X}} \left(\int_{\mcal{Y}} \Lambda(y,x)\,g(y)\,h_2(y)\,\mathrm{d}y\right)(P_t\,f)(x)\,\mathrm{d}x 
= \int_{\mcal{Y}} \left(\int_{\mcal{X}} \Lambda(y,x)\,(P_t\,f)(x)\,\mathrm{d}x\right)g(y)\,h_2(y)\,\mathrm{d}y\\
&= \int_{\mcal{Y}} (L\,P_t\,f)(y)\,g(y)\,h_2(y)\,\mathrm{d}y
= \int_{\mcal{Y}} (Q_t\,Lf)(y)\,g(y)\,\mathrm{d}\nu_2(y).
\end{split}
\en
On the other hand, a similar calculation shows
\eq\label{eq:timerev2}
\begin{split}
&\int_{\mcal{X}} (\lhat\,\qhat_t\,g)(x)\,f(x)\,\mathrm{d}\nu_1(x)
= \int_{\mcal{X}} \left(\int_{\mcal{Y}} \Lambda(y,x)\,(\qhat_t\,g)(y)\,\mathrm{d}\nu_2(y)\right) f(x)\,\mathrm{d}x \\
&=\int_{\mcal{X}} \left(\int_{\mcal{Y}} (Q_t\,\Lambda)(y,x)\,g(y)\,\mathrm{d}\nu_2(y)\right) f(x)\,\mathrm{d}x
=\int_{\mcal{Y}} \left(\int_{\mcal{X}} (Q_t\,\Lambda)(y,x)\,f(x)\,\mathrm{d}x\right) g(y)\,\mathrm{d}\nu_2(y)\\
&=\int_{\mcal{Y}} (Q_t\,Lf)(y)\,g(y)\,\mathrm{d}\nu_2(y).
\end{split}
\en
Consequently, the first expressions in \eqref{eq:timerev1} and \eqref{eq:timerev2} are equal, so that $\phat\,\langle\lhat\rangle\,\qhat$. \ep

\comment{
\medskip

\nin\tbf{Step 3.} We move on to showing that $(R_t)$ and $(\rhat_t)$ are in duality with respect to $\rho$. Repeating the argument preceding \eqref{Z_transition} we find that $R_t$ and $\rhat_t$ are given by
\[
P_t(x,\mathrm{d}x')\,\frac{Q_t(y,\mathrm{d}y')\,\Lambda(y',x')}
{\int_{\mcal{Y}} Q_t(y,\mathrm{d}z)\,\Lambda(z, x')}\quad \text{and}\quad
\phat_t(x,\mathrm{d}x')\,
\frac{\qhat_t(y,\mathrm{d}y')\,\lamhat(x',y')}
{\int_{\mcal{X}} \phat_t(x,\mathrm{d}z')\,\lamhat(z',y')},
\]
respectively. 

\medskip

Now, consider the intertwining $Z_2 \iprod{L} Z_1$ with initial distribution $\rho$ and fix a $t>0$. Since $Z_2(0)$ has distribution $\nu_2$, it follows from condition \eqref{i1} of Definition \ref{idef} that $Z_2(t)$ has also distribution $\nu_2$. Since the law of $Z_1(t)$ conditional on $Z_2(t)$ is given by $L$, the joint law of $(Z_1(t),Z_2(t))$ must be given by $\rho$. Thus, $\rho$ is an invariant distribution of $(Z_1,Z_2)$. Consequently, to show the desired duality it suffices to argue that $\rhat_t$ gives the conditional law of $(Z_1(0),Z_2(0))$ given $(Z_1(t),Z_2(t))$. This can be seen by taking the joint distribution
\[
h_2(y')\,\Lambda(y',x')\,P_t(x',\mathrm{d}x)\,\frac{Q_t(y',\mathrm{d}y)}{\int_{\mcal{Y}} Q_t(y',\mathrm{d}z)\,\Lambda(z,x)}
\]
of $(Z_1(0),Z_2(0),Z_1(t),Z_2(t))$, plugging in 
\[
\int_{\mcal{Y}} Q_t(y',\mathrm{d}z)\,\Lambda(z,x)
=\frac{h_1(x)}{h_2(y')}\,\int_{\mcal{X}} \phat_t(x,\mathrm{d}z')\,\lamhat(z',y')
\]
(which is based on $Q\iprod{L}P$ and the duality of $(P_t)$ and $(\phat_t)$), and using the duality of $(P_t)$ and $(\phat_t)$, the duality of $(Q_t)$ and $(\qhat_t)$, and \eqref{eq:bden}.

}

\medskip

\nin\tbf{Simultaneous intertwining.} Exhibiting examples of intertwining among multidimensional processes is difficult. One needs to solve the equation \eqref{PDE} explicitly. The next result gives a systematic method of constructing intertwinings with multidimensional processes starting from intertwinings with one-dimensional ones. An important example of this construction, which arose originally in random matrix theory, is detailed in Section \ref{sec_DBM}.

\medskip

We ask the following question. Suppose one has diffusions $S,\,X,\,Y$ with generators given by \eqref{Sgen}, \eqref{Xgen}, \eqref{Ygen}, respectively, all satisfying Assumption \ref{main_asmp}, and stochastic transition operators $L_1,\,L_2$ with kernels $\Lambda_1,\,\Lambda_2$ such that the triplets $({\mathcal A}^S,{\mathcal A}^X,\Lambda_1)$ and $({\mathcal A}^S,{\mathcal A}^Y,\Lambda_2)$ satisfy the conditions of Theorem \ref{main1}. Can one construct a coupling $(S,X,Y)$ on a suitable probability space such that $X$ and $Y$ are conditionally independent given $S$ with $X\iprod{L_1}S$ and $Y\iprod{L_2}S$, the process $(X,Y)$ is a diffusion, and $(X,Y)\iprod{L}S$? We refer to Figure \ref{fig:simintw} for a commutative diagram representation. 

\begin{figure}[t]
\centerline{
\xymatrix@=4em{
 & X(u) \ar[dd]_<<<<<<<<<{L_1} \ar[rr] & & X(u+t) \ar[dd]^<<<<<<<<<{L_1} \\
 Y(u) \ar[dr]_{L_2} \ar[rr] & & Y(u+t) \ar[dr]_{L_2} & \\
 & S(u) \ar[rr] &  & S(u+t)
}
}
\caption{Simultaneous intertwining.}
\label{fig:simintw}
\end{figure}
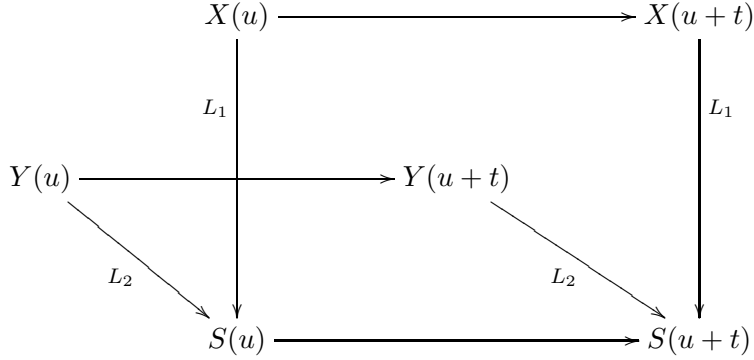

\medskip

One can take simple examples to check that this is not true in general, since the process $(X,Y)$ might not be Markovian. A consistency condition on $S$, $\Lambda_1$, $\Lambda_2$ is needed. {The answer to the above question turns out to be affirmative if the density $\Lambda_{12}(x,y,\cdot):=\Lambda_1(x,\cdot)\,\Lambda_2(y,\cdot)$ is integrable on $\mcal{S}$ and, viewed as a finite measure, satisfies
\eq\label{eq:orthokernel}
\Gamma\left(\Lambda_1(x,\cdot),\Lambda_2(y,\cdot)\right) 
:=(\gens)^*\Lambda_{12}(x,y,\cdot)
-((\gens)^*\Lambda_1(x,\cdot))\Lambda_2(y,\cdot)
-\Lambda_1(x,\cdot)(\gens)^*\Lambda_2(y,\cdot)=0
\en
for all $x\in\mcal{X}$, $y\in\mcal{Y}$ (in particular, we assume that $\Lambda_{12}(x,y,\cdot)$ is in the domain of $(\gens)^*$).} The operator $\Gamma$ is usually referred to as the carr\'e-du-champ operator and is of fundamental geometric and probabilistic importance. We refer to Section VIII.3 in \cite{RY} for an introduction and additional references. 

\begin{thm}\label{thm:combine}
{Suppose that \eqref{eq:orthokernel} holds, the total variation norm of $(\gens)^*\Lambda_{12}(x,y,\cdot)$ is locally bounded as $(x,y)$ varies in $\mcal{X}\times\mcal{Y}$, and the function 
\[
\tau(x,y):=\int_{\mcal{S}} \Lambda_{12}(x,y,s)\,\mathrm{d}s
\]
is continuously differentiable. Then, 
\begin{enumerate}[(i)]
\item $\tau$ is harmonic for ${\mathcal A}^X+{\mathcal A}^Y$ and, assuming it does not explode, the corresponding $h$-transform of the product diffusion with generator ${\mathcal A}^X+{\mathcal A}^Y$ is a Feller-Markov process on $\mcal{X}\times\mcal{Y}$ with generator
\[
{\mathcal A}^\tau ={\mathcal A}^X + {\mathcal A}^Y + (\nabla_x\log \tau)'\,a\,\nabla_x + (\nabla_y\log\tau)'\,\rho\,\nabla_y 
\]
and boundary conditions of $X$, $Y$ on $\partial\mcal{X}\times\mcal{Y}$, $\mcal{X}\times\partial\mcal{Y}$, respectively.
\item The kernel $\xi(x,y,s):=\frac{\Lambda_{12}(x,y,s)}{\tau(x,y)}$ of a stochastic transition operator $\Xi$ solves 
\[
{\mathcal A}^\tau\,\xi=({\mathcal A}^S)^*\xi\,.
\]
Moreover, if the triplet $({\mathcal A}^S,{\mathcal A}^\tau,\xi)$ satisfies the conditions of Theorem \ref{main1}, then the corresponding intertwining $(X,Y)\iprod{\Xi}S$ has the generator
\[
{\mathcal A}^S + {\mathcal A}^X + {\mathcal A}^Y + (\nabla_x\log\Lambda_1)'\,a\,\nabla_x + (\nabla_y\log\Lambda_2)'\,\rho\,\nabla_y
\]
with the boundary conditions of $S$, $X$, $Y$ on $\partial\mcal{S}\times\mcal{X}\times\mcal{Y}$, $\mcal{S}\times\partial\mcal{X}\times\mcal{Y}$, $\mcal{S}\times\mcal{X}\times\partial\mcal{Y}$, respectively, $X$ and $Y$ are conditionally independent given $S$ in that process, $(S,X)=S\iprod{L_1}X$, and $(S,Y)=S\iprod{L_2}Y$.   
\end{enumerate}}
\end{thm}

\noindent\textbf{Proof.} {Note first that, in view of ${\mathcal A}^X \Lambda_1=({\mathcal A}^S)^*\Lambda_1$, ${\mathcal A}^Y \Lambda_2=({\mathcal A}^S)^*\Lambda_2$, and \eqref{eq:orthokernel}, 
\[
({\mathcal A}^X+{\mathcal A}^Y)\,\Lambda_{12}
=({\mathcal A}^X \Lambda_1)\,\Lambda_2+\Lambda_1\,({\mathcal A}^Y\Lambda_2)
=(({\mathcal A}^S)^*\Lambda_1)\,\Lambda_2+\Lambda_1\,({\mathcal A}^S)^*\Lambda_2
=({\mathcal A}^S)^*\Lambda_{12}.
\]
Hence, according to Theorem \ref{thm:htrans} the function $\tau$ is harmonic for ${\mathcal A}^X+{\mathcal A}^Y$ and, provided it does not explode, the corresponding $h$-transform is a Feller-Markov process with the desired boundary conditions and generator given by 
\[
{\mathcal A}^\tau \phi=\tau^{-1}\,({\mathcal A}^X+{\mathcal A}^Y)(\tau\phi)
\]
on functions $\phi$ with $\tau\phi$ in the domain of ${\mathcal A}^X+{\mathcal A}^Y$.} 

\medskip

{Now, pick a function $\phi\in C_c^\infty(\mcal{X}\times\mcal{Y})$ in the domain of ${\mathcal A}^X+{\mathcal A}^Y$. Then the non-explosion of the $h$-transform shows that, for the product diffusion $(X,Y)$, the process $\tau(X(t),Y(t))$, $t\ge0$ is a martingale, so that by It\^o's formula
\eq\label{Itotauphi}
\begin{split}
(\tau\phi)(X(t),Y(t))-(\tau\phi)(X(0),Y(0))= \int_0^t \tau(X,Y)\,\mathrm{d}\phi(X,Y) &+\int_0^t \phi(X,Y)\,\mathrm{d}\tau(X,Y) \\
&+\langle \tau(X,Y),\phi(X,Y)\rangle(t).
\end{split}
\en
By Lemma \ref{lowreggito} in the appendix, we have the identity
\eq\label{QVtauphi}
\begin{split}
\langle \tau(X,Y),\phi(X,Y)\rangle(t) =\int_0^t ((\nabla \tau)'\,\kappa\,\nabla\phi)(X,Y)\,\mathrm{d}s,
\end{split}
\en
where $\kappa$ is the block matrix with blocks $a$ and $\rho$. Combining \eqref{Itotauphi}, \eqref{QVtauphi}, and the converse to Dynkin's formula (see, e.g., Proposition VII.1.7 in \cite{RY}) we conclude that $\tau\phi$ is in the domain of ${\mathcal A}^X+{\mathcal A}^Y$ with 
\[
({\mathcal A}^X+{\mathcal A}^Y)(\tau\phi)=\tau\,\mathcal A^X\phi +\tau\,{\mathcal A}^Y\phi + (\nabla_x\tau)'\,a\,\nabla_x\phi + (\nabla_y\tau)'\,\rho\,\nabla_y\phi.
\] 
This yields the desired representation of the closed operator ${\mathcal A}^\tau$, finishing the proof of (i).}

\medskip

{Using the equation $({\mathcal A}^X+{\mathcal A}^Y)\Lambda_{12}=({\mathcal A}^S)^*\Lambda_{12}$ and proceeding as in the proof of Theorem \ref{thm:htrans} (specifically, proving the analogue of \eqref{htrans_semi}), we obtain further that ${\mathcal A}^\tau\xi=({\mathcal A}^S)^*\xi$. Next, we employ the representation of the operator ${\mathcal A}^\tau$ in (i) and Theorem \ref{main1} to conclude that the intertwining $(X,Y)\iprod{\Xi}S$ has the described generator. Moreover, applying It\^o's formula to functions of $(S,X)$ ($(S,Y)$ resp.) one finds that $(S,X)$ ($(S,Y)$ resp.) is a realization of the intertwining $S\iprod{L_1}X$ ($S\iprod{L_2}Y$ resp.) via Theorem \ref{main1}. Finally, from the dynamics of $X$, $Y$ in $(S,X,Y)$ and the uniqueness for the (sub-)martingale problems associated with $S\iprod{L_1}X$, $S\iprod{L_2}Y$ it follows that, given $S$, the law of $(X,Y)$ is a product of the conditional law of $X$ given $S$ in $S\iprod{L_1}X$ and the conditional law of $Y$ given $S$ in $S\iprod{L_2}Y$. The proof of the theorem is finished.} \ep

\begin{rmk}
Theorem \ref{thm:combine} can be easily generalized to simultaneous intertwinings with any finite number of diffusions, provided the corresponding kernels jointly satisfy a product rule as in \eqref{eq:orthokernel}. 
\end{rmk}

\section{On various old and new examples} \label{sec:exmp}

\subsection{Some examples of intertwining not covered by Theorem \ref{main1}}

In \cite{CPY} the authors discuss various examples of intertwinings of Markov semigroups in continuous time. The perspective is somewhat different from ours and worth comparing. The set-up in \cite{CPY} is that of filtering. Let us first briefly describe their approach. 

\medskip

Consider two filtrations $(\mcal{F}_t:\;t\ge0)$ and $(\mcal{G}_t:\;t\ge0)$ such that $\mcal{G}_t$ is a sub-$\sigma$-algebra of $\mcal{F}_t$ for every $t$. Pick two processes: $X(t)$, $t\ge0$, which is $(\mcal{F}_t)$-adapted, and $Y(t)$, $t\ge 0$, which is $(\mcal{G}_t)$-adapted. Suppose that $X$ is Markovian with respect to $(\mcal{F}_t)$ with transition semigroup $(P_t)$, and $Y$ is Markovian with respect to $(\mcal{G}_t)$ with transition semigroup $(Q_t)$. Suppose further that there exists a stochastic transition operator $L$ such that 
\[
\ev[f(X(t))\,|\,\mcal{G}_t]=(Lf)(Y(t)),\quad t\ge 0
\] 
for all bounded measurable functions $f$. It is then shown in Proposition 2.1 of \cite{CPY} that the intertwining relation $Q_t\,L=L\,P_t$ holds for every $t\ge 0$. In the rest of the subsection we show that Theorems \ref{main1} and \ref{main2} do not cover the three major examples treated in \cite{CPY}. 


\smallskip

\begin{exm}\label{exm:dynkin} 
We start with the example in Section 2.1 of \cite{CPY} which is an instance of Dynkin's criterion for when a function of a Markov process is itself Markovian with respect to the same filtration. Take $Y$ to be an $n$-dimensional standard Brownian motion and let $X$ be its Euclidean norm. Let both $(\mcal{F}_t)$ and $(\mcal{G}_t)$ be the filtration generated by $Y$. Then the law of $X$ is that of a Bessel process of dimension $n$, and the transition operator $L$ is given by $(Lf)(y)=f\left(\abs{y}\right)$ for all bounded measurable functions $f$. 
However, $L$ does not admit a density, so that the regularity conditions in Theorem \ref{main2} do not hold. One can also see directly that the generator of the Feller-Markov process $(X,Y)$ is not of the form \eqref{Zgen}.  
\end{exm}

\begin{exm}\label{exm:pitman} 
The following example from Section 2.3 in \cite{CPY} is due to Pitman (see also \cite{RP} for similar ones). Let $B$ be a standard one-dimensional Brownian motion and take $X(t)= \abs{B(t)}$, $t\ge0$ and $Y(t)=\abs{B(t)} + \Theta(t)$, $t\ge0$ where $\Theta$ is the local time at zero of $B$. In addition, let $(\mcal{F}_t)$ and $(\mcal{G}_t)$ be the filtrations generated by $X$ and $Y$, respectively. Then, $X$ is a reflected Brownian motion and $Y$ is a Bessel process of dimension $3$. The transition operator $L$ is given by
\[
\ev[f(X(t))\,|\,\mcal{G}_t]= \int_0^1 f(x\,Y(t))\,\mathrm{d}x 
\]
for all bounded measurable functions $f$. In other words, the conditional law of $X(t)$ given $\mathcal{G}_t$ is the uniform distribution on $[0,Y(t)]$. 
Let $R$ be a $3$-dimensional Bessel process starting from zero and set $J(t)=\inf_{s\ge t} R(s), t\ge 0$. 
Then, according to Pitman's Theorem, the law of the process $(X,Y)$ is the same as that of $(R-J,R)$. Moreover, the Markov property of $R$ shows that, for any $t\ge0$, conditional on $R(t)$, the random variable $J(t)$ is independent of $R(s)$, $0\le s<t$. 
However, \eqref{Zgen} does not give the generator of $(X,Y)$. Nonetheless, \eqref{PDE} does hold for $\Lambda(y,x)=y^{-1}$ on its domain $\{(y,x)\in\rr^2:\;0<x<y\}$ in the sense specified in Theorem \ref{thm:reflect}. Indeed, $\int_0^y y^{-1}\,\frac{1}{2}\,f''(x)\,\mathrm{d}x=\frac{1}{2}\,y^{-1}\,f'(y)$ for any function $f\in C_c^\infty([0,\infty))$ with $f'(0)=0$, which is consistent with \eqref{PDE_mod} due to ${\mathcal A}^Y y^{-1}=0$.   
\end{exm}

\begin{exm}[Process extension of Beta-Gamma algebra]\label{exm} 
The primary example in \cite{CPY} (see Section 3 therein) is a process extension of the well-known Beta-Gamma algebra. For $\alpha,\beta>0$, let $X_\alpha$, $X_\beta$ be two independent squared Bessel processes of dimensions $2\alpha$, $2\beta$, respectively, both starting from zero. Set $X=X_\alpha$ and $Y=X_\alpha+X_\beta$ and define $(\mathcal{F}_t)$ and $(\mathcal{G}_t)$ as the filtrations generated by the pair $(X,Y)$ and the process $Y$, respectively. Introduce further the stochastic transition operator
\[
(L_{\alpha,\beta}f)(y) = \frac{1}{B(\alpha,\beta)} 
\int_0^1 f\left(yz\right)\,z^{\alpha-1}\,(1- z)^{\beta-1}\,\mathrm{d}z
\]
acting on bounded measurable functions on $[0,\infty)$, where $B(\cdot,\cdot)$ is the Beta function. Clearly, the transition kernel corresponding to $L$ is given by 
\eq\label{eq:lambg}
\Lambda_{\alpha,\beta}(y,x)= \frac{y^{-1}}{B(\alpha,\beta)} \left( \frac{x}{y}\right)^{\alpha-1} \left( 1 - \frac{x}{y}  \right)^{\beta-1}\,\mathbf{1}_{(0,y)}(x).
\en
Theorem 3.1 in \cite{CPY} proves the intertwining $Q_t\,L_{\alpha,\beta} = L_{\alpha, \beta}\,P_t$, $t\geq0$ of the semigroups $(P_t)$ and $(Q_t)$ associated with $X$ and $Y$. 

\medskip 

In the course of the proof of Theorem 3.1 in \cite{CPY} the authors verify condition (iv) of our Definition \ref{idef} (see the display in the middle of page 325 therein). However, \eqref{eq:condlaw} cannot hold for the pair $(X,Y)$, and it is easy to see from the SDEs for $X_\alpha$, $X_\beta$ that the generator of $(X,Y)$ is not given by \eqref{Zgen}. Indeed, Theorem 1 cannot be used to construct intertwinings $(X,Y)$ with non-trivial covariation between $X$ and $Y$. Nonetheless, $\Lambda_{\alpha,\beta}$ does solve \eqref{PDE} on its domain $\{(y,x)\in\rr^2:\;0<x<y\}$ in the sense specified in Theorem \ref{thm:reflect}. Indeed, considering $\int_0^y \Lambda_{\alpha,\beta}(y,x)\,(2\alpha\,f'(x)+2x\,f''(x))\,\mathrm{d}x$ for a function $f\in C^\infty_c([0,\infty))$ and integrating by parts one obtains 
\[
\begin{split}
& \int_0^y \frac{2(\beta-1)}{B(\alpha,\beta)}
\,x^{\alpha-1}\,y^{1-\alpha-\beta}\,(y-x)^{\beta-3}\,\big((\alpha+\beta-2)x-\alpha\,y\big)\,f(x)\,\mathrm{d}x \\
& + \Big(2\alpha\,\Lambda_{\alpha,\beta}(y,x)\,f(x)+2x\,\Lambda_{\alpha,\beta}(y,x)\,f'(x)-\partial_x(2x\,\Lambda_{\alpha,\beta}(y,x))\,f(x)\Big)\Big|_0^y\,.
\end{split}
\]
On the other hand, by direct differentiation one verifies
\[
\mcal{A}^Y\,\Lambda_{\alpha,\beta}(y,x)=\frac{2(\beta-1)}{B(\alpha,\beta)}
\,x^{\alpha-1}\,y^{1-\alpha-\beta}\,(y-x)^{\beta-3}\,\big((\alpha+\beta-2)x-\alpha\,y\big),
\]
and the boundary terms are consistent with those in \eqref{PDE_mod} (up to the non-trivial diffusion coefficient in this example). 
\end{exm}

\subsection{Whittaker $2d$-growth model}\label{sec_Whittaker}

The following is an example of intertwined diffusions that appeared in the study of a semi-discrete polymer model in \cite{Oc}. The resulting processes were investigated further in \cite{BC} under the name \textit{Whittaker $2d$-growth model}. In the latter article, it is shown that such processes arise as diffusive limits of certain intertwined Markov chains which are constructed by means of Macdonald symmetric functions.   

\medskip

Fix some $N\in\nn$ and $a=(a_1,a_2,\ldots,a_N)\in\rr^N$ and consider the diffusion process $R=\big(R_{i}^{(k)},\;1\leq i\leq k\leq N\big)$ on $\rr^{N(N+1)/2}$ defined through the system of SDEs 
\eq
\begin{split}
&\mathrm{d}R^{(1)}_1(t)=\mathrm{d}W^{(1)}_1(t)+a_1\,\mathrm{d}t, \\
&\mathrm{d}R^{(k)}_1(t)=\mathrm{d}W^{(k)}_1(t)+\left(a_k+e^{R^{(k-1)}_1(t)-R^{(k)}_1(t)}\right)\,\mathrm{d}t, \\
&\mathrm{d}R^{(k)}_2(t)=\mathrm{d}W^{(k)}_2(t)+\left(a_k+e^{R^{(k-1)}_2(t)-R^{(k)}_2(t)}-e^{R^{(k)}_2(t)-R^{(k-1)}_1(t)}\right)\,\mathrm{d}t, \\
&\vdots\\
&\mathrm{d}R^{(k)}_{k-1}(t)=\mathrm{d}W^{(k)}_{k-1}(t)+\left(a_k+e^{R^{(k-1)}_{k-1}(t)-R^{(k)}_{k-1}(t)}
-e^{R^{(k)}_{k-1}(t)-R^{(k-1)}_{k-2}(t)}\right)\,\mathrm{d}t, \\
&\mathrm{d}R^{(k)}_k(t)=\mathrm{d}W^{(k)}_k(t)+\left(a_k-e^{R^{(k)}_k(t)-R^{(k-1)}_{k-1}(t)}\right)\,\mathrm{d}t,
\end{split}
\en
where $\big(W^{(k)}_i,\; 1\leq i\leq k\leq N\big)$ are independent standard Brownian motions. 

\medskip

Define the following two functions acting on vectors $r=\big(r_i^{(k)},\; 1\le i\le k\le N\big)$ in $\rr^{N(N+1)/2}$:
\[
\begin{split}
T_1(r)&=\sum_{k=1}^N a_k\bigg(\sum_{i=1}^k r^{(k)}_i - \sum_{i=1}^{k-1} r^{(k-1)}_i\bigg),\\
T_2(r)&= \sum_{1\leq i\leq k\leq N-1} \Big[ \exp\big(r^{(k)}_i-r^{(k+1)}_i\big) + \exp\big(r^{(k+1)}_{i+1}-r^{(k)}_i\big)\Big].
\end{split}
\]
Let $X$ be the diffusion process on $\rr^{(N-1)N/2}$ comprised by the coordinates $R_{i}^{(k)}$, $1\leq i\leq k\leq N-1$, write ${\mathcal A}^X$ for its generator, and let $Y$ be the diffusion on $\rr^N$ with generator given by
\eq
\begin{split}
&{\mathcal A}^Y = \frac{1}{2}\,\Delta+(\nabla\log\psi_a(y))\cdot\nabla,\\
&\psi_a(y) = \int_{\rr^{(N-1)N/2}} \exp\left(T_1(r)
- T_2(r)\right) \mathrm{d}r^{(1)}_1\ldots\,\mathrm{d}r^{(N-1)}_{N-1}\Big|_{r^{(N)}_1=y_1,\ldots,r^{(N)}_N=y_N}.
\end{split}
\en
As observed in Theorem 3.1 of \cite{Oc}, the generator ${\mathcal A}^Y$ can be rewritten as
\eq\label{WhittakerAY}
\frac{1}{2}\,\psi_a(y)^{-1}\,\left(H-\sum_{i=1}^N a_i^2\right)\,\psi_a(y),
\en
where $H=\Delta-2\sum_{i=1}^{N-1} e^{y_{i+1}-y_i}$ is the operator known as the \textit{Hamiltonian of the quantum Toda lattice} (see Section 2 of \cite{Oc} and the references therein for more details on the latter). 

\medskip

Let $x=(x_i^{(k)},\; 1\le i \le k \le N-1)$ be a vector in $\rr^{(N-1)N/2}$ and $y$ be a vector in $\rr^N$. One can naturally concatenate $y$ ``above'' $x$ to get a vector $r\in\rr^{N(N+1)/2}$. Consider the stochastic transition kernel
\[
\Lambda(y,x)=\frac{1}{\psi_a(y)}\exp\big(T_1(r) - T_2(r)\big). 
\]
The formulas for ${\mathcal A}^Y$ and $\Lambda$ show that the generator of $R$ is of the form \eqref{Zgen}. Moreover, the statement that $\Lambda$ solves \eqref{PDE} in the sense specified in Theorem \ref{main1} is implicitly contained in Section 9 of \cite{Oc} (see also Proposition 8.2 and, in particular, equation (12) therein for a related statement). Therefore we expect the Whittaker $2d$-growth model to be an instance of the construction in Theorem \ref{main1}, even though the detailed analysis of the function $\psi_a$ needed for the verification of the regularity conditions in Theorem \ref{main1} is a significant technical challenge. 

\subsection{Constructing new examples} 

The main difficulty in constructing intertwining relationships consists in finding explicit solutions of \eqref{PDE} that are positive. Even in the case that one of the two diffusions is one-dimensional, in which semigroup theory can be used to prove the existence of solutions, showing their positivity is not easy. In this subsection we construct several classes of positive solutions. 

\medskip

\noindent\textbf{Diffusions on compact state spaces.} {Suppose that the state spaces $\mcal{X}$, $\mcal{Y}$ of the diffusions $X$, $Y$ are compact, and that $X$ has an invariant distribution on $\mcal{X}$ with a positive continuous density $f$. A simple example of such a diffusion is a normally reflected Brownian motion on a compact domain, in which case $f$ is constant. Let $u$ be a continuous function that solves \eqref{PDE} on the compact $\mcal{X}\times\mcal{Y}$. Then there is a large enough constant $M$ such that $u+Mf$ is a positive solution of \eqref{PDE} (note that $({\mathcal A}^X)^*f=0$). Clearly, $u+Mf$ gives rise to an intertwining via Theorem \ref{thm:htrans}.} 

\medskip

{One might wonder how the choice of $M$ affects the resulting intertwining relationship. Assuming that $\tau(y):=\int_{\mcal{X}} u(y,x)\,\mathrm{d}x$ is continuously differentiable in $y$, the generator of the $h$-transform of $Y$ associated with $u+Mf$ via Theorem \ref{thm:htrans} reads
\[
{\mathcal A}^{\tau,M}
:={\mathcal A}^Y+\big(\nabla\log(\tau+M)\big)'\rho\,\nabla_y
={\mathcal A}^Y+\frac{(\nabla\tau)'}{\tau+M}\,\rho\,\nabla_y.
\]
If, in addition, the triplet $({\mathcal A}^X,{\mathcal A}^{\tau,M},u+Mf)$ satisfies the conditions of Theorem \ref{main1}, then the generator of the corresponding intertwining is given by
\[
{\mathcal A}^X+{\mathcal A}^Y+\bigg(\frac{(\nabla\tau)'}{\tau+M}
+\frac{(\nabla_y u)'}{u+Mf}\bigg)\,\rho\,\nabla_y.
\]
Consequently, different choices of $M$ lead to non-trivial changes in ${\mathcal A}^{\tau,M}$ and the latter generator, as well as in the corresponding diffusions.} 
 
\medskip

{For an example of this construction consider $\mcal{X}=\mcal{Y}=[-1,1]$ and take 
\[
{\mathcal A}^X=-2x\,\partial_x+(1-x^2)\partial_x^2,\quad 
{\mathcal A}^Y=(1-2y)\partial_y+(1-y^2)\partial_y^2.
\]
The corresponding processes $X$, $Y$ are examples of Jacobi (or, Wright-Fisher) diffusions. The latter play an important role in population genetics. The operator $({\mathcal A}^X)^*$, viewed as a differential operator acting on twice continuously differentiable functions on $[-1,1]$, coincides with ${\mathcal A}^X$ and admits eigenfunctions $(f_q)_{q\in\nn}$ with eigenvalues $q(q+1)$, $q\in\nn$ which are known as Legendre polynomials. The eigenfunctions $(g_q)_{q\in\nn}$ of the operator ${\mathcal A}^Y$ are known as Jacobi polynomials, and the corresponding eigenvalues are also given by $q(q+1)$, $q\in\nn$. Consequently, $u(y,x)=\sum_{q\in\nn} c_q\,f_q(x)\,g_q(y)$ is a solution of \eqref{PDE} whenever $\sum_{q\in\nn} |c_q|\,\|f_q\|_\infty\,\|g_q\|_\infty<\infty$ and $\sum_{q\in\nn} |c_q|\,q(q+1)\|f_q\|_\infty\,\|g_q\|_\infty<\infty$. Moreover, the uniform distribution on $[-1,1]$ is invariant for $X$. Thus, the functions $\frac{M}{2}+\sum_{q\in\nn} c_q\,f_q(x)\,g_q(y)$ are positive solutions of \eqref{PDE} for all $M>2\sum_{q\in\nn} |c_q|\,\|f_q\|_\infty\,\|g_q\|_\infty$ and give rise to intertwinings of $X$ with $h$-transforms of $Y$ as described above.} 
 
\medskip

\noindent\textbf{Intertwinings of multidimensional Brownian motions with $h$-transforms of Bessel processes.} The following lemma is well-known and is usually used to solve the classical wave equation in multiple space dimensions. For its proof we refer to the proof of Lemma 1 on page 71 in \cite{Ev}. 

\begin{lemma}
Let $u$ be a positive twice continuously differentiable probability density on $\rr^m$ with $m>1$. Let $\gamma_m=\pi^{m/2}/\Gamma(1+m/2)$ denote the volume of the unit ball in dimension $m$. For $r>0$ and $x\in\rr^m$, define the spherical means of $u$ by 
\eq\label{eq:sphmean}
\Lambda(r,x)=\frac{1}{m\gamma_m}\int_{\partial B(0,1)} u\left(x+rz\right)\,\mathrm{d}\theta(z),
\en
where $B(0,1)$ is the unit ball centered at $0$, and $\theta$ is the Lebesgue measure on its boundary. Then, $\Lambda(r,x)$ is positive and a classical solution of
\eq\label{eq:besbm}
\frac{m-1}{2r}\,\partial_r\,\Lambda(r,x) + \frac{1}{2}\,\partial_r^2\,\Lambda(r,x)=\frac{1}{2}\,\Delta_x\,\Lambda(r,x).
\en
\end{lemma}

By Fubini's Theorem the kernel $\Lambda(r,x)$ is stochastic. This allows us to use Theorem \ref{main1} to construct intertwinings of multidimensional Brownian motions with Bessel processes of the same dimension. 
Note that such intertwinings are different from the one in Example \ref{exm:dynkin}, since for any given $r>0$ the density $\Lambda(r,\cdot)$ is  supported on the entire $\rr^m$.

\medskip

More generally, positive classical solutions of \eqref{eq:besbm} give rise to intertwinings of multidimensional Brownian motions with $h$-transforms of Bessel processes of the same dimension via Theorem \ref{thm:htrans}. Hereby, the possible $h$-transforms are characterized by the following proposition.  

\begin{prop}\label{prop:cons}
Let $\Lambda(r,x)$ be a positive, classical solution of \eqref{eq:besbm} with $m>1$. Suppose that $\int_{\rr^m} |\Delta_x \Lambda(r,x)|\,\mathrm{d}x$ is locally bounded as $r$ varies, and that the integral $\tau(r):=\int_{\rr^m} \Lambda(r,x)\,\mathrm{d}x$ is finite for all $r>0$ and continuous in $r$. Then, there exist constants $a,b\in\rr$ such that $\tau(r)=a+b\,r^{2-m}$ if $m>2$ and $\tau(r)=a+b\,\log r$ if $m=2$. In particular, if $\limsup_{r\downarrow0} |\tau(r)| < \infty$, then $\tau(r)$ is a constant. 
\end{prop}

\noindent\textbf{Proof.} {The regularity conditions on $\Lambda$ allow us to conclude that $\tau$ is harmonic for $\frac{m-1}{2r}\,\partial_r + \frac{1}{2}\,\partial_{rr}$ (see Theorem \ref{thm:htrans} and its proof). The proposition now follows from the remark at the bottom of p. 303 in \cite{RY} and the formulas for scale functions of Bessel processes in Section XI.1 of \cite{RY}.} \ep

\begin{rmk}\label{rmk:harfn}
The statement and the proof of Proposition \ref{prop:cons} readily extend to any one-dimensional diffusion instead of a Bessel process. All possible harmonic functions with respect to its generator are then given by affine transformations of a scale function of the process. For more details on scale functions we refer the reader to Section VII.3 in \cite{RY}. 
\end{rmk}

\medskip

\nin\tbf{$\sigma$-finite kernels.} In some cases $\sigma$-finite kernels can be combined to obtain finite ones via the procedure described in Theorem \ref{thm:combine}. As an example consider an orthonormal basis $\zeta_1,\zeta_2,\ldots,\zeta_k$ of $\rr^k$. Pick $k$ positive probability density functions  $f_1,f_2,\ldots,f_k$ on $\rr$ that are twice continuously differentiable, tend to zero at infinity together with their second derivatives, and whose second derivatives are integrable. Then, the $\sigma$-finite kernels
\[
\Lambda_i(x_i,s):=f_i(x_i+\iprod{s,\zeta_i}),\quad i=1,2,\ldots,k 
\]  
are classical solutions of $\Delta_s\Lambda_i=\partial_{x_i}^2\Lambda_i$. With $\Lambda(x,s):=\prod_{i=1}^k \Lambda(x_i,s)$, the orthonormality of the $\zeta_i$'s yields 
\[
\Delta_s\Lambda(x,s)=\sum_{j=1}^k \partial_{x_j}^2\Lambda_j(x_j,s)\,\prod_{i\neq j} \Lambda_i(x_i,s)=\Delta_x \Lambda(x,s)
\] 
in the classical sense and in the sense of Theorem \ref{main1}. Moreover, the kernel $\Lambda$ is stochastic and, hence, gives rise to an intertwining of two Brownian motions via Theorem \ref{main1}, provided the corresponding diffusion satisfies Assumption \ref{main_asmp}. 

\section{Interwinings of diffusions with reflections} \label{sec:reflection}

\subsection{Multilevel Dyson Brownian motion} \label{sec_DBM}

The following example is the main subject of study in \cite{W}. Consider the so-called \textit{Gelfand-Tsetlin cone}
\eq\label{GTc}
\overline{{\mathcal G}^N}:=\Big\{r=\big(r_{i}^{(k)}:\,1\leq i\leq k\leq N\big)\in\rr^{N(N+1)/2}:\;r^{(k-1)}_{i-1}\leq r_{i}^{(k)}\leq r_{i}^{(k-1)}\Big\}
\en 
for some $N\in\nn$, $N\ge2$. An element $r\in\overline{{\mathcal G}^N}$ is usually thought of in terms of the pattern of points $\big(r_i^{(k)},k\big)$, $1\leq i\leq k\leq N$ in the plane (see Figure \ref{GTpattern} for an illustration). 
\begin{figure}[t]\label{GTpattern}
\begin{centering}
\begin{tikzpicture}[scale=0.75]
\draw [help lines] (0,0) grid (9,5);
\draw [thick, <->] (0,6) -- (0,0) -- (10,0);
\draw [fill] (4,0) circle [radius=0.08];
\draw [fill] (3,1) circle [radius=0.08];
\draw [fill] (5,1) circle [radius=0.08];
\draw [fill] (2,2) circle [radius=0.08];
\draw [fill] (4,2) circle [radius=0.08];
\draw [fill] (7,2) circle [radius=0.08];
\draw [fill] (1,3) circle [radius=0.08];
\draw [fill] (3,3) circle [radius=0.08];
\draw [fill] (6,3) circle [radius=0.08];
\draw [fill] (8,3) circle [radius=0.08];
\draw [fill] (0,4) circle [radius=0.08];
\draw [fill] (2,4) circle [radius=0.08];
\draw [fill] (5,4) circle [radius=0.08];
\draw [fill] (6,4) circle [radius=0.08];
\draw [fill] (9,4) circle [radius=0.08];
\end{tikzpicture}
\end{centering}
\caption{An illustration of an element $r\in\overline{{\mathcal G}^N}.$}
\end{figure}
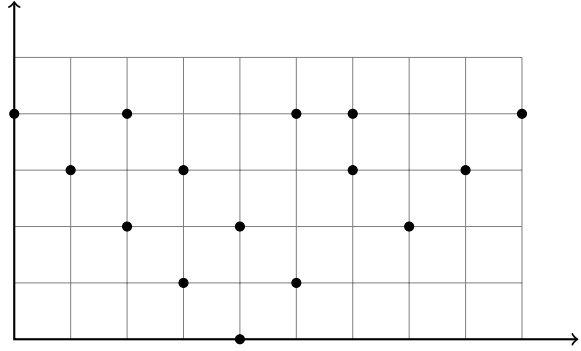
In \cite{W} the author defines a diffusion $R$ in $\overline{{\mathcal G}^N}$ through the system of SDEs
\eq\label{DBM}
\mathrm{d}R_{i}^{(k)}(t)=\mathrm{d}W^{(k)}_i(t)+\mathrm{d}L^{(k),+}_i(t)-\mathrm{d}L^{(k),-}_i(t),\quad 1\leq i\leq k\leq N,
\en 
equipped with the initial condition $R(0)=0\in\overline{{\mathcal G}^N}$ and 
entrance laws into $\overline{{\mathcal G}^N}$ whose probability densities are multiples of 
\eq\label{beta_corners}
\prod_{1\leq i<j\leq N} \big(r_j^{(N)}-r_i^{(N)}\big) 
\prod_{i=1}^N \exp\bigg(- \frac{\big(r_i^{(N)}\big)^2}{2t}\bigg),
\quad t>0.
\en
Here $L^{(k),\pm}_i$ are the local times accumulated at zero by the semimartingales $R_{i}^{(k)}-R^{(k-1)}_{i-1}$, $R_{i}^{(k-1)}-R_{i}^{(k)}$, respectively. The probability distributions given by \eqref{beta_corners} are of major importance in random matrix theory, as each of them describes the joint law of the eigenvalues of the top left $1\times1,\,2\times2,\,\ldots,\,N\times N$ submatrices of a (scaled) matrix from the Gaussian unitary ensemble (GUE). The diffusion $R$ is usually referred to as the \textit{multilevel Dyson Brownian motion}, or as the \textit{Warren process}. 
 
\medskip

{Write $X$ for $\big(R^{(k)}_i:\,1\le i\le k\le N-1\big)$ and $Y$ for $\big(R^{(N)}_i:\,1\le i\le N\big)$. It is clear that $X$ forms a multilevel Dyson Brownian motion in $\overline{{\mathcal G}^{N-1}}$. The main result of \cite{W} establishes that $Y$ is also a diffusion in its own filtration, namely an $N$-dimensional Dyson Brownian motion. Specifically, there exist independent standard Brownian motions $B_1,B_2,\ldots,B_N$ with respect to the filtration of $Y$ such that
\eq\label{DBM_SDE}
\mathrm{d}Y_j(t)=\sum_{l\neq j} \frac{1}{Y_j(t)-Y_l(t)}\,\mathrm{d}t+\mathrm{d}B_i(t),\quad j=1,2,\ldots,N.
\en
Moreover, the explicit description of the entrance laws through the formula \eqref{beta_corners} is used in \cite{W} to prove the intertwining of the semigroups of $X$ and $Y$.} 

\medskip

We show now that the process $R$ fits into the framework of our Theorem \ref{thm:reflect},  although we are unable to check the technical condition that an appropriate subset of $C_c^{\infty}(\overline{\mcal{G}^N})$ is a core for the domain of $R$. \comment{hence, the results of \cite{W} become corollaries of that more general result.}Indeed, consider $R(t)$, $t\ge t_0$ for some $t_0>0$. The state space of this process is
\[
D^{(N)}=\big\{r\in\overline{{\mathcal G}^N}:\;r^{(k)}_i<r^{(k)}_{i+1},\,1\le i<k\le N\big\},
\]
and we have the cross-sections
\[
D^{(N)}(y)=\big\{x\in D^{(N-1)}:\;y_1\le x^{(N-1)}_1\le y_2\le x^{(N-1)}_2\le\cdots\le x^{(N-1)}_{N-1}\le y_N \big\}
\]
for $y\in\rr^N$ with $y_1<y_2<\cdots<y_N$. The appropriate kernel $\Lambda$ for the case at hand turns out to be
\[
\Lambda(y,x)=\prod_{k=1}^{N-1} k!\,\prod_{1\le j<l\le N} (y_l-y_j)^{-1}\,\mathbf{1}_{D^{(N)}(y)}(x). 
\]
The stochasticity of $\Lambda$ can be checked by induction over $N$ relying on the identity
\[
\int_{y_1}^{y_2} \ldots \int_{y_{N-1}}^{y_N}\!
(N-1)!\!\prod_{1\leq i<m\leq N-1}\! \big(x^{(N-1)}_m-x^{(N-1)}_i\big)
\!\prod_{1\leq j<l\leq N}\! (y_l-y_j)^{-1}
\mathrm{d}x^{(N-1)}_1\ldots\mathrm{d}x^{(N-1)}_{N-1}=1.
\]
The latter integrand usually goes by the name \textit{Dixon-Anderson conditional probability density} and, in particular, its integral is known to be equal to $1$ (see, e.g., the introduction in \cite{Fo}). It is clear from the definitions that $\Lambda$ is positive and smooth on $D$, and that the corresponding operator $L$ maps $C_0(D^{(N-1)})$ to $C_0(\{y\in\rr^N:\,y_1<y_2<\cdots<y_N\})$.  

\medskip

{Next, we note that the submartingale problem associated with $R(t)$, $t\ge t_0$ is well-posed and that its solution is a Feller-Markov process, since any solution of it can be viewed as a reflected Brownian motion in $D^{(N)}$ and must therefore be given by the image of the driving Brownian motions under the appropriate (deterministic and Lipschitz) reflection map. Moreover, $\Lambda(\cdot,x)$ extends to the function $\tilde{\Lambda}(y)=\prod_{k=1}^{N-1} k!\,\prod_{1\le j<l\le N} (y_l-y_j)^{-1}$ and the latter satisfies ${\mathcal A}^Y\tilde{\Lambda}=0$ where ${\mathcal A}^Y$ is the generator of the Dyson Brownian motion $Y$ interpreted as a differential operator. We now obtain the representation \eqref{PDE_mod} via Remark \ref{PDE_simpl_rmk} after noting that here $({\mathcal A}^X)^*$ (interpreted as a differential operator) is one half times the Laplacian on $D^{(N)}(y)$, so that $({\mathcal A}^X)^*\Lambda(y,\cdot)=0$ on $D^{(N)}(y)$. It is also straightforward to check that both terms on the left-hand side of \eqref{LambdaNeumann} and the paranthesis on the right-hand side of \eqref{LambdaNeumann} vanish identically.

In order to check condition (iv) of Assumption \ref{Asmp3}, fix a $y \in \mathbb{R}^N$ satisfying $y_1 < \cdots < y_N$. Recall that when started from $y$, $Y$ can be viewed as an $h$-transform of a Brownian motion killed upon exiting the state space of $Y$ (see, e.g., Section 2.1 in \cite{Bi2}). We recognize $\frac{\tilde{\Lambda}(Y(t))}{\tilde{\Lambda}(y)}$ as the density of the law of the killed Brownian motion on $[0,t]$ with respect to the law of Dyson Brownian motion on $[0,t]$. Denote the law of the killed Brownian motion started from $y$ as $\tilde{\mathbb{P}}_y$. Define $V(x) = \prod_{1 \le j < l \leq N}|x_l - x_j|$ and define $\tau$ as the first time $Y_i(t) = Y_{i+1}(t)$ for some $i=1,\ldots,N-1$. Fix some small $\epsilon >0$ and note
\begin{equation}\label{DBMest}
\begin{split}
\mathbb{E}_y[\tilde{\Lambda}(Y(t))^{1+\epsilon}] &= C\tilde{\Lambda}(y) \tilde{\mathbb{E}}_y[V(Y(t))^{-\epsilon}\mathbf{1}_{\{\tau > t\}}]
\\ &\leq C_y \mathbb{E}[V(B(t)+y)^{-\epsilon}]
\\ &\leq C_{y,N}\sum_{j \neq i}\mathbb{E}\left[|B_i(t)-B_j(t) - y_i + y_j|^{-\epsilon \frac{N(N-1)}{2}}\right]\\
&\leq \tilde{C}_{y,N}\sum_{j \neq i}\mathbb{E}\left[|B_i(t)-B_j(t)|^{-\epsilon\frac{N(N-1)}{2}} \right] + \tilde{C}_{y,N},
\end{split}
\end{equation}
where $B$ is a standard Brownian motion. We have used the AM-GM inequality and the bound $(\sum_{i=1}^n |a_i|)^p \leq n^{p-1}\sum_{i=1}^n |a_i|^p$ for the second inequality. Up to a factor of $t^{-\frac{\epsilon}{2}\frac{N(N-1)}{2}}$, we may replace $B(t)$ by a standard Gaussian vector in the bottom expression in \eqref{DBMest}. This expectation is readily checked to be finite for small enough $\epsilon$, and so we have checked condition (iv).

At this point, up to checking that the intersection of $C_c^{\infty}(D^{(N)})$ with the domain of $R$ is a core for the domain of $R$, we may apply Theorem \ref{thm:reflect} to obtain $R=Y\iprod{L}X$ on $[t_0,\infty)$. In particular, we recover the results of \cite{W} by taking the limit $t_0\downarrow0$.}

\subsection{$\sigma$-finite kernels}

{In this subsection, we explain how the kernel of the previous subsection can be obtained by combining suitable $\sigma$-finite kernels via the procedure described in Theorem \ref{thm:combine}. Let ${\mathcal A}^X$ be the generator of the process $X:=\big(R^{(k)}_i:\;1\le i\le k\le N-1\big)$ defined in the previous subsection. In other words, ${\mathcal A}^X$ is one half times the Laplacian on $D^{(N-1)}$, endowed with Neumann boundary conditions dictated by \eqref{DBM}. In addition, abbreviate $\frac{1}{2}\,\frac{\mathrm{d}^2}{\mathrm{d}y_i^2}$ by ${\mathcal A}^{Y_i}$ for $i=1,\,2,\,\ldots,\,N$ and define the regions
\[
\begin{split}
& D_1^{(N)}(y_1)=\big\{x\in D^{(N-1)}:\;x^{(N-1)}_1 \ge y_1 \big\}, \\
& D_i^{(N)}(y_i)=\big\{x\in D^{(N-1)}:\;x^{(N-1)}_{i-1} \le y_i \le x^{(N-1)}_i \big\}\quad\text{for}\quad i=2,\,3,\,\ldots,\,N-1, \\
& D_N^{(N)}(y_N)=\big\{x\in D^{(N-1)}:\;x^{(N-1)}_{N-1} \le y_N \big\}.
\end{split}
\]
Then, for each $i=1,\,2,\,\ldots,\,N$, the $\sigma$-finite kernel $\Lambda_i(y_i,x)=\mathbf{1}_{D_i^{(N)}(y_i)}(x)$ trivially satisfies $({\mathcal A}^X)^*\Lambda_i={\mathcal A}^{Y_i}\Lambda_i$ on $\cup_{y_i} \big(\{y_i\}\times D_i^{(N)}(y_i)\big)$ in the classical sense (with $({\mathcal A}^X)^*$ being interpreted as a differential operator).} 

\medskip

Next, combine the $\sigma$-finite kernels $\Lambda_i$, $i=1,\,2,\,\ldots,\,N$ according to the recipe of Theorem \ref{thm:combine} to obtain the finite kernel 
\[ 
\prod_{i=1}^N \mathbf{1}_{D_i^{(N)}(y_i)}(x)=\mathbf{1}_{D^{(N)}(y)}(x)
\]
where $D^{(N)}(y)$ is defined as in the previous subsection.\hspace{-2.5pt} Theorem \ref{thm:combine} suggests that the normalizing function
\[
\tau(y):=\int_{D^{(N-1)}} \mathbf{1}_{D^{(N)}(y)}(x)\,\mathrm{d}x
\]
should be harmonic for $\sum_{i=1}^N {\mathcal A}^{Y_i}=\frac{1}{2}\,\Delta_y$. Indeed, as in the previous subsection one finds 
\[
\tau(y)=\bigg(\prod_{k=1}^{N-1} k!\bigg)^{-1}\prod_{1\le j<l\le N} (y_l-y_j)\,\mathbf{1}_{\{y:\,y_1<y_2<\cdots<y_N\}},
\]
and the latter function is harmonic for $\frac{1}{2}\,\Delta_y$ on $\{y:\,y_1<y_2<\cdots<y_N\}$. The corresponding $h$-tranform of $\frac{1}{2}\,\Delta_y$ gives rise to the generator of the $N$-dimensional Dyson Brownian motion $Y$ from \eqref{DBM_SDE} (see, e.g., Section 2.1 in \cite{Bi2} for more details). It remains to observe that the normalized kernel $\frac{\mathbf{1}_{D^{(N)}(y)}(x)}{\tau(y)}$ is precisely the stochastic kernel employed in the previous subsection.
 
\medskip 

\comment{One might think that it could be possible to obtain the multilevel Dyson Brownian motion with $N$ levels from the multilevel Dyson Brownian motion with $N-1$ levels by combining intertwinings with $N$ suitable one-dimensional diffusions into a simultaneous intertwining, thereby circumventing the use of $\sigma$-finite kernels. In the following we explain that already for $N=2$ this is \textit{not} the case. When $N=2$, the process $R=(R^{(1)}_1,R^{(2)}_1,R^{(2)}_2)$ consists of two building blocks: $(R^{(1)}_1,R^{(2)}_1)$ given by a standard Brownian motion reflected downwards on an indepedent standard Brownian motion and $(R^{(1)}_1,R^{(2)}_2)$ given by a standard Brownian motion reflected upwards on an independent standard Brownian motion. Consider the second building block, for example, and recall that it solves the system of SDEs    
\eq\label{BuBl}
\mathrm{d}R^{(1)}_1(t)=\mathrm{d}W^{(1)}_1(t), \quad
\mathrm{d}R^{(2)}_2(t)=\mathrm{d}W^{(2)}_2(t)+\mathrm{d}L^{(2),+}_2(t). 
\en

If the diffusion $(R^{(1)}_1,R^{(2)}_2)$ could be realized as an intertwining of a standard Brownian motion with a suitable one-dimensional diffusion with semigroup $(Q_t)$, then the transition semigroup of $(R^{(1)}_1,R^{(2)}_2)$ would be given by
\[
\varphi_t(x_1-x_0)\,\frac{Q_t(y_0,y_1)\,\Lambda(y_1,x_1)}{\int_\rr Q_t(y_0,\mathrm{d}\tilde{y}_1)\,\Lambda(\tilde{y}_1,x_1)},\quad t>0
\]
(see \eqref{Z_transition}) where $\varphi_t$ is the Gaussian density with mean $0$ and variance $t$, and $\Lambda$ is the intertwining kernel. On the other hand, the transition densities of the solution to \eqref{BuBl} are computed in Proposition 8 of \cite{W} to
\[
\varphi_t(x_1-x_0)\,\varphi_t(y_1-y_0)-\varphi_t'(y_1-x_0)\,\Phi_t(x_1-y_0),\quad t>0
\]
where $\Phi_t$ is the Gaussian cumulative distribution function with mean $0$ and variance $t$. Equating the two expressions and rearranging one obtains
\[
\frac{\varphi_t'(y_1-x_0)\,\Phi_t(x_1-y_0)}{\varphi_t(x_1-x_0)}
=\varphi_t(y_1-y_0)-\frac{Q_t(y_0,y_1)\,\Lambda(y_1,x_1)}{\int_\rr Q_t(y_0,\mathrm{d}\tilde{y}_1)\,\Lambda(\tilde{y}_1,x_1)},\quad t>0.
\]
The left-hand side of the latter equation depends on $x_0$ and the right-hand side does not, resulting in a contradiction. 

We end the subsection by remarking that in \cite{SW}, \cite{STW} the authors did construct an intertwining bearing resemblance to the system of SDEs \eqref{BuBl}. The SDE for $R^{(1)}_1$ is replaced therein by the SDE of a Brownian bridge ending at $0$, say, $R^{(1)}_1(0)$ is chosen according to the standard normal distribution, and $R^{(2)}_2$ is an independent standard Brownian motion reflected upwards on the bridge if $R^{(2)}_2(0)>R^{(1)}_1(0)$ and downwards if $R^{(2)}_2(0)<R^{(1)}_1(0)$. The main result of \cite{SW}, \cite{STW} states that $R^{(2)}_2$ is a standard Brownian motion in its own filtration. Since the Brownian bridge is a time-inhomogeneous Markov process, to put this example into the context of intertwining one needs to consider a time-dependent intertwining kernel. Such an extension of our theory is beyond the scope of the present paper, and we leave it for future research.} 


\comment{
\section{Intertwining and convergence rate of diffusion semigroups}\label{sec:rates}

In this section, we show how intertwinings can be used to turn questions about the rate of convergence of diffusion semigroups into controllability questions for equations of the form \eqref{PDE}. We measure the distance to stationarity using the continuous space version of the Aldous-Diaconis separation distance
\eq
d(\alpha,\nu):=\sup_{A\subset{\mathcal X}} \bigg(1-\frac{\alpha(A)}{\nu(A)}\bigg).
\en
Here the supremum is taken over all Borel subsets of ${\mathcal X}$. We expect the following theorem to be particularly useful in the case of non-reversible diffusions, in which very few methods for the study of convergence rates are available.

\begin{thm}\label{conv_thm}
In the setting of Theorem \ref{main1} suppose that $X$ has an invariant distribution $\nu$ and that the density of $\nu$ is given by $\Lambda(y^*,\cdot)$ for some $y^*\in\mcal{Y}$. Fix a $y\in\mcal{Y}$, let $\alpha$
be the probability measure on $\mcal{X}$ with density $\Lambda(y,\cdot)$, and write $\tau_{y^*}$ for the first hitting time of $y^*$ by $Y$ when started at $y$. Then,
\eq\label{mixing_bnd}
d((\alpha P_t),\nu)\leq \pp(\tau_{y^*}\geq t), \quad t\ge 0.
\en  
\end{thm}

\begin{proof}
{Consider the diffusion $Z$ of Theorem \ref{main1} with $Z_1(0)$ distributed according to $\alpha$ and $Z_2(0)=y$. Recall that $Z_2\stackrel{d}{=}Y$ and, with a slight abuse of notation, write $\tau_{y^*}$ for the first hitting time of $y^*$ by $Z_2$. Next, fix some $0\le s\le t<\infty$. The regular conditional distribution of $Z_1(s)$ given $\tau_{y^*}=s$ is $\nu$ thanks  to condition \eqref{i4} in Definition \ref{def:intertwin} and the fact that the density of $\nu$ is $\Lambda(y^*,\cdot)$. Putting this together with the Markov property of $Z$, condition \eqref{i3} in Definition \ref{def:intertwin}, and the invariance of $\nu$ (recall that $Z_1\stackrel{d}{=}X$) one concludes further that the regular conditional distribution of $Z_1(t)$ given $\tau_{y^*}=s$ is $\nu$. Writing $\chi$ for the law of $\tau_{y^*}$ one therefore finds
\eq\label{ADtrick}
\pp(Z_1(t)\in A)\geq \int_0^t \pp(Z_1(t)\in A\,|\,\tau_{y^*}=s)\,\mathrm{d}\chi(s)
=\nu(A)\,\pp(\tau_{y^*}\leq t)=\nu(A)\,(1-\pp(\tau_{y^*}>t))
\en
for any Borel $A\subset\mcal{X}$. The theorem readily follows.}
\end{proof}

\begin{rmk}
The computation \eqref{ADtrick} in the case of finite state space Markov chains is due to Aldous and Diaconis and can be found in the proof of Proposition 3.2 (a) in \cite{AD}.
\end{rmk}

To understand the order of magnitude of the upper bound \eqref{mixing_bnd} take, for example, ${\mathcal A}^Y=\frac{1}{2}\,\partial_y^2-\kappa\,\partial_y$ for some $\kappa\ge 0$, so that $Y$ is a Brownian motion with drift $-\kappa$. Then, with $y^*=0$ and some fixed $y>0$, \eqref{mixing_bnd} leads to
\eq\label{mixing_bnd2}
d((\alpha P_t),\nu)\leq\int_t^\infty \frac{y}{\sqrt{2\pi s^3}}\exp\bigg(-\frac{(y-\kappa s)^2}{2s}\bigg)\,\mathrm{d}s
\leq e^{\kappa y}\,y\,\sqrt{\frac{2}{\pi t}}\,e^{-\kappa^2 t/2}.
\en
To get the first inequality we have simply inserted the explicit formula for the first passage time of a Brownian motion with drift (see, e.g., Section 3.5.C in \cite{KS}) into \eqref{mixing_bnd}, and the second inequality has been obtained by the following chain of elementary estimates:
\[
\begin{split}
& \int_t^\infty \frac{y}{\sqrt{2\pi s^3}}\exp\bigg(-\frac{(y-\kappa s)^2}{2s}\bigg)\,\mathrm{d}s
\leq \frac{e^{\kappa y}\,y}{\sqrt{2\pi}} \int_t^\infty s^{-3/2}\,e^{-\kappa^2 s/2}\,\mathrm{d}s \\
& = \frac{e^{\kappa y}\,y}{\sqrt{2\pi}}\bigg(2\,\sqrt{2\pi}\,\kappa\,\Phi\big(\kappa\sqrt{t}\big)-2\,\sqrt{2\pi}\,\kappa+\frac{2}{\sqrt{t}}\,e^{-\kappa^2 t/2}\bigg)
\leq e^{\kappa y}\,y\,\sqrt{\frac{2}{\pi t}}\,e^{-\kappa^2 t/2}.
\end{split}
\]
Here $\Phi$ is the standard Gaussian cumulative distribution function. We illustrate our findings on a concrete example. 

\begin{exm} 
{Let $X$ be a reflected Brownian motion on $[0,1]$, that is, a process evolving as a standard Brownian motion while in $(0,1)$ and instantaneously reflecting when it hits $0$ or $1$. It is well-known that $X$ is positive recurrent and that its invariant distribution $\nu$ is given by the uniform distribution on $[0,1]$. Consider the initial conditions $\alpha_{p,c}$, $p\in\nn$, $c\in\big[0,\frac{1}{2}\big)$ with densities $1+\frac{1}{2}\,\sin(2\pi pc)\,\cos(2\pi px)$, $p\in\nn$, $c\in\big[0,\frac{1}{2}\big)$, respectively. In order to estimate the rate of convergence to stationarity from these initial conditions, we let $Y$ be a standard Brownian motion on a circle of length $1$ parametrized by $[0,1)$ and introduce the kernels
\[
\Lambda_p(y,x):=1+\frac{1}{2}\,\sin(2\pi py)\,\cos(2\pi px),\quad p\in\nn.
\]
Clearly, $\Lambda_p(0,x)=1$ and $\Lambda_p(c,x)=1+\frac{1}{2}\,\sin(2\pi pc)\,\cos(2\pi px)$. Moreover, it is not hard to check that each kernel $\Lambda_p$ is stochastic and solves \eqref{PDE} for the generators of $X$ and $Y$, with both sides being equal to $-\pi^2p^2\sin(2\pi py)\,\cos(2\pi px)$ (note that the generator of $X$ is one half times the Neumann second derivative operator on $[0,1]$ and that the generator of $Y$ is one half times the second derivative operator on $[0,1)$). In addition, for every fixed $p\in\nn$, one can verify the other conditions in Theorem \ref{main1} by using the smoothness of $\Lambda_p$ and the boundedness of $\Lambda_p$, $\partial_y\log\Lambda_p$,  and ${\mathcal A}^Y\Lambda_p$. Consequently, we can apply Theorem \ref{conv_thm} to get the upper bound
\[
d((\alpha_{p,c}P_t),\nu)\leq \pp(\tau_0\geq t), \quad t\ge 0
\]
where $\tau_0$ is the first hitting time of $0$ by $Y$ when started at $c$. The latter can be bounded above further by the rightmost expression in \eqref{mixing_bnd2} with $\kappa=0$ and $y=c$, that is, $c\,\sqrt{\frac{2}{\pi t}}$.}  
\end{exm}
}

\appendix
\section{Some solutions of hyperbolic PDEs} 

Theorem \ref{main1} shows, in particular, that classical solutions of \eqref{PDE} (with $({\mathcal A}^X)^*$ and ${\mathcal A}^Y$ being interpreted as differential operators) give rise to intertwinings of diffusions, provided they are stochastic and have the appropriate boundary behavior. In this appendix, we have therefore collected some known explicit formulas for classical solutions of hyperbolic PDEs as in \eqref{PDE}, as well as some general existence results for such PDEs.  

\begin{exm}[Classical wave equations] 
We start with the simplest example of ${\mathcal A}^X=\partial_x^2$ on $\rr$ and ${\mathcal A}^Y=\Delta_y$ on $\rr^n$ (the case of ${\mathcal A}^X=\Delta_x$ on $\rr^m$ and ${\mathcal A}^Y=\partial_y^2$ on $\rr$ being analogous). The equation \eqref{PDE} is then the classical wave equation
\eq\label{wave}
\partial_x^2\,\Lambda=\Delta_y\,\Lambda. 
\en
When $n=1$, all classical solutions of \eqref{wave} can be written as 
\[
\phi(y-x)+\psi(y+x)
\] 
thanks to the well-known d'Alembert's formula. When $n\geq2$, the classical solutions of \eqref{wave} are given by the following formulas (see, e.g., Section 2.4 in \cite{Ev}): 
\[
\partial_x\bigg(\frac{1}{x}\,\partial_x\bigg)^{\frac{n-3}{2}}
\bigg(\frac{1}{x}\int_{\partial B(y,x)} \phi(\tilde{y})\,\mathrm{d}\theta(\tilde{y})\bigg)
+\bigg(\frac{1}{x}\,\partial_x\bigg)^{\frac{n-3}{2}}
\bigg(\frac{1}{x}\int_{\partial B(y,x)} \psi(\tilde{y})\,\mathrm{d}\theta(\tilde{y})\bigg)
\]
if $n$ is odd, and
\[
\begin{split}
\partial_x\bigg(\frac{1}{x}\,\partial_x\bigg)^{\frac{n-2}{2}}
\bigg(\int_{B(y,x)} \frac{\phi(\tilde{y})}{(x^2-|\tilde{y}-y|^2)^{1/2}}\,\mathrm{d}\tilde{y}\bigg) 
+\bigg(\frac{1}{x}\,\partial_x\bigg)^{\frac{n-2}{2}}
\bigg(\int_{B(y,x)} \frac{\psi(\tilde{y})}{(x^2-|\tilde{y}-y|^2)^{1/2}}\,\mathrm{d}\tilde{y}\bigg)
\end{split}
\]
if $n$ is even. Here $B(y,x)$ is the ball of radius $x$ around $y$, $\partial B(y,x)$ is its boundary, and $\theta$ is the Lebesgue measure on $\partial B(y,x)$. 
\end{exm}    

\begin{exm}[Divergence form operators] 
{Next, we consider the situation where ${\mathcal A}^X=\frac{1}{v(x)}\,\partial_x\,v(x)\,\partial_x$ for some $v>0$ on an interval in $\rr$ and ${\mathcal A}^Y = \partial_y^2$ on $\rr$. Note that, if $v$ is continuously differentiable, the diffusion $X$ corresponding to ${\mathcal A}^X$ is well-defined provided it does not explode, and in the case of non-explosion it is reversible with respect to the measure $v(x)\,\mathrm{d}x$. In this situation, classical solutions of \eqref{PDE} can be obtained by a procedure described in \cite{Ca} and the references therein. Consider eigenfunctions 
\[
{\mathcal A}^X\,\phi_\lambda=\lambda\,\phi_\lambda,\quad
{\mathcal A}^Y\,\psi_\lambda=\lambda\,\psi_\lambda 
\]
where $\lambda$ varies over the set of eigenvalues of ${\mathcal A}^X$. Then, superpositions of the functions $v(x)\,\phi_\lambda(x)\,\psi_\lambda(y)$ for varying values of $\lambda$ are classical solutions of \eqref{PDE}. One case, in which this procedure leads to explicit solutions, is that of $v(x)=x^{2\nu+1}$ and ${\mathcal A}^X=\partial_{xx}+\frac{2\nu+1}{x}\,\partial_x$ on $(0,\infty)$ where $\nu\ge0$. In this case, one can let $\lambda$ vary in $(-\infty,0]$ and choose each $\phi_\lambda$ as a linear combination of $x^{-\nu}\,J_\nu\big(-\sqrt{-\lambda}\,x\big)$ and $x^{-\nu}\,Y_\nu\big(-\sqrt{-\lambda}\,x\big)$ and each $\psi_\lambda$ as a linear combination of $\sin\big(\sqrt{-\lambda}\,y\big)$ and $\cos\big(\sqrt{-\lambda}\,y)$ where $J_\nu$ and $Y_\nu$ are Bessel functions of the first and second kind, respectively. Another formula for classical solutions of \eqref{PDE} in the same case, which is more amenable to the selection of positive solutions, has been given earlier in \cite{De} and reads
\[
\int_0^\pi \phi\big(\sqrt{x^2+y^2-2xy\cos\alpha}\big)(\sin\alpha)^{2\nu}\,\mathrm{d}\alpha.
\]
Note that the latter function is positive as soon as $\phi$ is positive.} 

\end{exm} 

\begin{exm}[Euler-Poisson-Darboux equation] 
Now, consider the case ${\mathcal A}^X=\Delta_x$, ${\mathcal A}^Y=\partial_y^2+\frac{2\nu+1}{y}\,\partial_y$. In this case, the equation \eqref{PDE} is known as the Euler-Poisson-Darboux (EPD) equation. While particular solutions of this equation go back to Euler and Poisson, a full understanding of the Cauchy problem for the EPD equation with initial conditions $\Lambda(0,x)=f(x)$, $(\partial_y\Lambda)(0,x)=0$ has been achieved more recently in \cite{As}, \cite{We}, \cite{DW}, and \cite{We2}. The following summary of their results is taken from the introduction of \cite{Bl}. When $2\nu+1=m-1$, the solution reads 
\eq\label{EPD1}
\frac{1}{c_{m-1}}\int_{\partial B(0,1)} f(x+y\tilde{x})\,\mathrm{d}\theta(\tilde{x})
\en  
where $c_{m-1}$ is the volume of the $(m-1)$-dimensional unit sphere $\partial B(0,1)$ and $\theta$ is the Lebesgue measure on the latter. When $2\nu+1>m-1$, the solution is
\eq\label{EPD2}
\frac{c_{2\nu+2-m}}{c_{2\nu+2}}\int_{B(0,1)} f(x+y\tilde{x})(1-|\tilde{x}|^2)^{\nu-m/2}\,\mathrm{d}\tilde{x}
\en
where $B(0,1)$ is the $m$-dimensional unit ball. Finally, when $0<2\nu+1<m-1$, the solution is given by
\eq\label{EPD3}
y^{-2\nu}\bigg(\frac{1}{y}\,\partial_y\bigg)^q y^{2\nu+2q}\,\tilde{\Lambda}(y,x)
\en 
where $\tilde{\Lambda}(y,x)$ is the solution of the EPD equation with $2\nu+1$ replaced by $2\nu+2q+1$, {$f$ replaced by $\frac{f}{(2\nu+2)(2\nu+4)\cdots(2\nu+2q)}$, and $q\in\nn$ such that $2\nu+2q+1\ge m-1$}. 
\end{exm}

We supplement the explicit solutions above by some general existence results for equations of the type \eqref{PDE} taken from Section 7.2 in \cite{Ev}.  

\begin{prop}
Suppose the coefficients of ${\mathcal A}^X$ and ${\mathcal A}^Y$ are smooth. Then, in each of the following cases classical solutions of the equation \eqref{PDE} exist.  
\begin{enumerate}[(a)]
\item $m=1$, ${\mathcal A}^X=\partial_x^2$, $n$ is arbitrary, and ${\mathcal A}^Y$ is uniformly elliptic.
\item $m$ is arbitrary, ${\mathcal A}^X$ is uniformly elliptic, $n=1$, and ${\mathcal A}^Y=\partial_y^2$. 
\end{enumerate}
\end{prop} 

To the best of our knowledge, conditions for positivity of these solutions have not been studied in this generality.

\section{A result about $C^1$ functions of Semimartingales} 
Since $\mcal{Y}$ is a locally compact subset of $\mathbb{R}^n$ it can be expressed as $\mcal{Y} = O \cap \overline{\mcal{Y}}$ where $O$ is open. When we write $C^m(\mcal{Y})$, we mean restrictions of $C^m(O)$ functions to $\mcal{Y}$ for some $O$ such that $\mcal{Y} = O \cap \overline{\mcal{Y}}$ holds.
\begin{lemma}\label{lowreggito}
Let $Y(t) = Y(0) + M(t) + A(t)$ be a continuous semimartingale taking values in a locally compact state space $\mcal{Y} \subseteq \mathbb{R}^n$ with (vector) local martingale part $M$ and bounded variation part $A$. Let $f \in C^1(\mcal{Y})$ be a function such that $f(Y)$ is a semimartingale with local martingale part $N$. Then, we have the equality
\begin{equation}\label{lowregito}
N(t) = \sum_{j=1}^n \int_0^t \partial_j f(Y(s))\,\mathrm{d}M^j(s).
\end{equation}
\end{lemma}
\begin{proof}
It is easily seen (e.g., \cite[Proposition 3.2.24]{KS}) that the right-hand side of (\ref{lowregito}) is the unique continuous local martingale $R$ such that the following equality holds for all continuous local martingales $P$:
\begin{equation}\label{covidentity}
\langle R, P \rangle_t =\sum_{j=1}^n \int_0^t \partial_j f(Y(s)) \,\mathrm{d} \langle M^j, P \rangle(t).
\end{equation}
Therefore, it suffices to show that $N$ has this property. Fix $t >0$ and consider a mesh $\mathbf{t} = (t_0,\ldots,t_T)$  with $0=t_0< t_1 < \ldots < t_T=t$ with maximum mesh size $\Delta : =\max_{k=0,\ldots,T-1}(t_{k+1} - t_{k})$. Then, by standard arguments (see, e.g., \cite[Proposition IV.1.18]{RY}),
\[
\lim_{\Delta \downarrow 0} \sum_{k=0}^{T-1} \big(f(Y(t_{k+1})) - f(Y(t_k)) \big)\big(P(t_{k+1})-P(t_k)\big) = \langle N,P \rangle_t,
\]
where the limit is understood as a limit in probability. We now proceed to calculate the limit explicitly.

Fix an open set $O$ such that $\mcal{Y} = O \cap \overline{\mcal{Y}}$ and such that $f \in C^1(O)$. Define a sequence of compact subsets of $O$ as $K_p = \{ y \in O : |y| \leq p, \text{dist}(y, \partial O ) \geq \frac{1}{p} \}$, $ p\in \mathbb{N}$. Also, define the events 
\[
E(\mathbf{t},p,\delta) :=\Big\{Y([0,t]) \subseteq K_p, \max_{k=0,...,T-1}|Y(t_{k+1})-Y(t_k)| < \delta\Big\}.
\]
There exists a finite set of points $y^1,\ldots,y^{\kappa(p)}$ such that $\{ B(y^l,\frac{1}{4p}) \}_{l=1}^{\kappa(p)}$ is an open cover of $K_p$. This open cover admits a Lebesgue number $\lambda_p$. Note that on the event $E(\mathbf{t},p,\delta)$ with $\delta < \frac{\lambda_p}{2}$, which we assume throughout, we have that
\[
\big\{ \lambda Y(t_{k+1}) + (1-\lambda) Y(t_k) :\, k=0,\ldots,T-1,\, \lambda \in [0,1]\big\} \subseteq \Big\{y \in \mathbb{R}^n : |y| \leq p, \text{dist}(y, \partial O) \geq \frac{1}{2p} \Big\}.
\]
Denote the set on the right-hand side above as $\tilde{K}_p$. On the event $E(\mathbf{t},p,\delta)$, by the Mean Value Theorem, there exists a random variable $Z_k \in \tilde{K}_p$ which is a (random) convex combination of  $Y({t_{k+1}} )$ and $ Y(t_k) $ such that $f(Y(t_{k+1})) - f(Y(t_k)) = \nabla f(Z_k)' (Y(t_{k+1}) - Y(t_k))$.
For any continuous process $X$, write $\delta_k X = X(t_{k+1}) - X(t_k)$. Then, we first note that on the event $E(\mathbf{t},p,\delta)$,
\[
\sum_{k=0}^{T-1} \delta_k f(Y) \delta_k P= \sum_{k=0}^{T-1} \nabla f(Z_k)'\delta_k M\, \delta_k P + \sum_{k=0}^{T-1}\nabla f(Z_k)' \delta_k A \,\delta_k P.
\]
Using the facts that $A$ is continuous with finite variation, $P$ is continuous, and $\nabla f$ is bounded on compact sets, arguments as in \cite[Proposition IV.1.18]{RY} show that on the event $E(\mathbf{t},p,\delta)$, the second term above converges to $0$ in probability as $\Delta \downarrow 0$. Also, note that 
\[
\lim_{p \rightarrow \infty} \limsup_{\delta \downarrow 0} \limsup_{\Delta\downarrow 0} \mathbb{P}(E(\mathbf{t},p,\delta)^c) = 0.
\]
Next, since $\nabla f$ is uniformly bounded and uniformly continuous on $\tilde{K}_p$, because 
\[ 
\sum_{k=0}^{T-1}(\delta_k M^j)^2\rightarrow \langle M^j \rangle_t\text{,}\hspace{4pt} \sum_{k=0}^{T-1}(\delta_k P)^2\rightarrow \langle P \rangle_t\hspace{6pt}\text{in probability,}
\]
and by the Cauchy-Schwarz inequality, we know that 
\[
\sum_{k=0}^{T-1}\big(\nabla f(Z_k) - \nabla f(Y(t_k))\big)'\delta_k M \,\delta_k P
\]
converges in probability to $0$ on the event $E(\mathbf{t},p,\delta)$. To finish the proof, it suffices to show that for all $j=1,\ldots,n$, the following converges to $0$ in probability:
\begin{equation}\label{laststep}
\sum_{k=0}^{T-1} \partial_j f(Y(t_k))\big(\delta_k M^j\,\delta_k P - \delta_k \langle M^j, P \rangle\big).
\end{equation}
Since nothing in (\ref{laststep}) depends on the event $E(\mathbf{t},p,\delta)$, we now drop the requirement that we are on said event. By localization, we may assume $Y$, $M$, and $P$ take values in a compact set and that the quadratic variations of $M^j$ and $P$ are uniformly bounded. Under these assumptions, we claim that the term (\ref{laststep}) converges to $0$ in $L^2$. 

To see this, note that after squaring the term (\ref{laststep}), the cross terms resulting from the sum in $k$ vanish in expectation. Therefore, it suffices to bound
\begin{equation}\label{lastest}
\sum_{k=0}^{T-1}\partial_j f(Y(t_k))^2\big(\delta_k M^j\,\delta_k P - \delta_k \langle M^j, P \rangle \big)^2.
\end{equation}
We may bound the partial derivatives of $f$ by a constant. Define the term 
\[
D(t,\Delta) = \max_{\substack{j=1,...,n \\ s,\tilde{s} \in [0,t],|s-\tilde{s}|\leq\Delta}} (M^j_s - M^j_{\tilde{s}})^2 + \max_{ s,\tilde{s} \in [0,t],|s-\tilde{s}|\leq\Delta}(P_s - P_{\tilde{s}})^2.
\]
Now, by the Itô product rule, the inequality $(a+b)^2 \leq 2a^2 + 2b^2$, and the Itô isometry, we have that 
\begin{equation}\notag
\begin{split}
\mathbb{E}\Big[\big(\delta_k M^j \delta_k P - \delta_k \langle M^j , P \rangle\big)^2\Big] &\leq 2\mathbb{E}\bigg[ \int_{t_k}^{t_{k+1}}(M_s^j - M^j_{t_k})^2\, \mathrm{d}\langle P \rangle_s +  \int_{t_k}^{t_{k+1}}(P_s - P_{t_k})^2 \,\mathrm{d}\langle M^j \rangle_s\bigg]\\
 &\leq 2\mathbb{E} \Big[D(t,\Delta) \big(\langle P \rangle_{t_{k+1}} - \langle P\rangle_{t_k}+\langle M^j \rangle_{t_{k+1}} - \langle M^j\rangle_{t_k} \big)\Big].
 \end{split}
\end{equation}
Therefore, in expectation, the term (\ref{lastest}) can be upper bounded by 
\[
C\mathbb{E}[D(t,\Delta)]
\]
which converges to $0$ by the Bounded Convergence Theorem. This concludes the proof of the lemma.
\end{proof}

\bibliographystyle{amsalpha}
\bibliography{intertwining}

\end{document}